\documentclass[preprint,11pt]{elsarticle}
\usepackage[T1]{fontenc}
\usepackage[margin=1in]{geometry}
\usepackage{booktabs,array,tabularx}
\usepackage{listings}
\usepackage{xcolor}
\usepackage{seqsplit}
\usepackage{amsmath,amssymb,amsthm,mathtools,mathrsfs}
\usepackage{hyperref}
\hypersetup{colorlinks=true,linkcolor=blue!55!black,urlcolor=blue!55!black}
\allowdisplaybreaks[4]

\newtheorem{theorem}{Theorem}[section]
\newtheorem{lemma}[theorem]{Lemma}

\newtheorem{proposition}[theorem]{Proposition}

\theoremstyle{remark}

\theoremstyle{definition}

\numberwithin{equation}{section}

\DeclareMathOperator{\Log}{Log}
\DeclareMathOperator{\Arg}{Arg}
\begin{document}
\begin{frontmatter}
\title{Infinite log-concavity of the Taylor coefficients of the Riemann xi-function}

\author[]{Yanxin Liu}
\ead{liu-yanxin@outlook.com}

\author[]{Jianxi Mao\corref{cor1}}
\ead{maojx@dlut.edu.cn}
\cortext[cor1]{Corresponding author}

\address{School of Mathematical Sciences, Dalian University of Technology, Dalian 116024,\\ P. R. China}

\begin{abstract}
Let $\xi(s)$ denote the Riemann xi-function, and write
\begin{equation*}
F(x)=\xi\!\left(\frac12+\sqrt{x}\right)
=\sum_{n=0}^{\infty}\lambda_n x^n,
\qquad
\lambda_n=\frac{\xi^{(2n)}(1/2)}{(2n)!}.
\end{equation*}
The Riemann hypothesis is equivalent to $F(x)$ belonging to the Laguerre--P\'olya class.
Br\"and\'en [J. Reine Angew. Math., 2011] proved that if
an entire function in the Laguerre--P\'olya class has nonnegative Taylor coefficients, then its coefficient sequence is infinitely log-concave.
Consequently, the Riemann hypothesis implies the
infinite log-concavity of $(\lambda_n)_{n\ge0}$.
In this paper, we prove that the sequence $(\lambda_n)_{n\ge0}$ is strictly
infinitely log-concave.
This resolves a conjecture of Zhu [Math. Z., 2023].
The proof combines explicit complex-analytic estimates
for the iterated logarithmic ratios, rigorous interval arithmetic
for a finite range of indices, and a global closure argument.
\end{abstract}
\end{frontmatter}

\tableofcontents

%=================================================
\section{Introduction}
%=================================================
Let
\begin{equation*}
\xi(s)=\frac12 s(s-1)\pi^{-s/2}\Gamma(s/2)\zeta(s)
\end{equation*}
denote the Riemann xi-function, where $\zeta(s)$ is the Riemann zeta function.
Its functional equation $\xi(s)=\xi(1-s)$ implies that
its Taylor expansion about $s=1/2$ contains only even powers.
Write
\begin{equation}\label{eq:F(x)}
F(x)=\xi\!\left(\frac12+\sqrt{x}\right)
=\sum_{n=0}^{\infty}\lambda_n x^n
=\sum_{n=0}^{\infty}\gamma_n\frac{x^n}{n!},
\qquad
\lambda_n=\frac{\xi^{(2n)}(1/2)}{(2n)!}.
\end{equation}
Thus $(\lambda_n)_{n\ge0}$ is the sequence of even Taylor
coefficients of $\xi$ at its central point.

A real entire function $\psi(x)$
is said to be in the \emph{Laguerre-P\'olya class}, denoted by
$\psi(x)\in\mathcal{LP}$, if it can be represented in the form
\begin{equation}\label{eq:product}
\psi(x)=cx^m e^{-\alpha x^2+\beta x}
\prod_{n=1}^{\infty}(1+x/x_n)\,e^{-x/x_n},
\end{equation}
where $c,\beta,x_n$ are real numbers, $\alpha\geq0$,
$m$ is a nonnegative integer and $\sum x_n^{-2}<\infty$.
We refer the reader to~\cite{Rs02} for further background on the theory of the $\mathcal{LP}$ class.

P\'olya  proved that the Riemann hypothesis holds if and only if
the function $F(x)$ belongs to the $\mathcal{LP}$ class with positive coefficients,
i.e., $F(x)$ has only negative real zeros~\cite{SP14}.
The Riemann hypothesis has been extensively studied from the perspective that the function  $F(x)$ belongs to the $\mathcal{LP}$ class; see, for example, \cite{CC89, CVV90, Osu21, WY24}.
The famous P\'olya-Schur criterion~\cite{SP14} states that $F(x)$ belongs to the $\mathcal{LP}$ class if and only if the Jensen polynomials associated with $\{\gamma_n\}$ have only real zeros.
In 2019, Griffin, Ono, Rolen and Zagier~\cite{GORZ19} proved that the Jensen polynomials associated with $\{\gamma_n\}$ have only real zeros for each fixed degree and all sufficiently large shifts.
In 2022, Griffin, Ono, Rolen, Thorner, Tripp and Wagner~\cite{GORTTW22}
made this approach effective and obtained an effective bound for the threshold.

Another closely related direction concerns log-concavity for the Taylor coefficients of the Riemann xi-function. 
Let $\alpha=(a_n)_{n\ge0}$ be a sequence of real numbers with $a_{-1}=0$. 
We say that the sequence $\alpha$ is~\emph{log-concave} (or satisfies the \emph{Tur\'an inequalities}) if $a_n^2\ge a_{n-1}a_{n+1}$ for all $n\ge0$.
It was P\'olya~\cite{Pol27} who first asked whether the sequence $\{\gamma_n\}$ is log-concave. Note that the log-concavity of $\{\gamma_n\}$ implies that of $\{\lambda_n\}$.
Csordas, Norfolk, and Varga~\cite{CNV86} settled this problem.
Dimitrov and Lucas~\cite{DL11} subsequently established higher-order Tur\'an inequalities for the Riemann xi-function, which, in particular, yield the hyperbolicity of the associated Jensen polynomials of degree three.

For a sequence $\alpha=(a_n)_{n\ge0}$ of real numbers,  define the sequence $\mathcal L\alpha$ by
\begin{equation*}
(\mathcal L\alpha)_0=a_0^2,\quad
(\mathcal L\alpha)_n=a_n^2-a_{n-1}a_{n+1}\quad \text{for }n\ge1.
\end{equation*}
Clearly, the sequence $\alpha$ is log-concave if and only if the sequence $\mathcal L\alpha$ is a nonnegative sequence.
Define $\mathcal L^j\alpha=\mathcal L(\mathcal L^{j-1}\alpha)$.
We say the sequence $\alpha$ is \emph{$r$-log-concave} if $\mathcal L^j \alpha$ is nonnegative for all $1\le j\le r$, and $\alpha$ is \emph{infinitely log-concave} if $\mathcal L^j \alpha$ is nonnegative for every integer $j\ge1$.

Boros and Moll~\cite{BM04} introduced the notions of $r$-log-concavity and infinite log-concavity. 
They conjectured that the sequence of binomial coefficients,
$\left\{\binom{n}{k}\right\}_{k}$, is infinitely log-concave for each $n\in\mathbb{N}$.
If the polynomial $\sum_{n=0}^{d}a_n z^n$ has only real
and non-positive zeros, then Newton's inequality implies that the sequence
$\{a_n\}$ is log-concave.
Motivated by this fact and Boros and Moll's conjecture on binomial coefficients,
McNamara and Sagan~\cite{MS10} conjectured that the coefficient sequence of a polynomial with only negative real zeros is infinitely log-concave.
Independently, Fisk and Stanley formulated the same conjecture.
Note that a polynomial with only nonpositive real zeros belongs to the Laguerre-P\'olya class $\mathcal{LP}$.
Br\"and\'en~\cite{Bra11} proved this conjecture by showing that if an entire function of nonnegative Taylor coefficients belongs to the Laguerre-P\'olya class, 
then its coefficient sequence is infinitely log-concave.
Consequently, if the Riemann hypothesis holds, then the sequence $(\lambda_n)_{n\ge0}$ in~\eqref{eq:F(x)} is infinitely log-concave.
Recently, Zhu~\cite{Zhu23} proved that
$(\lambda_n)_{n\ge0}$ is $6$-log-concave and conjectured that it is infinitely log-concave.
In this paper, we give an affirmative answer to this conjecture, without assuming the Riemann hypothesis.
\begin{theorem}\label{thm:main}
Let
\begin{equation*}
F(x)=\xi(1/2+\sqrt{x})
=\sum_{n=0}^{\infty}\lambda_nx^n.
\end{equation*}
Then \begin{equation*}
(\mathcal L^m\lambda)_n>0
\quad
\text{for all }m\ge 0,n\ge0.
\end{equation*}
In particular, the sequence $(\lambda_n)_{n\ge0}$ is infinitely log-concave.
\end{theorem}

To prove Theorem~\ref{thm:main}, we use analytic estimates for infinitely many cases, and then we verify the finitely many remaining cases directly.
The main difficulty is to control two indices simultaneously:
the coefficient index and the iteration order.
Motivated by McNamara and Sagan~\cite{MS10},
we overcome this difficulty by showing that, for each coefficient
index $n\ge1$, there exists a finite iteration order $\tau(n)$
such that
\begin{equation*}
\bigl((\mathcal L^{\tau(n)}\lambda)_n\bigr)^2
\ge \frac{3+\sqrt5}{2}\,
(\mathcal L^{\tau(n)}\lambda)_{n-1}
(\mathcal L^{\tau(n)}\lambda)_{n+1}.
\end{equation*}
We obtain the required bounds by estimating the
logarithmic ratios step by step on successively smaller disks.
Combining these estimates with finite verification and a global closure argument, we prove $(\lambda_n)_{n\ge0}$ is
infinitely log-concave.

The paper is organized as follows. 
In Section~\ref{sec:Ratio_moment}, we introduce the logarithmic ratios $g_{m,n}$ of a sequence $\alpha$ associated with $\mathcal L$. 
A positive sequence $\alpha$ is strictly infinitely log-concave if and only if $g_{m,n}>0$ for all $m\ge 0,n\ge 1$.
We present the recurrence of the logarithmic ratios $g_{m,n}$.
Then we introduce a positive moment representation for the coefficients $(\lambda_n)_{n\ge0}$. 
In Section~\ref{sec:analytic}, we prove that  $g_{m,n}\ge 0$
for $n\ge 12000$ and $m\le \lceil\log_2 n\rceil$.
Using a saddle point estimate,
we provide explicit estimates for the initial logarithmic ratios $g_{0,n}$ and bound their iterates $g_{m,n}$ on successively smaller disks. 
In Section~\ref{sec:finite}, we prove  that  $g_{m,n}\ge 0$ for $1\le n < 12000$ and $m\le 14$ by a computer-assisted verification.
In Section~\ref{sec:closure}, we prove Theorem~\ref{thm:main}. 
We use the results that for any fixed $n\ge 12000$, $g_{\lceil\log_2 n\rceil,n}>1$,
and  for any fixed $n\le 12000$, $g_{14,n}\ge 1$,
which have been proved in Sections~\ref{sec:analytic} and~\ref{sec:finite}, respectively.
These results imply the iterated log-concavity for $n\ge 12000, m\ge \lceil\log_2 n\rceil$,
and  for $1\le n\le 12000, m\ge 14$.
Together with the results in Sections~\ref{sec:analytic} and~\ref{sec:finite},
we complete the proof.
Finally, \ref{app:source} provides the implementation details and source code to ensure complete reproducibility of the computer-assisted verification in Section~\ref{sec:finite}.

%=================================================
\section{Preliminaries}
\label{sec:Ratio_moment}
%=================================================
In this section, we introduce logarithmic ratios $g_{m,n}$ associated with $\mathcal L$, and establish the basic tools needed for the later analysis.
For a function $f(z)$, the second difference is defined by
$\Delta^2 f(z):=f(z-1)-2f(z)+f(z+1)$.

For $m\ge0$,
let $\alpha=(a_n)_{n\ge0}$ be a sequence and 
$A_{m,n}=(\mathcal L^m\alpha)_n$, where $A_{0,n}=a_n$.
By definition,
\begin{equation}\label{eq:receAmk}
A_{m+1,n}=A_{m,n}^2-A_{m,n-1}A_{m,n+1}
\end{equation}
with $A_{m,0}=(a_0)^{2^m}$. 
For $n\ge1$, whenever
$A_{m,n-1},\, A_{m,n},\, A_{m,n+1}>0$, we set
\begin{equation}\label{eq:rmk&gmk}
r_{m,n}
:=
\frac{A_{m,n}^2}
{A_{m,n-1}A_{m,n+1}},
\quad
g_{m,n}:=\log r_{m,n}.
\end{equation}
A direct computation gives
\begin{equation}\label{A-g}
A_{m+1,n}
=
A_{m,n}^2-A_{m,n-1}A_{m,n+1}
=
A_{m,n}^2\bigl(1-e^{-g_{m,n}}\bigr).
\end{equation}
Now we give the recurrence for $g_{m,n}$.

\begin{lemma}\label{lem:ratio_identity_and_persistence}
Let $m\ge 0$ and $n\ge2$. 
In the notation above, suppose that
$g_{m,n-1},g_{m,n},g_{m,n+1}$ are well defined and strictly positive. 
Then
$g_{m+1,n}$ is well defined and
\begin{equation}\label{eq:lemma 2.1-1}
g_{m+1,n}
=
2g_{m,n}-\Delta^2 h(g_{m,n})=2g_{m,n}-h(g_{m,n-1})+2h(g_{m,n})-h(g_{m,n+1}),
\end{equation}
where
$h(t)=\log(1-e^{-t})$.
\end{lemma}

\begin{proof}
Recall that by~\eqref{A-g}
\begin{equation*}
A_{m+1,n}
=
A_{m,n}^2-A_{m,n-1}A_{m,n+1}
=
A_{m,n}^2\bigl(1-e^{-g_{m,n}}\bigr).
\end{equation*}
Since all the entries involved are positive, we can obtain that
\begin{equation}\label{log-h}
\log A_{m+1,n}
=
2\log A_{m,n}+h(g_{m,n}).
\end{equation}
On the other hand,
\begin{equation*}
g_{m,n}
=
\log\frac{A_{m,n}^2}{A_{m,n-1}A_{m,n+1}}
=
2\log A_{m,n}
-\log A_{m,n-1}
-\log A_{m,n+1}.
\end{equation*}
Hence $g_{m,n}=-\Delta^2\log A_{m,n}$.
Then applying $-\Delta^2$ to both sides of~\eqref{log-h} gives
\begin{equation*}
g_{m+1,n}=
2g_{m,n}-\Delta^2 h(g_{m,n}).
\end{equation*}
This proves~\eqref{eq:lemma 2.1-1}, and completes the proof.
\end{proof}

The following lemma connects the auxiliary functions $g_m(x)$ with $g_{m,n}$ defined in~\eqref{eq:rmk&gmk}. It shows that if these functions $g_m(x)$ satisfy some required conditions, then their values at integer points coincide with the discrete logarithmic ratios $g_{m,n}$.

\begin{lemma}\label{lem:local_ratio_identification}
Let $\alpha=(a_k)_{k\ge0}$ be a strictly positive sequence,
and let $A_{m,k}$, $r_{m,k}$, and $g_{m,k}$ be defined
as in~\eqref{eq:receAmk} and~\eqref{eq:rmk&gmk}.

Fix integers $n\ge1$ and $M\ge0$. For $0\le m\le M,$ Let
\begin{equation*}
D_m=\{z\in\mathbb C:|z-n|<R_m\},
\end{equation*}
where
$
0<R_m\le n
$
and
$
R_m-R_{m+1}>1.
$
Suppose that there exist functions
$g_m:D_m\to\mathbb C$, $0\le m\le M$,
satisfying the following conditions:
\begin{itemize}
\item[\textup{(i)}]
For every $0\le m\le M$, the function $g_m$ is
real-valued and strictly positive on $D_m\cap\mathbb R$.

\item[\textup{(ii)}]
For every $0\le m<M$ and $x\in D_{m+1}\cap\mathbb R$,
\begin{equation}\label{eq:local-g-recurrence}
g_{m+1}(x)
=
2g_m(x)
-h(g_m(x-1))
+2h(g_m(x))
-h(g_m(x+1)),
\end{equation}
where $h(t)=\log(1-e^{-t})$ for $t>0$.

\item[\textup{(iii)}]
For every $k\in D_0\cap\mathbb Z$,
\begin{equation}\label{eq:initial-identification}
g_0(k)
=
g_{0,k}
=
\log r_{0,k}
=
\log\frac{a_k^2}{a_{k-1}a_{k+1}}.
\end{equation}
\end{itemize}
Set $D_{-1}=D_0$. Then, for every $0\le m\le M$,
\begin{equation}\label{eq:local-positive}
A_{m,k}>0
\qquad
\text{for all }k\in D_{m-1}\cap\mathbb Z.
\end{equation}
Moreover, for every $k\in D_m\cap\mathbb Z$,  we have  $A_{m,k-1},A_{m,k},A_{m,k+1}>0$,
and
\begin{equation}\label{eq:local-identify}
g_m(k)=g_{m,k}=\log r_{m,k}.
\end{equation}
\end{lemma}

\begin{proof}
We prove~\eqref{eq:local-positive} and~\eqref{eq:local-identify} simultaneously by induction on $m$.
For $m=0$, it is trivial that
$A_{0,k}=a_k>0$ for every $k\ge0$, and
\eqref{eq:local-identify} follows from
\eqref{eq:initial-identification}.
Assume that all conclusions hold for $0\le m<M$.

We first prove~\eqref{eq:local-positive} for $m+1$.
Let $k\in D_m\cap\mathbb Z$.
By the induction hypothesis, the three entries
$A_{m,k-1}$, $A_{m,k}$, and $A_{m,k+1}$ are strictly
positive, and
\begin{equation*}
g_m(k)=\log r_{m,k}
=
\log\frac{A_{m,k}^2}{A_{m,k-1}A_{m,k+1}}.
\end{equation*}
Consequently,
\begin{equation*}
\begin{aligned}
A_{m+1,k}
=
A_{m,k}^2-A_{m,k-1}A_{m,k+1}=
A_{m,k}^2\bigl(1-e^{-g_m(k)}\bigr)
>0,
\end{aligned}
\end{equation*}
where the last inequality follows from condition~\textup{(i)}.
Hence
\begin{equation*}
A_{m+1,k}>0
\qquad
\text{for every }k\in D_m\cap\mathbb Z.
\end{equation*}
This proves~\eqref{eq:local-positive}.

Let $k\in D_{m+1}\cap\mathbb Z$.
Since $R_m-R_{m+1}>1$, we have
\begin{equation*}
|(k\pm1)-n|
\le |k-n|+1
<R_{m+1}+1
<R_m.
\end{equation*}
Therefore
$\
k-1,\ k,\ k+1\in D_m\cap\mathbb Z.
$
Thus, by~\eqref{eq:local-positive}
\begin{equation*}
A_{m+1,k-1}>0,\qquad
A_{m+1,k}>0,\qquad
A_{m+1,k+1}>0.
\end{equation*}
In particular, $r_{m+1,k}$ is well defined.

For each $j\in\{k-1,k,k+1\}$, the induction hypothesis gives
$g_m(j)=\log r_{m,j}$, and hence
\begin{equation*}
A_{m+1,j}
=
A_{m,j}^2\bigl(1-e^{-g_m(j)}\bigr).
\end{equation*}
Substituting these identities into the definition of
$r_{m+1,k}$ yields
\begin{equation*}
\begin{aligned}
r_{m+1,k}
&=
\frac{A_{m+1,k}^2}
{A_{m+1,k-1}A_{m+1,k+1}}\\
&=
\frac{
A_{m,k}^4\bigl(1-e^{-g_m(k)}\bigr)^2
}{
A_{m,k-1}^2A_{m,k+1}^2
\bigl(1-e^{-g_m(k-1)}\bigr)
\bigl(1-e^{-g_m(k+1)}\bigr)
}\\
&=
r_{m,k}^2
\frac{
\bigl(1-e^{-g_m(k)}\bigr)^2
}{
\bigl(1-e^{-g_m(k-1)}\bigr)
\bigl(1-e^{-g_m(k+1)}\bigr)
}.
\end{aligned}
\end{equation*}
Thus, we obtain
\begin{equation*}
\begin{aligned}
g_{m+1}(k)=\log r_{m+1,k}
=
2g_m(k)
-h(g_m(k-1))
+2h(g_m(k))
-h(g_m(k+1)).
\end{aligned}
\end{equation*}
This proves~\eqref{eq:local-identify} and completes the proof.
\end{proof}

Recall from~\eqref{eq:F(x)} that $\lambda_n = \xi^{(2n)}(1/2)/(2n)!$. 
These coefficients admit a classical representation as positive moments; see~\cite[Eq.~(7.1)]{Osu21} and~\cite[Eq.~(5.1)]{Rom21}.
This integral formula naturally provides an analytic interpolation of the sequence $(\lambda_n)_{n\ge0}$. 
Its holomorphy follows from standard results on integrals depending holomorphically on a parameter, while its positivity on the real axis follows from the kernel (see~\cite{Mic26}).
We summarize these properties in the following lemma.
For the sake of completeness, we present a proof here.

\begin{lemma}\label{lem:theta_moments_and_interpolation}
For $\Re z>-1/2$, let
\begin{equation}\label{eq:H-Lambda}
H(z)
=
2^{1-2z}
\int_0^\infty
u^{2z}e^{u/4}\omega(e^u)\,du,
\quad\textrm{and}\quad
\Lambda(z)
=
\frac{H(z)}{\Gamma(2z+1)},
\end{equation}
where $\omega(t)=\sum_{j\ge1}(2\pi^2j^4t^2-3\pi j^2t)e^{-\pi j^2t}$.
Then 
\begin{itemize}
\item[\rm(i)] for any integer $n>0$, we have 
\begin{equation}\label{eq:lambda-moment}
\lambda_n=\Lambda(n)
=
\frac1{(2n)!}
\int_0^\infty 4e^{t/2}
\sum_{j\ge1}
(2\pi^2 j^4e^{4t}-3\pi j^2e^{2t})e^{-\pi j^2e^{2t}}t^{2n}\,dt;
\end{equation}
\item[\rm(ii)] $H(z)$ and $\Lambda(z)$ are holomorphic on  $\{z\in\mathbb C:\Re z>-1/2\};$
\item[\rm(iii)] $\Lambda(x)>0$ for real $x>-1/2$.
\end{itemize}
\end{lemma}

\begin{proof}
We use the following classical integral representation (see~\cite[Eq.~(7.1)]{Osu21})
\begin{equation}\label{eq:OSullivan}
\xi^{(2n)}\left(\frac12\right)
=
2^{1-2n}
\int_1^\infty
(\log s)^{2n}\omega(s)s^{-3/4}\,ds,
\quad n\ge 0.
\end{equation}
Let $s=e^u$ in \eqref{eq:OSullivan}.
Then we have
\begin{equation}\label{eq:Rie_xi_form}
\xi^{(2n)}\left(\frac12\right)
=
2^{1-2n}
\int_0^\infty
u^{2n}e^{u/4}\omega(e^u)\,du
=
H(n).
\end{equation}
Together with~\eqref{eq:F(x)}, we can obtain that
\begin{equation*}
\lambda_n=
\frac{\xi^{(2n)}(1/2)}{(2n)!}
=
\frac{H(n)}{\Gamma(2n+1)}
=\Lambda(n).
\end{equation*}
Substituting $u = 2t$ into~\eqref{eq:Rie_xi_form} yields~\eqref{eq:lambda-moment}.
This proves (i).

We next prove that $H(z)$ is holomorphic on
$\Omega:=\left\{z\in\mathbb C:\Re z>-1/2\right\}$.
For every fixed $u>0$, it is easy to see that the
function $2^{1-2z}u^{2z}e^{u/4}\omega(e^u)$ is holomorphic in $z$ on $\Omega$.
By the standard theorem on holomorphic parameter integrals (see~\cite[Section~2.7]{Has18}),
if the integrand is holomorphic and, on every compact subset of the parameter domain, is bounded in absolute value by an integrable function independent of $z$, then the integral is holomorphic.
In the following, we will show that $|2^{1-2z}u^{2z}e^{u/4}\omega(e^u)|$ is bounded by an integrable function independent of $z$.

Let $E \subseteq \Omega$ be any compact set. 
Then there exist real numbers $M_1$ and $M_2$ such that
\begin{equation*}
-\frac12<M_1 \le \Re z \le M_2, \quad \text{for}\ z\in E.
\end{equation*}
Moreover, $|2^{1-2z}|$ is bounded on $E$.

For $0<u\le1$,
the function $e^{u/4}\omega(e^u)$ is bounded because $e^u\in[1,e]$ and $\omega$ is continuous on $(0,\infty)$. 
Thus, for some constant $C_E>0$,
\begin{equation*}
\left|
2^{1-2z}
u^{2z}e^{u/4}\omega(e^u)
\right|
\le
C_Eu^{2M_1},
\quad
\text{for $0<u\le1$ and $z\in E$.}
\end{equation*}
Since $M_1>-1/2$,
\begin{equation*}
\int_0^1u^{2M_1}\,du
=
\frac1{2M_1+1}
<\infty.
\end{equation*}

For $u\ge1$, we have
\begin{equation*}
|\omega(e^u)|
\le
\sum_{j\ge1}
\left(
2\pi^2j^4e^{2u}
+
3\pi j^2e^u
\right)
e^{-\pi j^2e^u}.
\end{equation*}
For $j\ge1$, we have 
\begin{equation*}
e^{-\pi j^2e^u}
=
e^{-\pi e^u}
e^{-\pi e^u(j^2-1)}
\le
e^{-\pi e^u}
e^{-\pi(j^2-1)}.
\end{equation*}
Since
\begin{equation*}
\sum_{j\ge1}j^4e^{-\pi(j^2-1)}
<\infty
\quad\text{and}\quad
\sum_{j\ge1}j^2e^{-\pi(j^2-1)}
<\infty,
\end{equation*}
there exist constants $C,C_1,C_2>0$ such that
\begin{equation*}
|e^{u/4}\omega(e^u)|
\le
e^{u/4}\left(C_1 2\pi^2 e^{-\pi e^u}e^{2u}+C_2 3\pi e^{-\pi e^u}e^{u} \right)
\le
Ce^{9u/4}e^{-\pi e^u},
\quad u\ge1.
\end{equation*}
Since $\Re z\le M_2$ on $E$,
\begin{equation*}
|u^{2z}|=u^{2\Re z}\le u^{2M_2},
\quad u\ge1.
\end{equation*}
It follows that
\begin{equation*}
\left|
2^{1-2z}
u^{2z}e^{u/4}\omega(e^u)
\right|
\le
Cu^{2M_2}e^{9u/4}e^{-\pi e^u},
\quad
\text{for $u\ge1$ and $z\in E$,}
\end{equation*}
and the function on the right is integrable over $[1,\infty)$. 
Hence,
we obtain that $H(z)$ 
is holomorphic on $\Omega$.
Since $1/\Gamma(1+2z)$ is entire, $\Lambda(z)$ is holomorphic on $\Omega$ as well.
This proves (ii).

We now prove $\Lambda(x)>0$ for real $x>-1/2$.
For $u\ge0$, set $v_j=\pi j^2e^u$.
Then $v_j\ge\pi>3/2$, and
$2v_j^2-3v_j=v_j(2v_j-3)>0$.
Thus $\omega(e^u)>0$ for any $u\ge0$.
Since $2^{1-2x}>0$, $u^{2x}>0$, $e^{u/4}>0$ for $x>-1/2$ and $u>0$, it follows that $H(x)>0$.
Consequently,
\begin{equation*}
\Lambda(x)
=
\frac{H(x)}{\Gamma(2x+1)}
>0.
\end{equation*}
This completes the proof.
\end{proof}

%=================================================
\section{The case $n\ge 12000$ and $m\le \lceil\log_2 n\rceil$}\label{sec:analytic}
%=================================================
In this section, we establish the positivity of $A_{m,n}$ for every integer $n\ge12000$ and $0\le m\le\lceil\log_2 n\rceil$, together with the ratio bound $r_{\lceil\log_2 n\rceil,n}>(3+\sqrt5)/2$.

Throughout the paper, we use the following conventions for logarithms.
For the complex integration variable
$u\in\mathbb C\setminus(-\infty,0]$, $\Log u$ denotes
the principal logarithm:
\begin{equation*}
\Log u=\log|u|+i\Arg u,\qquad -\pi<\Arg u<\pi.
\end{equation*}
Complex powers are defined by $u^{2z}=\exp(2z\Log u)$.
For a nonvanishing holomorphic function $f$ on a simply
connected domain $D$, $\Log f(z)$ denotes a holomorphic
logarithm of $f$ on $D$, satisfying
\begin{equation*}
\exp(\Log f(z))=f(z).
\end{equation*}
Whenever $f$ is positive on a real interval
$I\subset D$, we choose the branch so that
$\Log f(x)=\log f(x)$ for $x\in I$.
The same branch is used at all points involved in each
expression.

We use 
$W$  to denote the principal branch $W_0$ of the
Lambert $W$ function. It is holomorphic on
$\mathbb C\setminus(-\infty,-e^{-1}]$
and satisfies
\begin{equation*}
W(\zeta)e^{W(\zeta)}=\zeta.
\end{equation*}
For real $x>0$, $W(x)$ is the unique positive solution
of $we^w=x$; see~\cite{CGHJK96}.
Square roots in these estimates are taken on the
principal branch:
\begin{equation*}
\sqrt{\zeta}
=
\exp\!\left(\frac12\Log\zeta\right),
\qquad
\zeta\in\mathbb C\setminus(-\infty,0].
\end{equation*}

We will bound the function $H(z)$ defined in~\eqref{eq:H-Lambda} in the sector by a saddle-point estimate,
where $w=W(2z/\pi)$ is the saddle point of the required function in the integrand of $H(z)$. 
Before that, we present the following result.

\begin{lemma}
Suppose that
\begin{equation}\label{sector}
|\arg z|\le\frac\pi4,\qquad |z|\ge4000.
\end{equation}
Let $w=W(2z/\pi)=a+ib$ and $T=|z|/a$.
Then we have 
\begin{equation}\label{eq:aT}
a\ge6,\qquad T\ge\frac\pi2e^6>630,
\end{equation}
\end{lemma}

\begin{proof}

First, we show that $a>0$.
For the principal branch $W$, we have $|b|<\pi$;
see~\cite[Section~4]{CGHJK96}.
Suppose that $a\le0$.
Since $we^w=2z/\pi$, using $e^{-a}\ge 1-a$, 
we obtain
\begin{equation*}
\frac{2|z|}{\pi}
=
e^{a}\sqrt{a^2+b^2}
<
e^{a}(-a+\pi)\le e^{a}(-a+1)\pi
\le \pi,
\end{equation*}
This implies $|z|<\pi^2/2<4000$, contradicting
the hypothesis. Therefore $a>0$.

Let $\vartheta=\arg z$. 
We claim that 
\begin{equation}\label{claim:proof 3.1}
|z|=\frac\pi2e^a\sqrt{a^2+b^2},\qquad
b+\arctan(b/a)=\vartheta.
\end{equation}

Taking $|we^w|=|2z/\pi|$, 
\begin{equation*}
\frac{2|z|}{\pi}
=|we^w|
=|w|e^{\operatorname{\Re}w}
=\sqrt{a^2+b^2}\cdot e^a.
\end{equation*}
This yields the first equation in~\eqref{claim:proof 3.1}.
Since $a>0$, we have $\arg w=\arctan(b/a)$, and hence
\begin{equation*}
we^w=\sqrt{a^2+b^2} \,\exp{\left(i\,\arctan\frac{b}{a}\right)} \,e^{a+ib}
=e^a\sqrt{a^2+b^2}\,
\exp\!\left(i\left(b+\arctan\frac{b}{a}\right)\right).
\end{equation*}
Comparing this with
$
{2z}/{\pi}={2|z|}e^{i\vartheta}/\pi
$
gives
\begin{equation*}
b+\arctan\frac{b}{a}-\vartheta\in 2\pi\mathbb{Z}.
\end{equation*}
Note that $b+\arctan\frac{b}{a}-\vartheta$ is continuous in $z$ on the sector. 
For $z>0$, we have $w>0$,
and $b=\vartheta=0$. Thus 
$b+\arctan\frac{b}{a}-\vartheta=0$,
and this proves the claim.

For each $a>0$, 
let 
$$f_a(b)=b+\arctan(b/a)$$
be the function of $b$ with $b\in(-\infty,\infty)$.
Because $f'_a(b)=1+a/(a^2+b^2)>0$,
the function $f_a(b)$ is increasing and $f_a(b)\in(-\infty,\infty)$.
Hence, for a fixed  $-\pi/4\le \vartheta\le \pi/4$,
the second equality in~\eqref{claim:proof 3.1}
$b+\arctan(b/a)=\vartheta$  uniquely determines
$b=b(a)$.
Since $a>0$, 
$$
|\vartheta|=|b+\arctan(b/a)|=|b|+|\arctan(b/a)|\ge |b|.
$$
Thus, 
\begin{equation}\label{b<pi/4}
|b|\le|\vartheta|\le\pi/4. 
\end{equation}

On the other hand, implicit differentiation of $b+\arctan(b/a)=\vartheta$ with respect to $a$ gives
\begin{equation*}
b'(a)=\frac{b}{a^2+b^2+a}.
\end{equation*}
Let $N_\vartheta(a)=(\pi/2)e^a\sqrt{a^2+b^2}$, where $b=b(a)$ is a function of $a$.
It follows that
\begin{equation*}
\frac{N_\vartheta'(a)}{N_\vartheta(a)}
=1+\frac{a+bb'}{a^2+b^2}
=\frac{(a+1)^2+b^2}{a^2+b^2+a}>0.
\end{equation*}
Since $N_\vartheta(a)>0$, we have $N_\vartheta'(a)>0$  for all $a>0$ and $N_\vartheta(a)$ is increasing.
Moreover,
\begin{equation*}
N_\vartheta(6)
\le\frac\pi2e^6\sqrt{36+\frac{\pi^2}{16}}<4000,
\qquad
N_\vartheta(a)\longrightarrow\infty\quad(a\longrightarrow\infty).
\end{equation*}
Thus, for each $|z|\ge4000$, we have $a>6$. This proves the first inequality in~\eqref{eq:aT}.

Recall that $w=a+ib$. We have proved $a>6$ and by~\eqref{b<pi/4}, $|b|\le \pi/4$. Then
\begin{equation}\label{1009}
1\le\frac{|w|}{a}=\sqrt{1+\frac{b^2}{a^2}}\le \sqrt{1+\frac{\pi^2}{16}\cdot\frac{1}{36}}
=\sqrt{1+\frac{\pi^2}{576}}<1.009.
\end{equation}
Since $we^w=2z/\pi$, 
we have 
\begin{equation}\label{eq:xi-saddle-parameters}
T=\frac{|z|}{a}=\frac{\pi |e^w|}{2}\frac{|w|}{a}=\frac{\pi e^a}{2}\frac{|w|}{a}.
\end{equation}
Thus $T\ge\frac\pi2e^6>630$. This proves~\eqref{eq:aT}.
\end{proof}

Moreover, since $a>6$ and $$T=\frac{\pi e^a}{2}\frac{|w|}{a},$$
by~\eqref{1009}, we have 
\begin{align}\label{piea}
\pi e^a=2T\frac{a}{|w|}\le2T,\qquad
\frac{\pi}{2}e^a\le T<\frac{\pi}{2}\cdot1.009e^a<1.585e^a, \qquad T=\frac{|z|}{a}\le\frac{|z|}{6}.
\end{align}

\subsection{A saddle estimate of $H(z)$}
In this subsection, we will give an explicit complex saddle estimate of $H(z)$.
Recall that
\begin{equation*}
\omega(t)=\sum_{k\ge1}(2\pi^2k^4t^2-3\pi k^2t)e^{-\pi k^2t},
\qquad \Re t>0,
\end{equation*}
and
\begin{equation*}
H(z)=2^{1-2z}\int_0^\infty u^{2z}e^{u/4}\omega(e^u)\,du,
\qquad \Re z>-\frac12.
\end{equation*}

\begin{lemma}\label{lem:xi-explicit-saddle}
Suppose that
$|\arg z|\le\frac\pi4$ and $|z|\ge4000$.
Let $w=W(2z/\pi)=a+ib$ and $T=|z|/a$
and let $H(z)$ be the function defined in~\eqref{eq:H-Lambda}.	Define
\begin{equation*}
H_*(z)=4\pi^2e^{7w/4}\sqrt{\frac{2w}{1+w}}
\exp\!\left(2z\left(\operatorname{Log}\frac w2-\frac1w\right)\right).
\end{equation*}
Then
\begin{equation}\label{eq:xi-saddle-error}
\left|\frac{H(z)}{H_*(z)}-1\right|\le\frac{30}{T},
\qquad
\left|\operatorname{Log}\frac{H(z)}{H_*(z)}\right|
\le\frac{32}{T}<0.051.
\end{equation}
Here the logarithm is the holomorphic branch of $\operatorname{Log}(1+v)$
with $|v|<1$. In particular, $H(z)$ has no zeros in the sector $\{z:|\arg z|\le\frac\pi4$, $|z|\ge4000\}$.
\end{lemma}

We now prove Lemma~\ref{lem:xi-explicit-saddle}.
In the sector~\eqref{sector}, we have $\Re z>0$ and
\begin{align*}
H(z)&=2^{1-2z}\int_0^\infty u^{2z}e^{u/4}\sum_{k\ge1}(2\pi^2k^4e^{2u}-3\pi k^2e^u)e^{-\pi k^2e^u}\,du\\
&=2^{1-2z}\int_0^\infty u^{2z}e^{u/4}\left((2\pi^2e^{2u}-3\pi e^u)e^{-\pi e^u}+\sum_{k\ge2}(2\pi^2k^4e^{2u}-3\pi k^2e^u)e^{-\pi k^2e^u}\right)\,du.
\end{align*}	
Let 
$$
H_1(z)=2^{1-2z}\int_0^\infty u^{2z}e^{u/4}(2\pi^2e^{2u}-3\pi e^u)e^{-\pi e^u}\,du
$$
and
$$
H_2(z)=2^{1-2z}\int_0^\infty u^{2z}e^{u/4}\sum_{k\ge2}(2\pi^2k^4e^{2u}-3\pi k^2e^u)e^{-\pi k^2e^u}\,du.
$$
Thus, $H(z)=H_1(z)+H_2(z)$.
We will estimate $H_1(z)$ and $H_2(z)$ separately in Sections 3.1.1 and 3.1.2.

\subsubsection{Estimate of $H_1(z)$}\label{subsec:EstH1}
We estimate $H_1(z)$ by the saddle-point method.
The integrand $u^{2z}e^{u/4}(2\pi^2e^{2u}-3\pi e^u)e^{-\pi e^u}$ equals
$$2\pi^2\exp\left({2z\operatorname{Log}u-\pi e^u+9u/4}\right)-3\pi \exp\left({2z\operatorname{Log}u-\pi e^u+5u/4}\right).
$$
Define
\begin{equation}\label{def:phi}
\phi(u)=2z\operatorname{Log}u-\pi e^u,\qquad
J_\beta(z)=\int_0^\infty e^{\phi(u)+\beta u}\,du,
\qquad \beta\in\left\{\frac94,\frac54\right\}.
\end{equation}
Then
\begin{equation}\label{H1-J}
H_1(z)=2^{1-2z}
\left(2\pi^2J_{9/4}(z)-3\pi J_{5/4}(z)\right).
\end{equation}
Clearly,
\begin{equation*}
\phi'(u)=\frac{2z}{u}-\pi e^u.
\end{equation*}
Thus 
$\phi'(u)=0$ if and only if $ue^u={2z}/{\pi}$.
Hence, 
\begin{equation}\label{saddle}
w=W(2z/\pi)
\end{equation}
is  the exact saddle of $\phi(u)$.

\smallskip
\noindent
\textbf{Contour deformation.}\par\nobreak\smallskip
In this part, we will give a contour deformation of 
$$
J_\beta(z)=\int_0^\infty e^{\phi(u)+\beta u}\,du,
\qquad \beta\in\left\{\frac94,\frac54\right\},
$$
in terms of the saddle point $w=a+ib$ of $\phi(u)=2z\operatorname{Log}u-\pi e^u$.
Let $\mathcal C_b$ be the segment from $0$ to $ib$, followed by the horizontal ray $s+ib$, $s\ge0$. 
We claim that 
\begin{equation}\label{deform}
J_\beta(z)
=\int_0^{ib}e^{\phi(u)+\beta u}\,du
+\int_0^\infty e^{\phi(s+ib)+\beta(s+ib)}\,ds.
\end{equation}
The horizontal ray passes through the saddle point $w=a+ib$.	
Put $y=s-a$. Then the saddle point $w$ corresponds to $y=0$ and
\begin{equation}\label{Jbeta}
J_\beta(z)=
\int_0^{ib}e^{\phi(u)+\beta u}\,du+\int_{-a}^\infty e^{\phi(w+y)+\beta(w+y)}\,dy.
\end{equation}
This is the desired contour deformation.

Now we prove~\eqref{deform}.
The region lies in the right half-plane, away
from the branch cut of $\operatorname{Log}u$. 
For $b>0$, take $0<\varepsilon<b$ and consider the
quarter-circle $C$ parametrized by
\begin{equation*}
u=\varepsilon e^{i\theta},
\qquad 0\le\theta\le\frac{\pi}{2}.
\end{equation*}
Since $e^{-\pi e^u+\beta u}$ is bounded near the origin,
we have, for $u$ in the quarter-circle,
\begin{equation*}
\left|e^{\phi(u)+\beta u}\right|=\left|e^{2z\Log u}\right| \left|e^{-\pi e^u+\beta u}\right|
=
\varepsilon^{2\operatorname{Re}z}
e^{-2\theta\operatorname{Im}z}
\left|e^{-\pi e^u+\beta u}\right|
\le C_z\varepsilon^{2\operatorname{Re}z},
\end{equation*}
where $C_z$ is independent of $\varepsilon$.
The length of $C$ is $\pi\varepsilon/2$, and hence
\begin{equation*}
\left|\int_C e^{\phi(u)+\beta u}\,du\right|
\le
\frac{\pi C_z}{2}\,
\varepsilon^{2\operatorname{Re}z+1}
\longrightarrow 0
\qquad (\varepsilon\to0^+),
\end{equation*}
because $\operatorname{Re}z>-1/2$.
For $b<0$, the same estimate holds for the quarter-circle
in the lower half-plane, parametrized by
$u=\varepsilon e^{i\theta}$ with $-\pi/2\le\theta\le0$.
When $b=0$, no contour deformation is needed.

For $R\ge1$, applying Cauchy's theorem to the finite contour
and letting $\varepsilon\to0^+$ gives
\begin{equation*}
\int_0^R e^{\phi(u)+\beta u}\,du
+\int_R^{R+ib} e^{\phi(u)+\beta u}\,du
=
\int_0^{ib} e^{\phi(u)+\beta u}\,du
+\int_0^R e^{\phi(s+ib)+\beta(s+ib)}\,ds.
\end{equation*}
On the right joining segment $u=R+it$, where
$\min\{0,b\}\le t\le\max\{0,b\}$, we have
$|t|\le |b|\le\pi/4$ and $|\Arg(R+it)|\le\pi/4$.
Consequently,
\begin{equation*}
\left|(R+it)^{2z}\right|
=
(R^2+t^2)^{\Re z}
e^{-2\Im z\,\Arg(R+it)}
\le C_z R^{2\Re z},
\qquad
C_z=(1+b^2)^{\Re z}e^{\pi|\Im z|/2},
\end{equation*}
where $C_z$ is independent of $R$ and $t$.
Since $\cos t\ge1/\sqrt2$, it follows that
\begin{equation*}
\left|
\int_R^{R+ib}e^{\phi(u)+\beta u}\,du
\right|
\le
|b|C_zR^{2\Re z}
\exp\!\left(\beta R-\frac{\pi}{\sqrt2}e^R\right)
\longrightarrow0
\qquad(R\to\infty).
\end{equation*}
Letting $R\to\infty$ in the finite contour identity
proves~\eqref{deform}.

\smallskip
\noindent
\textbf{Estimate of $J_\beta(z)$.}\par\nobreak\smallskip
We now use the expression~\eqref{Jbeta} to estimate $J_\beta(z)$ for $\beta\in\left\{9/4,5/4\right\}$.
Recall that 
\begin{equation*}
J_\beta(z)=
\int_0^{ib}e^{\phi(u)+\beta u}\,du+\int_{-a}^\infty e^{\phi(w+y)+\beta(w+y)}\,dy
\end{equation*}
and $y=0$ corresponds to the saddle point $w=a+ib$. We estimate the contributions from four parts of the
deformed contour:
\begin{itemize}
\item[\rm(i)]  \textbf{Gaussian main term and local error estimate.} $$I_G=\int_{-1/10}^{1/10} e^{\phi(w+y)+\beta(w+y)}\,dy;$$
\item[\rm(ii)] \textbf{horizontal tails} 
$$I_-=\int_{-a/2}^{-1/10} e^{\phi(w+y)+\beta(w+y)}\,dy,\qquad I_+=\int_{1/10}^{\infty} e^{\phi(w+y)+\beta(w+y)}\,dy;$$
\item[\rm(iii)] \textbf{the distant left interval}
$$I_d=\int_{-a}^{-a/2} e^{\phi(w+y)+\beta(w+y)}\,dy;$$
\item[\rm(iv)] \textbf{vertical segment}
$$
I_v=\int_0^{ib}e^{\phi(u)+\beta u}\,du
$$
\end{itemize}

\smallskip
\noindent
\textbf{Part(i): Gaussian main term and local error estimate.}
We estimate the integral
$$I_G=\int_{-1/10}^{1/10} e^{\phi(w+y)+\beta(w+y)}\,dy.$$
Recall that the horizontal ray passes through the saddle point $w$ (corresponding to $y=0$)
and $\phi'(w)=0$.
Taylor's formula at $y=0$ gives
\begin{equation}\label{taylor}
\phi(w+y)=\phi(w)+\frac{\phi''(w)}{2} y^2+\frac{\phi'''(w)}{6}y^3+R_4(y).
\end{equation}
Let $\kappa=-\phi''(w)$ and $c_3={\phi'''(w)}/{6}$. Then
\begin{equation}\label{kappa-first}
e^{\phi(w+y)+\beta(w+y)}=e^{\phi(w)+\beta w} e^{{-\kappa y^2}/{2}}e^{\beta y+c_3y^3+R_4(y)}.
\end{equation}
Let
$
X(y):=\beta y+c_3y^3+R_4(y).
$
Then
\begin{equation}\label{kappa-second}
e^{\phi(w+y)+\beta(w+y)}=e^{\phi(w)+\beta w} e^{{-\kappa y^2}/{2}}+e^{\phi(w)+\beta w} e^{{-\kappa y^2}/{2}}\left(e^{X(y)}-1\right)
\end{equation}
and
$$
I_G
=e^{\phi(w)+\beta w} \int_{-1/10}^{1/10} e^{{-\kappa y^2}/{2}} \,dy
+e^{\phi(w)+\beta w} \int_{-1/10}^{1/10} e^{{-\kappa y^2}/{2}}\left(e^{X(y)}-1\right) \,dy.
$$
We will estimate the two integrals on the right-hand side above.

\smallskip
\noindent
\textbf{\emph{The Gaussian main term.}}
For the first part,
define the Gaussian main term  by
\begin{equation}\label{Gaussmain}
M_\beta= e^{\phi(w)+\beta w}\int_{-\infty}^\infty e^{{-\kappa y^2}/{2}}\,dy.
\end{equation}
Then 
$$e^{\phi(w)+\beta w} \int_{-1/10}^{1/10} e^{{-\kappa y^2}/{2}} \,dy=
M_\beta-
e^{\phi(w)+\beta w} \int_{y\in\mathbb{R},|y|\ge 1/10} e^{{-\kappa y^2}/{2}} \,dy.$$
We will prove in Section 3.1.3 that $$2^{1-2z}2\pi^2M_{9/4}=H_*(z).$$
In each case (i)-(iv), we estimate the corresponding error contribution
relative to the Gaussian main term $M_\beta$.
Combining these estimates will then give a bound for
\begin{equation*}
\frac{|J_\beta(z)-M_\beta|}{|M_\beta|}
=
\left|\frac{J_\beta(z)}{M_\beta}-1\right|.
\end{equation*}

We now give a bound for $M_\beta$, and firstly
we estimate $\kappa$.
Recall that $we^w=2z/\pi$ and $w=a+ib$.
A direct computation gives
\begin{equation*}
\kappa=-\phi''(w)
=\frac{2z}{w^2}+\pi e^w
=\pi e^w\left(1+\frac1w\right)=\pi e^a(\cos b+i \sin b) \left(1+\frac{a-bi}{a^2+b^2}\right).
\end{equation*}
Since $a>6, |b|\le \pi/4$ by~\eqref{eq:aT} and~\eqref{b<pi/4},
we have $a \cos b\ge 0$ and $b\sin b\ge 0$. Hence,
\begin{equation*}
\Re\kappa=\pi e^a\left(\cos b+\frac{a\cos b+b\sin b}{a^2+b^2}\right)\ge \pi e^a \cos b\ge \frac{\pi e^a}{\sqrt 2}.
\end{equation*}
Recall that by~\eqref{1009}, ${|w|}/{a}<1.009$
and by~\eqref{eq:xi-saddle-parameters},
$\pi e^a={2aT}/{|w|}$.
We obtain that 
\begin{equation}\label{eq:xi-saddle-curvature}
\Re\kappa
\ge \frac{\pi e^a}{\sqrt{2}}
= \sqrt{2}\,T\frac{a}{|w|}
> \frac{\sqrt{2}}{1.009}\,T
> 1.4T.
\end{equation}
Similarly, 
\begin{equation}\label{kappa}
|\kappa|=\pi|e^w|\left|1+\frac{1}{w}\right|\le\pi e^a \left(1+\frac{1}{|w|}\right)\le 2T\left(1+\frac{1}{6}\right)=\frac73T.
\end{equation}
Thus we have 
\begin{equation}\label{bound-kappa}
1.4T<|\kappa|\le 7T/3.
\end{equation}

Since $\operatorname{Re}\kappa>0$, the complex Gaussian
integral formula \cite[Appendix A, Theorem 1]{Fol89}
gives
\begin{equation}\label{eq:G}
G:=\int_{-\infty}^{\infty}e^{-\kappa y^2/2}\,dy
=\sqrt{2\pi}\,\kappa^{-1/2},
\end{equation}
where $\kappa^{1/2}$ denotes the principal square root,
which is positive for positive real $\kappa$.
Since $|\kappa|\le 7T/3$ by~\eqref{kappa},
\begin{equation}\label{low-G}
|G|
=
\frac{\sqrt{2\pi}}{\sqrt{|\kappa|}}
\ge
\sqrt{\frac{6\pi}{7}}\,T^{-1/2}>1.64\,T^{-1/2}.
\end{equation}
Then $$M_\beta=
e^{\phi(w)+\beta w}G=\sqrt{2\pi}\,\kappa^{-1/2}\,e^{\phi(w)+\beta w}$$ and
\begin{equation*}
|M_\beta|
=
e^{\Re \phi(w)+\beta a} |G| \ge 
e^{\Re \phi(w)+\beta a}\sqrt{\frac{6\pi}{7}}\,T^{-1/2}\ge 1.64 e^{\Re \phi(w)+\beta a}T^{-1/2}.
\end{equation*}

We next estimate 
$$
M_\beta-
e^{\phi(w)+\beta w} \int_{y\in\mathbb{R},|y|\ge 1/10} e^{{-\kappa y^2}/{2}} \,dy.
$$
By~\eqref{eq:G}, we have 
\begin{equation*}
\varepsilon_G(z)
:=
\frac{1}{G}
\int_{-1/10}^{1/10}e^{-\kappa y^2/2}\,dy-1
=
-\frac{1}{G}
\int_{|y|\ge 1/10}e^{-\kappa y^2/2}\,dy.
\end{equation*}
Since $M_\beta=e^{\phi(w)+\beta w}G$, we obtain
\begin{equation}\label{eq:xi-gaussian-truncation}
e^{\phi(w)+\beta w}
\int_{-1/10}^{1/10}e^{-\kappa y^2/2}\,dy
=
M_\beta\bigl(1+\varepsilon_G(z)\bigr).
\end{equation}
By~\eqref{eq:xi-saddle-curvature} and~\eqref{low-G}, we have 
$
\operatorname{Re}\kappa\ge 1.4T,
$ and 
$|G|\ge 1.64T^{-1/2}$.
Consequently,
\begin{align}\label{epsilonG}
|\varepsilon_G(z)|
&\le
\frac{2\sqrt{T}}{1.64}
\int_{1/10}^{\infty}e^{-0.7Ty^2}\,dy
=
\frac{2\sqrt{T}}{1.64}
\int_0^{\infty}e^{-0.7T(1/10+t)^2}\,dt\nonumber\\
&\le
\frac{2\sqrt{T}}{1.64}\,e^{-0.007T}
\int_0^{\infty}e^{-0.14Tt}\,dt
=
\frac{2}{0.14\cdot1.64}
\frac{e^{-0.007T}}{\sqrt{T}}
\le
\frac{8.8e^{-0.007T}}{\sqrt{T}}.
\end{align}

\smallskip
\noindent
\textbf{\emph{The local error estimate.}}
Now we bound 
$$e^{\phi(w)+\beta w} \int_{-1/10}^{1/10} e^{{-\kappa y^2}/{2}}\left(e^{X(y)}-1\right) \,dy,$$
where $X(y)=\beta y+c_3y^3+R_4(y)$, $c_3={\phi'''(w)}/{6}$ and $\beta\in\{9/4,5/4\}$. 

Since $\phi(u)=2z\operatorname{Log}u-\pi e^u$, we have
\begin{equation*}
\phi'''(u)=\frac{4z}{u^3}-\pi e^u,
\qquad
\phi^{(4)}(u)=-\frac{12z}{u^4}-\pi e^u.
\end{equation*}
Using $|w|\ge a\ge6$, $|z|=aT$, and
$\pi e^a\le2T$ by~\eqref{piea}, we obtain
\begin{equation}\label{upper-c3}
|c_3|
=\frac{|\phi'''(w)|}{6}
\le
\frac16\left(\frac{4|z|}{|w|^3}+\pi e^a\right)
\le
\frac16\left(\frac4{a^2}+2\right)T
\le
\frac{19}{54}T
<0.4T.
\end{equation}

To control the Taylor remainder, we estimate the fourth
derivative uniformly on the local segment.
For real $|y|\le1/10$,
\begin{equation*}
|w+y|\ge a-0.1,
\qquad
|e^{w+y}|=e^{a+y}\le e^{a+0.1}.
\end{equation*}
Hence
\begin{equation*}
\begin{aligned}
|\phi^{(4)}(w+y)|
\le
\frac{12|z|}{|w+y|^4}+\pi|e^{w+y}|\le
T\left(\frac{12a}{(a-0.1)^4}+2e^{0.1}\right).
\end{aligned}
\end{equation*}
The function $a/(a-0.1)^4$ is decreasing for $a\ge6$, since
\begin{equation*}
\frac{d}{da}\left(\frac{a}{(a-0.1)^4}\right)
=
-\frac{3a+0.1}{(a-0.1)^5}<0.
\end{equation*}
Consequently, for every real $y$ with $|y|\le1/10$,
\begin{equation*}
|\phi^{(4)}(w+y)|
\le
T\left(\frac{72}{5.9^4}+2e^{0.1}\right)
<2.4T.
\end{equation*}

Taylor's formula~\eqref{taylor}, together with
$\phi'(w)=0$, yields
\begin{equation}\label{eq:xi-saddle-taylor}
\phi(w+y)=\phi(w)-\frac{\kappa y^2}{2}+c_3y^3+R_4(y),
\qquad |c_3|\le0.4T,\quad |R_4(y)|\le0.1T|y|^4.
\end{equation}
Indeed, $c_3=\phi'''(w)/6$, and the integral remainder is
\begin{equation*}
R_4(y)=\frac{y^4}{6}\int_0^1(1-t)^3\phi^{(4)}(w+ty)\,dt,
\end{equation*}
so its modulus is at most $2.4T|y|^4/24$.

Recall that  $\beta\in\left\{9/4,5/4\right\}$ and 
$
X(y)=\beta y+c_3y^3+R_4(y).
$
For $|y|\le0.1$, \eqref{eq:xi-saddle-taylor} gives
\begin{equation*}
|X(y)|\le \beta |y|+0.4 T |y|^3+0.1T|y|^4\le\beta|y|+0.041Ty^2.
\end{equation*}
Since $T>0$ and $y$ is real, we have
\begin{equation*}
0\le (0.1T|y|-\beta)^2
=0.01T^2y^2-0.2T\beta|y|+\beta^2.
\end{equation*}
Dividing by $0.2T$ and rearranging yields
\begin{equation*}
\beta|y|
\le 0.05Ty^2+\frac{\beta^2}{0.2T},\qquad 
|X(y)|
\le \beta|y|+0.041Ty^2
\le 0.091Ty^2+\frac{\beta^2}{0.2T}.
\end{equation*}
Because $\beta^2\le81/16$, $T>630$, and
$e^{(81/16)/(0.2\cdot630)}<1.2$, we have 
\begin{equation}\label{expon}
e^{|X(y)|}\le e^{0.091Ty^2+\frac{\beta^2}{0.2T}}\le1.2e^{0.091Ty^2}.
\end{equation}
By~\eqref{eq:xi-saddle-taylor}, $|c_3|\le0.4T$, and  $|R_4(y)|\le0.1T|y|^4$. Then $(A+B+C)^2\le3(A^2+B^2+C^2)$ gives
\begin{equation}\label{three-summand}
|X(y)|^2\le3\beta^2y^2+0.48T^2y^6+0.03T^2y^8.
\end{equation}
Since
\begin{equation*}
e^X-1-\beta y-c_3y^3=(e^X-1-X)+R_4(y),
\qquad |e^X-1-X|\le\tfrac12|X|^2e^{|X|},
\end{equation*}
by~\eqref{eq:xi-saddle-taylor},~\eqref{expon} and~\eqref{three-summand}, we obtain that
\begin{align}\label{eq:xi-saddle-exp-remainder}
|e^{X(y)}-1-\beta y-c_3y^3|&\le \tfrac12|X|^2e^{|X|}+|R_4(y)|\nonumber\\
&\le 0.6e^{0.091Ty^2}
\bigl(3\beta^2y^2+0.48T^2y^6+0.03T^2y^8\bigr)+0.1Ty^4.
\end{align}

Let
\begin{equation*}
E_{\mathrm{loc},\beta}
=\int_{-1/10}^{1/10}
e^{-\kappa y^2/2}\bigl(e^{X(y)}-1\bigr)\,dy.
\end{equation*}
Since $y$ is real and the interval is symmetric,
$$
\int_{-1/10}^{1/10}
e^{-\kappa y^2/2}\beta y\,dy=0,\qquad \int_{-1/10}^{1/10}
e^{-\kappa y^2/2}c_3y^3\,dy=0
$$
Note that by~\eqref{eq:xi-saddle-curvature},
$\Re \kappa\ge 1.4T$. Then $|e^{-\kappa y^2/2}|\le e^{-0.7Ty^2}$.
By~\eqref{eq:xi-saddle-exp-remainder}, we have 
\begin{align*}
|E_{\mathrm{loc},\beta}|
&=\left|\int_{-1/10}^{1/10}
e^{-\kappa y^2/2}\bigl(e^{X(y)}-1-\beta y-c_3y^3\bigr)\,dy\right|\\
&\le 0.1T\int_{\mathbb R}y^4e^{-0.7Ty^2}\,dy
+0.6\int_{\mathbb R}e^{-0.6Ty^2}
\bigl(3\beta^2y^2+0.48T^2y^6+0.03T^2y^8\bigr)\,dy.
\end{align*}
Here $0.7-0.091=0.609>0.6$. 
Recall that by~\eqref{low-G}, $|G|\ge \sqrt{6\pi/7}\cdot T^{-1/2}$.
In the following, we estimate $|E_{\mathrm{loc},\beta}|/|G|$.

For $v>0$ and integers $r\ge0$, the Gaussian moment
formula \cite[Eq.~(5.1)]{Mol07}, together with
the evenness of the integrand, gives
\begin{equation*}
\int_{\mathbb R}y^{2r}e^{-vy^2}\,dy
=
2\int_0^\infty y^{2r}e^{-vy^2}\,dy
=
\frac{(2r-1)!!\sqrt{\pi}}{2^r v^{r+1/2}},
\end{equation*}
where $(-1)!!=1$.
Using the Gaussian moment formula
and $\beta^2\le81/16$ (since $\beta\in\{9/4,5/4\}$), we obtain
\begin{align*}
\frac{0.1T}{|G|}
\int_{\mathbb R} y^4e^{-0.7Ty^2}\,dy
&\le
\frac{0.1T}{\sqrt{6\pi/7}\cdot T^{-1/2}}\,
\frac{3\sqrt{\pi}}{4(0.7T)^{5/2}}=
\frac{0.1\cdot3\sqrt{\pi}}
{4(0.7)^{5/2}\sqrt{6\pi/7}}\,\frac{1}{T};\\
\frac{0.6\cdot3\beta^2}{|G|}
\int_{\mathbb R} y^2e^{-0.6Ty^2}\,dy
&\le
\frac{0.6\cdot3\beta^2}{\sqrt{6\pi/7}\cdot T^{-1/2}}\,
\frac{\sqrt{\pi}}{2(0.6T)^{3/2}}
\le
\frac{0.6\cdot3(81/16)}{\sqrt{6\pi/7}\cdot T^{-1/2}}\,
\frac{\sqrt{\pi}}{2(0.6T)^{3/2}}\\
&=
\frac{0.6\cdot3(81/16)\sqrt{\pi}}
{2(0.6)^{3/2}\sqrt{6\pi/7}}\,\frac{1}{T};\\
\frac{0.6\cdot0.48T^2}{|G|}
\int_{\mathbb R} y^6e^{-0.6Ty^2}\,dy
&\le
\frac{0.6\cdot0.48T^2}{\sqrt{6\pi/7}\cdot T^{-1/2}}\,
\frac{15\sqrt{\pi}}{8(0.6T)^{7/2}}
=
\frac{0.6\cdot0.48\cdot15\sqrt{\pi}}
{8(0.6)^{7/2}\sqrt{6\pi/7}}\,\frac{1}{T}\\
\frac{0.6\cdot0.03T^2}{|G|}
\int_{\mathbb R} y^8e^{-0.6Ty^2}\,dy
&\le
\frac{0.6\cdot0.03T^2}{\sqrt{6\pi/7}\cdot T^{-1/2}}\,
\frac{105\sqrt{\pi}}{16(0.6T)^{9/2}}=
\frac{0.6\cdot0.03\cdot105\sqrt{\pi}}
{16(0.6)^{9/2}\sqrt{6\pi/7}}\,\frac{1}{T^2}.
\end{align*}
We conclude that
\begin{equation}\label{eq:xi-saddle-local-error}
\frac{|E_{\mathrm{loc},\beta}|}{|G|}
\le\frac{14.4}{T}+\frac{1.3}{T^2}.
\end{equation}
Combining the Gaussian truncation estimate~\eqref{eq:xi-gaussian-truncation},~\eqref{epsilonG} with~\eqref{eq:xi-saddle-local-error},
we obtain the following relative error bound for $I_G$:
\begin{equation}\label{bound:IG}
\left|\frac{I_G}{M_\beta}-1\right|
\le |\varepsilon_G(z)|+\frac{|E_{\mathrm{loc},\beta}|}{|G|}
\le \frac{8.8e^{-0.007T}}{\sqrt{T}}
+\frac{14.4}{T}+\frac{1.3}{T^2}.
\end{equation}

\smallskip
\noindent
\textbf{Part(ii): Horizontal tails.}
Let
\begin{equation*}
I_-:=\int_{-a/2}^{-1/10}
e^{\phi(w+y)+\beta(w+y)}\,dy,
\qquad
I_+:=\int_{1/10}^{\infty}
e^{\phi(w+y)+\beta(w+y)}\,dy.
\end{equation*}
In this part, we estimate $I_-$ and $I_+$.

Let
$
\Psi(y):=\Re\bigl(\phi(w+y)-\phi(w)\bigr).
$
Since $w=a+ib$ and $\beta,y\in\mathbb R$, we have 
\begin{align}\label{estimate-main}
\left|e^{\phi(w+y)+\beta(w+y)}\right|
&=
e^{\operatorname{\Re}\left(\phi(w+y)+\beta(w+y)\right)}
=
e^{\operatorname{\Re}\phi(w+y)+\beta(a+y)}\nonumber\\
&=
e^{\operatorname{\Re}\phi(w)+\Psi(y)+\beta a+\beta y}
=
e^{\operatorname{\Re}\phi(w)+\beta a}
e^{\Psi(y)+\beta y}.
\end{align}
We will derive an upper bound of $\Psi(y)+\beta y$.

We first establish the concavity needed to control
the horizontal tails. For real $y\ge-a/2$,
\begin{equation*}
w+y=(a+y)+ib,
\qquad a+y\ge a/2>0.
\end{equation*}
Recall that $\vartheta=\arg z$ and $|\vartheta|\le \pi/4$.
If $b\ge0$,  then by~\eqref{claim:proof 3.1}, $b\le \vartheta$
and
\begin{equation*}
\begin{aligned}
0\le 2\arg(w+y)
&=2\arctan\frac{b}{a+y}
\le \frac{2b}{a+y}
\le \frac{4b}{a}
\le \frac{2}{3}\vartheta,
\end{aligned}
\end{equation*}
where we used $\arctan t\le t$ for $t\ge0$,
$a>6$.
Consequently,
\begin{equation}\label{aug}
\arg\left(\frac{2z}{(w+y)^2}\right)
=
\vartheta-2\arg(w+y)
\in
\left[\frac{\vartheta}{3},\vartheta\right]
\subset
\left[0,\frac{\pi}{4}\right].
\end{equation}
Hence
\begin{equation*}
\operatorname{\Re}\frac{2z}{(w+y)^2}>0.
\end{equation*}
For $b<0$, the same positivity follows by complex
conjugation. Since $|b|\le\pi/4$, we also have
$\cos b>0$. Therefore, for $y\ge-a/2$, $\phi(u)=2z\operatorname{Log}u-\pi e^u$ implies that 
\begin{equation*}
\begin{aligned}
\frac{d^2}{dy^2}{\Re}\phi(w+y)
=\operatorname{\Re}\phi''(w+y)
=-\operatorname{\Re}\frac{2z}{(w+y)^2}
-\pi e^{a+y}\cos b
<0.
\end{aligned}
\end{equation*}
Thus $\operatorname{\Re}\phi(w+y)$ is strictly concave
on $[-a/2,\infty)$.

Since $
\Psi(y):=\Re\bigl(\phi(w+y)-\phi(w)\bigr)
$, we have  $\Psi(0)=\Psi'(0)=0$. On $|y|\le 1/10$, since $|b|\le \pi/4$,  by~\eqref{aug}, we obtain
\begin{equation*}
\begin{aligned}
-\Psi''(y)
=\Re\frac{2z}{(w+y)^2}+\pi e^{a+y}\cos b\ge \pi e^{a+y}\cos b
\ge\frac{\pi e^{a-0.1}}{\sqrt2}
>\frac{\sqrt2e^{-0.1}}{1.009}T>1.26T,
\end{aligned}
\end{equation*}
where we use $\pi e^a\ge 2T/1.009$ by~\eqref{1009} and~\eqref{eq:xi-saddle-parameters}.
For $0\le y\le 1/10$, integrating from $0$ to $y$ yields
\begin{equation*}
\begin{aligned}
\Psi'(y)
=\Psi'(0)+\int_0^y\Psi''(t)\,dt
\le -1.26T\int_0^y dt
=-1.26Ty.
\end{aligned}
\end{equation*}
Integrating once more, we obtain
\begin{equation*}
\begin{aligned}
\Psi(y)
=\Psi(0)+\int_0^y\Psi'(t)\,dt
\le -1.26T\int_0^y t\,dt
=-\frac{1.26T}{2}y^2.
\end{aligned}
\end{equation*}
Similarly,
for $-1/10\le y\le0$,  we have 
\begin{equation*}
\Psi'(y)\ge -1.26Ty,\qquad \Psi(y)\le -\frac{1.26T}{2}y^2.
\end{equation*}
Evaluating these bounds at $y=\pm 1/10$, we conclude that
\begin{equation*}
\begin{aligned}
\Psi(\pm 1/10)
\le-0.0063T,\qquad
\Psi'(1/10)
\le-0.126T,\qquad 
\Psi'(-1/10)
\ge0.126T.
\end{aligned}
\end{equation*}

We have proved that $\operatorname{Re}\phi(w+y)$ is strictly concave
on $[-a/2,\infty)$.
Hence
\begin{equation}\label{psit}
\Psi(y)\le-0.0063T-0.126T(|y|-1/10)
\end{equation}	on both horizontal tails $y\in[-a/2,-1/10]\cup[1/10,\infty)$.
Writing $t=|y|-1/10\ge0$ and using $\beta\le2.25$, we obtain
\begin{equation*}
\begin{aligned}
\beta y
\le \beta|y|
\le 2.25|y|
=2.25\left(t+\frac{1}{10}\right)
=2.25t+0.225.
\end{aligned}
\end{equation*}
Combining these two inequalities, we obtain
\begin{equation}\label{psi+beta}
\begin{aligned}
\Psi(y)+\beta y
\le -0.0063T-0.126Tt+2.25t+0.225
=0.225-0.0063T-(0.126T-2.25)t.
\end{aligned}
\end{equation}

Now we are ready to estimate $I_-$ and $I_+$.
Recall that
$
M_\beta=e^{\phi(w)+\beta w}G,
$
and
\begin{equation*}
|M_\beta|
=
e^{\operatorname{\Re}\phi(w)+\beta a}|G|.
\end{equation*}
Consequently, 
\begin{equation*}
\begin{aligned}
|I_-|
&=
\left|
\int_{-a/2}^{-1/10}
e^{\phi(w+y)+\beta(w+y)}\,dy
\right|
\le
\int_{-a/2}^{-1/10}
\left|e^{\phi(w+y)+\beta(w+y)}\right|\,dy\\
&=
e^{\operatorname{\Re}\phi(w)+\beta a}
\int_{-a/2}^{-1/10}e^{\Psi(y)+\beta y}\,dy.
\end{aligned}
\end{equation*}
Similarly, for the right-tail integral,
\begin{equation*}
\begin{aligned}
|I_+|
=
\left|
\int_{1/10}^{\infty}
e^{\phi(w+y)+\beta(w+y)}\,dy
\right|
\le
e^{\operatorname{\Re}\phi(w)+\beta a}
\int_{1/10}^{\infty}e^{\Psi(y)+\beta y}\,dy.
\end{aligned}
\end{equation*}
Hence, 
\begin{align}\label{I-I+}
\frac{|I_-|+|I_+|}{|M_\beta|}
&\le
\frac{e^{\operatorname{\Re}\phi(w)+\beta a}}
{e^{\operatorname{\Re}\phi(w)+\beta a}|G|}
\left(
\int_{-a/2}^{-1/10}e^{\Psi(y)+\beta y}\,dy
+
\int_{1/10}^{\infty}e^{\Psi(y)+\beta y}\,dy
\right)\nonumber\\
&=
\frac{1}{|G|}
\left(
\int_{-a/2}^{-1/10}e^{\Psi(y)+\beta y}\,dy
+
\int_{1/10}^{\infty}e^{\Psi(y)+\beta y}\,dy
\right).
\end{align}

On both tail intervals, put $t=|y|-1/10\ge0$.
By~\eqref{psit} and~\eqref{psi+beta},
\begin{equation*}
\Psi(y)\le-0.0063T-0.126Tt,\qquad 
e^{\Psi(y)+\beta y}
\le
e^{0.225-0.0063T}
e^{-(0.126T-2.25)t}.
\end{equation*}
For the right tail, the substitution $t=y-1/10$
gives
\begin{equation*}
\begin{aligned}
\int_{1/10}^{\infty}e^{\Psi(y)+\beta y}\,dy
&=
\int_0^\infty
e^{\Psi(t+1/10)+\beta(t+1/10)}\,dt
\le
e^{0.225-0.0063T}
\int_0^\infty e^{-(0.126T-2.25)t}\,dt.
\end{aligned}
\end{equation*}
For the left tail, use $t=-y-1/10$.
After reversing the limits of integration, we obtain
\begin{equation*}
\begin{aligned}
\int_{-a/2}^{-1/10}e^{\Psi(y)+\beta y}\,dy
&=
\int_0^{a/2-1/10}
e^{\Psi(-t-1/10)-\beta(t+1/10)}\,dt\\
&\le
e^{0.225-0.0063T}
\int_0^{a/2-1/10}
e^{-(0.126T-2.25)t}\,dt\\
&\le
e^{0.225-0.0063T}
\int_0^\infty
e^{-(0.126T-2.25)t}\,dt.
\end{aligned}
\end{equation*}

Since $T>630$, we have $0.126T-2.25>0$. Therefore,
\begin{equation*}
\int_0^\infty e^{-(0.126T-2.25)t}\,dt
=
\frac{1}{0.126T-2.25}.
\end{equation*}
By~\eqref{I-I+}, combining the bounds for the two tails yields
\begin{equation*}
\begin{aligned}
\frac{|I_-|+|I_+|}{|M_\beta|}
\le
\frac{2e^{0.225-0.0063T}}{|G|}
\int_0^\infty e^{-(0.126T-2.25)t}\,dt
=
\frac{2e^{0.225-0.0063T}}
{|G|(0.126T-2.25)}.
\end{aligned}
\end{equation*}
By~\eqref{low-G},
$
|G|\ge \sqrt{6\pi/7}\cdot T^{-1/2}>1.64 T^{-1/2}.
$
Since $T>630$, we conclude that
\begin{align}\label{bound:I-I+}
\frac{|I_-|+|I_+|}{|M_\beta|}
&\le
\frac{2e^{0.225-0.0063T}\sqrt{T}}
{1.64\cdot(0.126T-2.25)}
=
\frac{2e^{0.225}}
{1.64\cdot(0.126-2.25/T)}
\frac{e^{-0.0063T}}{\sqrt{T}}\nonumber\\
&\le \frac{2e^{0.225}}
{1.64(0.126-2.25/630)}\frac{e^{-0.0063T}}{\sqrt{T}}\le\frac{12.6\,e^{-0.0063T}}{\sqrt{T}}.
\end{align}

\smallskip
\noindent
\textbf{Part(iii): Distant left interval.}
In this part, we estimate
$$I_d=\int_{-a}^{-a/2} e^{\phi(w+y)+\beta(w+y)}\,dy.$$
Since $w=a+ib$ and $\beta,y\in\mathbb R$, we have
\begin{equation*}
\begin{aligned}
\left|e^{\phi(w+y)+\beta(w+y)}\right|
&=
e^{\operatorname{\Re}(\phi(w+y)+\beta(w+y))}=
e^{\operatorname{\Re}\phi(w+y)+\beta(a+y)}.
\end{aligned}
\end{equation*}
Moreover, $M_\beta=e^{\phi(w)+\beta w}G$ implies
$
|M_\beta|
=
e^{\operatorname{\Re}\phi(w)+\beta a}|G|.
$
This gives
\begin{equation*}
\begin{aligned}
\frac{\left|e^{\phi(w+y)+\beta(w+y)}\right|}
{|M_\beta|}
&=
\frac{1}{|G|}
e^{\operatorname{Re}(\phi(w+y)-\phi(w))+\beta y}.
\end{aligned}
\end{equation*}
Since $M_\beta$ and $G$ are independent of $y$,
integration over $[-a,-a/2]$ yields
\begin{equation}\label{Id-G}
\left|\frac{I_d}{M_\beta}\right|\le\frac{1}{|M_\beta|}
\int_{-a}^{-a/2}
\left|e^{\phi(w+y)+\beta(w+y)}\right|\,dy
=
\frac{1}{|G|}
\int_{-a}^{-a/2}
e^{\operatorname{\Re}(\phi(w+y)-\phi(w))+\beta y}\,dy.
\end{equation}

Now we estimate $\operatorname{\Re}(\phi(w+y)-\phi(w))$.
For $-a\le y\le -a/2$, we have $0\le a+y\le a/2$.
Since $a\ge 6$ and $|b|\le \pi/4$, it follows that
\begin{align}\label{0.52}
\frac{|w+y|}{|w|}
&=
\frac{\sqrt{(a+y)^2+b^2}}{\sqrt{a^2+b^2}}\le \sqrt{\frac{a^2/4+b^2}{a^2+b^2}}\le
\sqrt{\frac14+\frac{b^2}{a^2}}
\le
\sqrt{\frac14+\frac{\pi^2}{576}}
<0.52.
\end{align}

Recall that $\phi(u)=2z\operatorname{Log}u-\pi e^u$
and $w=a+ib$. For real $y\in[-a,-a/2]$ with $w+y\ne0$,
the principal logarithm gives
\begin{equation*}
\operatorname{Log}(w+y)-\operatorname{Log}w
=
\log\frac{|w+y|}{|w|}
+i\bigl(\arg(w+y)-\arg w\bigr).
\end{equation*}
Consequently,
\begin{align*}
\phi(w+y)-\phi(w)
&=
2z\bigl(\operatorname{Log}(w+y)-\operatorname{Log}w\bigr)
+\pi e^w(1-e^y)
\\
&=
2z\log\frac{|w+y|}{|w|}
+2iz\bigl(\arg(w+y)-\arg w\bigr)
+\pi e^a(\cos b+i\sin b)(1-e^y).
\end{align*}
Since $y$ is real, the quantities
$\log(|w+y|/|w|)$, $\arg(w+y)-\arg w$, and $1-e^y$
are all real. Taking real parts and using
$\Re(iz)=-\Im z$, we obtain
\begin{align*}
\Re\bigl(\phi(w+y)-\phi(w)\bigr)
=
2\Re z\log\frac{|w+y|}{|w|}
-2\Im z\bigl(\arg(w+y)-\arg w\bigr)
+\pi e^a\cos b\,(1-e^y).
\end{align*}

We claim that the argument term
$
-2\Im z\bigl(\arg(w+y)-\arg w\bigr)
$
is nonpositive.
Indeed, by~\eqref{claim:proof 3.1},
\begin{equation*}
\vartheta=b+\arctan(b/a).
\end{equation*}
Thus $\vartheta$ has the same sign as $b$.
Since $|\vartheta|\le\pi/4$, the same is true of
$\operatorname{\Im}z=|z|\sin\vartheta$.
For $-a<y\le-a/2$, we have $0<a+y\le a/2$. Hence
\begin{equation*}
\arg(w+y)-\arg w
=
\arctan\frac{b}{a+y}-\arctan\frac{b}{a}.
\end{equation*}
By the strict monotonicity of $\arctan$, this difference
also has the same sign as $b$. Consequently,
\begin{equation*}
-2\operatorname{\Im}z
\bigl(\arg(w+y)-\arg w\bigr)\le0.
\end{equation*}

When $b\ne0$, the same inequality holds at $y=-a$
by continuity.
When $b=0$, the argument term vanishes for $y>-a$,
and the integrand $e^{\phi(w+y)+\beta(w+y)}$
extends continuously to zero at $y=-a$.

Moreover, since $y<0$ and $\cos b>0$, the saddle relation
$\pi e^a|w|=2|z|$ gives
\begin{equation*}
0
\le
\pi e^a\cos b\,(1-e^y)
\le
\pi e^a
=
\frac{2|z|}{|w|}
\le
\frac{2|z|}{a}.
\end{equation*}
Consequently, using $\Re z\ge |z|/\sqrt{2}$ (since $|\vartheta|\le \pi/4$),  and $a\ge 6$, 
by~\eqref{0.52}, we obtain
\begin{align}\label{0.58}
\Re\bigl(\phi(w+y)-\phi(w)\bigr)
&\le
2\Re z\log(0.52)+\frac{2|z|}{a}
\le
\sqrt{2}\,|z|\log(0.52)+\frac{2|z|}{a}
\nonumber\\
&\le
|z|\left(\sqrt{2}\log(0.52)+\frac13\right)<-0.58|z|.
\end{align}

Since $\beta>0$ and $y<0$,
it follows by~\eqref{Id-G} and~\eqref{0.58} that 
\begin{align*}
\left|\frac{I_d}{M_\beta}\right|\le
\frac{1}{|G|}\int_{-a}^{-a/2}e^{\operatorname{\Re}(\phi(w+y)-\phi(w))+\beta y}\,dy\le 
\frac{1}{|G|}\int_{-a}^{-a/2}e^{\operatorname{\Re}(\phi(w+y)-\phi(w))}\,dy
\end{align*}
By~\eqref{low-G},
$
|G|\ge1.64T^{-1/2},
$
we obtain
\begin{align*}
\left|\frac{I_d}{M_\beta}\right|\le
\frac{a}{2|G|}e^{-0.58|z|}\le
\frac{a\sqrt{T}}{2\cdot1.64}e^{-0.58|z|}=
\frac{|z|}{3.28\sqrt{T}}e^{-0.58|z|}
<
|z|e^{-0.58|z|},
\end{align*}
where the last inequality uses $T>630$.

Since $|z|\ge4000$ and $T\le |z|/6$, the elementary inequality
$e^x>x^3/6$ for $x>0$ gives
\begin{equation*}
20|z|e^{-0.58|z|}
<e^{-|z|/2}
<\frac{1}{100T}.
\end{equation*}
Consequently,
\begin{align}\label{bound:Id}
\left|\frac{I_d}{M_\beta}\right|=\frac{
\left|
\displaystyle\int_{-a}^{-a/2}
e^{\phi(w+y)+\beta(w+y)}\,dy
\right|
}{|M_\beta|}
<
\frac{1}{100T}.
\end{align}

\smallskip
\noindent
\textbf{Part(iv): Vertical segment.}
In this part, we estimate
$$
I_v=\int_0^{ib}e^{\phi(u)+\beta u}\,du
$$

If $b=0$, the vertical contribution vanishes.
Otherwise, parametrize the segment by $u=it$, where
$t$ runs from $0$ to $b$. Then $du=i\,dt$ and
$|t|\le |b|\le\pi/4$. By the definition
$\phi(u)=2z\operatorname{Log}u-\pi e^u$, we have
\begin{equation}\label{argit}
\int_0^{ib}e^{\phi(u)+\beta u}\,du
=
i\int_0^b
\exp\!\left(
2z\operatorname{Log}(it)-\pi e^{it}+i\beta t
\right)\,dt.
\end{equation}

For $t\ne0$, the principal logarithm satisfies
\begin{equation*}
\operatorname{Log}(it)
=
\log|t|+i\arg(it),
\qquad
\arg(it)=
\begin{cases}
\pi/2, & t>0,\\
-\pi/2, & t<0.
\end{cases}
\end{equation*}
As shown above, $b$ and $\Im z$ have the same sign.
Since every nonzero $t$ on the segment has the same
sign as $b$, it follows that
\begin{equation*}
2\Im z\,\arg(it)=\pi|\Im z|.
\end{equation*}
Using $e^{it}=\cos t+i\sin t$ and
$\Re(i\beta t)=0$, we obtain
\begin{align*}
\Re\bigl(\phi(it)+i\beta t\bigr)
&=
\Re\!\left(
2z\operatorname{Log}(it)-\pi e^{it}+i\beta t
\right)
\\
&=
2\Re z\log|t|
-2\Im z\,\arg(it)
-\pi\cos t
\\
&=
2\Re z\log|t|
-\pi|\Im z|
-\pi\cos t.
\end{align*}
Consequently,
\begin{equation*}
\left|e^{\phi(it)+i\beta t}\right|
=
\exp\!\left(
2\Re z\log|t|-\pi|\Im z|-\pi\cos t
\right).
\end{equation*}

Since $\Re z>0$ and
$0<|t|\le |b|\le\pi/4<1$, we have
$2\Re z\log|t|\le0$.
Moreover, $\cos t\ge1/\sqrt{2}>0$, so all three
terms in the last exponent are nonpositive.
Thus,
\begin{equation*}
\left|e^{\phi(it)+i\beta t}\right|\le1.
\end{equation*}
As $t$ approaches $0$ along the segment, this modulus
tends to zero because $\Re z>0$.
The endpoint $u=0$ is therefore interpreted by a limit.

Finally, applying the triangle inequality and using
the length $|b|$ of the segment, we obtain
\begin{align*}
\left|
\int_0^{ib}e^{\phi(u)+\beta u}\,du
\right|
&=
\left|
i\int_0^b e^{\phi(it)+i\beta t}\,dt
\right|
\le
\int_{\min\{0,b\}}^{\max\{0,b\}}
\left|e^{\phi(it)+i\beta t}\right|\,dt
\\
&\le
\int_{\min\{0,b\}}^{\max\{0,b\}}1\,dt
=
|b|
\le\frac{\pi}{4}
<1.
\end{align*}

Using $|\arg w|\le\pi/24$, $|w|\ge6$,
$\pi e^a\le|z|/3$, and the sector condition, we obtain
\begin{equation*}
\operatorname{Re}\phi(w)
\ge
|z|\left(
\sqrt2\log6-\frac{\sqrt2\pi}{24}-\frac13
\right)
>2|z|.
\end{equation*}
Therefore,
\begin{equation*}
|M_\beta|
=
e^{\operatorname{Re}\phi(w)+\beta a}|G|
>
\frac{e^{2|z|}}{\sqrt T}
>
\frac{2|z|^2}{\sqrt T}
\ge72T^{3/2}
>100T,
\end{equation*}
where we used $\beta>0$, $|G|\ge1.64T^{-1/2}$,
$e^{2|z|}>2|z|^2$, $|z|=aT\ge6T$, and $T>630$.
Combining this with the preceding absolute bound gives
\begin{equation}\label{bound:Iv}
\frac{|I_v|}{|M_\beta|}=\frac{
\left|\displaystyle\int_0^{ib}
e^{\phi(u)+\beta u}\,du\right|
}{|M_\beta|}
<
\frac{1}{100T}.
\end{equation}

\subsubsection{Estimate of $H_2(z)$.}
Let $$
I_\theta(z)=\int_0^\infty u^{2z}e^{u/4}\sum_{k\ge2}(2\pi^2k^4e^{2u}-3\pi k^2e^u)e^{-\pi k^2e^u}\,du.
$$
Then 
$
H_2(z)=2^{1-2z}I_\theta(z).
$
We now estimate the contribution of $I_\theta(z)$
relative to the leading term $2\pi^2M_{9/4}$,
where $M_{9/4}=e^{\phi(w)+9w/4}G$.
More precisely, we will prove that
\begin{equation*}
\frac{|I_\theta(z)|}{2\pi^2|M_{9/4}|}
\le \frac{0.03}{T}.
\end{equation*}

\smallskip
\noindent
\textbf{\emph{The integrand bound.}}
We first rewrite the integrand of $I_\theta(z)$. Define
\begin{equation*}
E_\theta(u)
:=
e^{u/4}
\sum_{k\ge2}
\bigl(2\pi^2k^4e^{2u}-3\pi k^2e^u\bigr)
e^{-\pi k^2e^u},\qquad I_\theta(z)=\int_0^\infty u^{2z}E_\theta(u)\,du.
\end{equation*}
The series defining $E_\theta(u)$ converges locally uniformly
in the strip $|\operatorname{\Im}u|<\pi/2$.
Since $E_\theta$ is bounded near the origin and decays
super-exponentially as $\operatorname{\Re}u\to\infty$
in $|\operatorname{\Im}u|\le\pi/4$, the same contour
deformation used for $J_\beta$ applies here, and we use 
$$
I_G'=\int_{-1/10}^{1/10} (w+y)^{2z}E_\theta(w+y)\,dy.
$$
Similarly, we define $I_-',I_+',I_d'$.
We define 
$$
I_v'=i\int_{0}^{b} (it)^{2z}E_\theta(it)\,dt.
$$

Recall that $\phi(u)=2z\operatorname{Log}u-\pi e^u$.
For each $k\ge2$, we have
\begin{align*}
u^{2z}e^{-\pi k^2e^u}=
\exp\!\left(2z\operatorname{Log}u-\pi k^2e^u\right)=
\exp\!\left(\phi(u)-\pi(k^2-1)e^u\right).
\end{align*}
Consequently, for $u\ne0$ on the deformed contour,
\begin{align*}
u^{2z}E_\theta(u)
&=
2\pi^2e^{\phi(u)+9u/4}
\sum_{k\ge2}k^4e^{-\pi(k^2-1)e^u}
-3\pi e^{\phi(u)+5u/4}
\sum_{k\ge2}k^2e^{-\pi(k^2-1)e^u}.
\end{align*}

On the deformed contour, we have
$\Re u\ge0$ and $|\Im u|\le\pi/4$.
Set $r:=\Re(e^u)$. Then
\begin{equation*}
r=e^{\Re u}\cos(\Im u)\ge\frac1{\sqrt{2}}.
\end{equation*}
Consequently, for every $k\ge2$,
\begin{equation*}
\left|e^{-\pi(k^2-1)e^u}\right|
=
e^{-\pi(k^2-1)r}
=
e^{\pi r}e^{-\pi k^2r}.
\end{equation*}

Now we estimate $u^{2z}E_\theta(u)$.
For the second term in $E_\theta(u)$,
\begin{equation*}
\begin{aligned}
&u^{2z}e^{u/4}
\sum_{k\ge2}(-3\pi k^2e^u)e^{-\pi k^2e^u}\\
&\qquad=
-3\pi\sum_{k\ge2}k^2
\exp\!\left(
2z\operatorname{Log}u+\frac{5u}{4}-\pi k^2e^u
\right)\\
&\qquad=
-3\pi e^{\phi(u)+5u/4}
\sum_{k\ge2}k^2e^{-\pi(k^2-1)e^u}.
\end{aligned}
\end{equation*}
Hence,
\begin{equation*}
\begin{aligned}
\left|u^{2z}e^{u/4}
\sum_{k\ge2}(-3\pi k^2e^u)e^{-\pi k^2e^u}\right|&=
\left|
-3\pi e^{\phi(u)+5u/4}
\sum_{k\ge2}k^2e^{-\pi(k^2-1)e^u}
\right|\\
&\le
3\pi\left|e^{\phi(u)+5u/4}\right|
\sum_{k\ge2}k^2e^{-\pi(k^2-1)r}\\
&=
3\pi e^{\pi r}\left|e^{\phi(u)+5u/4}\right|
\sum_{k\ge2}k^2e^{-\pi k^2r}.
\end{aligned}
\end{equation*}
Finally, using
\begin{equation*}
\left|e^{\phi(u)+5u/4}\right|=\left|e^{\phi(u)+9u/4-u}\right|= |e^{-u}|\left|e^{\phi(u)+9u/4}\right|
=
e^{-\Re u}\left|e^{\phi(u)+9u/4}\right|.
\end{equation*}
we obtain
\begin{equation*}
\begin{aligned}
&\left|
-3\pi e^{\phi(u)+5u/4}
\sum_{k\ge2}k^2e^{-\pi(k^2-1)e^u}
\right|\le
e^{\pi r}\left|e^{\phi(u)+9u/4}\right|
\left(
3\pi e^{-\Re u}
\sum_{k\ge2}k^2e^{-\pi k^2r}
\right).
\end{aligned}
\end{equation*}

Applying the same argument to the first term and
combining the two bounds,
we obtain

\begin{align}\label{Etheta}
|u^{2z}E_\theta(u)|
&\le
e^{\pi r}\left|e^{\phi(u)+9u/4}\right|
\left(
2\pi^2\sum_{k\ge2}k^4e^{-\pi k^2r}
+
3\pi e^{-\Re u}\sum_{k\ge2}k^2e^{-\pi k^2r}
\right).
\end{align}
This reduces the pointwise estimate of the theta remainder
to bounds for the two positive series
\begin{equation*}
\sum_{k\ge2}k^4e^{-\pi k^2r},
\qquad
\sum_{k\ge2}k^2e^{-\pi k^2r},
\qquad r\ge\frac1{\sqrt{2}}.
\end{equation*}

We now prove that
\begin{equation*}
\sum_{k\ge2}k^4e^{-\pi k^2r}
\le17e^{-4\pi r},
\qquad
\sum_{k\ge2}k^2e^{-\pi k^2r}
\le5e^{-4\pi r}.
\end{equation*}
For $p\in\{2,4\}$, write
\begin{equation*}
\sum_{k\ge2}k^pe^{-\pi k^2r}
=
e^{-4\pi r}
\left(
2^p+\sum_{k\ge3}k^pe^{-\pi(k^2-4)r}
\right).
\end{equation*}
To bound the remaining sum, set
$t_k=k^pe^{-\pi(k^2-4)r}$ for $k\ge3$.
Since $r\ge1/\sqrt{2}$, we have
\begin{equation*}
t_3\le81e^{-5\pi/\sqrt{2}}<0.002,
\qquad
\frac{t_{k+1}}{t_k}
=
\left(\frac{k+1}{k}\right)^p
e^{-\pi(2k+1)r}
\le
\left(\frac43\right)^4e^{-7\pi/\sqrt{2}}
<10^{-6}.
\end{equation*}
Thus, comparison with a geometric series gives
\begin{equation*}
\sum_{k\ge3}t_k
\le
\frac{t_3}{1-10^{-6}}
<
\frac{0.002}{1-10^{-6}}
<1.
\end{equation*}
Consequently,
\begin{equation}\label{17-5}
\sum_{k\ge2}k^4e^{-\pi k^2r}
<17e^{-4\pi r},
\qquad
\sum_{k\ge2}k^2e^{-\pi k^2r}
<5e^{-4\pi r}.
\end{equation}

By~\eqref{Etheta} and~\eqref{17-5},
we obtain
\begin{align*}
|u^{2z}E_\theta(u)|
&\le
e^{\pi r}
\left|e^{\phi(u)+9u/4}\right|
\left(
34\pi^2+15\pi e^{-\Re u}
\right)e^{-4\pi r}
\\
&=
\left(
34\pi^2+15\pi e^{-\Re u}
\right)
e^{-3\pi r}
\left|e^{\phi(u)+9u/4}\right|
\\
&\le
(34\pi^2+15\pi)
e^{-3\pi r}
\left|e^{\phi(u)+9u/4}\right|
\\
&<
20(2\pi^2)e^{-3\pi r}
\left|e^{\phi(u)+9u/4}\right|,
\end{align*}
where we used $\Re u\ge0$ and
$34\pi^2+15\pi<20(2\pi^2)$.

Note that $$r=\Re(e^{w+y})=\Re(e^{(a+y)+ib})=e^{(a+y)} \cos b.$$
On the horizontal ray $u=w+y$, we have 
\begin{equation}\label{key-bound-H2}
\left|(w+y)^{2z}E_\theta(w+y)\right|
\le
20(2\pi^2)e^{-3\pi e^{a+y}\cos b}
\left|e^{\phi(w+y)+9(w+y)/4}\right|,
\qquad y\ge-a.
\end{equation}

\smallskip
\noindent
\textbf{\emph{The region $y\ge-a/2$: $I'_G+I'_-+I'_+$.}}
We first estimate the integrand $(w+y)^{2z}E_\theta(w+y)$ by~\eqref{key-bound-H2}.
In this region, $a+y\ge a/2$, so
\begin{equation*}
e^{a+y}\cos b
\ge\frac{e^{a/2}}{\sqrt{2}},
\end{equation*}
by the fact~\eqref{b<pi/4} that $|b|\le \pi/4$.
Consequently,
\begin{equation*}
20e^{-3\pi e^{a+y}\cos b}
\le
20\exp\!\left(
-\frac{3\pi}{\sqrt{2}}e^{a/2}
\right).
\end{equation*}
The right-hand side is independent of $y$ and can
therefore be taken outside the integral.

We next bound the integral of
$\left|e^{\phi(w+y)+9(w+y)/4}\right|$
relative to $|M_{9/4}|$.
Recall that
$X(y)=9y/4+c_3y^3+R_4(y)$ with $\beta=9/4$,
and by~\eqref{kappa-first},
\begin{equation*}
e^{\phi(w+y)+9(w+y)/4}
=
e^{\phi(w)+9w/4}
e^{-\kappa y^2/2}e^{X(y)}.
\end{equation*}
For $|y|\le 1/10$,
by~\eqref{eq:xi-saddle-curvature}, $\Re\kappa>1.4T$;
and by~\eqref{expon}, 	$e^{|X(y)|}\le1.2e^{0.091Ty^2}$.
These give
\begin{equation*}
\left|e^{-\kappa y^2/2}e^{X(y)}\right|
\le
1.2e^{-0.6Ty^2}.
\end{equation*}
Moreover, the definitions~\eqref{Gaussmain} and~\eqref{eq:G},
together with~\eqref{low-G}, give
\begin{equation*}
|M_{9/4}|
=
e^{\Re\phi(w)+9a/4}|G|,
\qquad
|G|\ge1.64T^{-1/2}.
\end{equation*}
This yields
\begin{equation*}
\frac{
\left|e^{\phi(w+y)+9(w+y)/4}\right|
}{
|M_{9/4}|
}
\le
\frac{1.2}{|G|}e^{-0.6Ty^2}.
\end{equation*}

Note that by~\eqref{bound:I-I+}, 
\begin{align*}
\frac{|I_-|+|I_+|}{|M_{9/4}|}&\le\frac{1}{|M_{9/4}|}
\left(
\int_{-a/2}^{-1/10}
\left|e^{\phi(w+y)+9(w+y)/4}\right|\,dy
+
\int_{1/10}^{\infty}
\left|e^{\phi(w+y)+9(w+y)/4}\right|\,dy
\right)\\
&\le
\frac{12.6e^{-0.0063T}}{\sqrt{T}}
<
\frac{9}{T}.
\end{align*}
Indeed, the function $T\mapsto \sqrt{T}e^{-0.0063T}$
is decreasing for $T\ge630$, and
$
12.6\sqrt{630}\,e^{-0.0063\cdot630}<9.
$

Combining these bounds, we obtain
\begin{align*}
\frac{
\displaystyle\int_{-a/2}^{\infty}
\left|e^{\phi(w+y)+9(w+y)/4}\right|\,dy
}{
|M_{9/4}|
}
&\le
\frac{1.2}{|G|}
\int_{-1/10}^{1/10}e^{-0.6Ty^2}\,dy
+\frac{9}{T}
\\
&\le
\frac{1.2}{|G|}
\int_{-\infty}^{\infty}e^{-0.6Ty^2}\,dy
+\frac{9}{T}
\\
&=
\frac{1.2}{|G|}
\sqrt{\frac{\pi}{0.6T}}
+\frac{9}{T}
\\
&\le
\frac{1.2}{1.64}
\sqrt{\frac{\pi}{0.6}}
+\frac{9}{T}
<2,
\end{align*}
where the last line uses~\eqref{low-G}, and the final strict
inequality follows from $T>630$.

By the triangle
inequality and the pointwise comparison above,
\begin{align*}
&\frac{
\left|
\displaystyle\int_{-a/2}^{\infty}
(w+y)^{2z}E_\theta(w+y)\,dy
\right|
}{
2\pi^2|M_{9/4}|
}
\le
\frac{1}{2\pi^2|M_{9/4}|}
\int_{-a/2}^{\infty}
\left|(w+y)^{2z}E_\theta(w+y)\right|\,dy
\\
&\quad\le
20\exp\!\left(
-\frac{3\pi}{\sqrt{2}}e^{a/2}
\right)
\frac{
\displaystyle\int_{-a/2}^{\infty}
\left|e^{\phi(w+y)+9(w+y)/4}\right|\,dy
}{
|M_{9/4}|
}
<
40\exp\!\left(
-\frac{3\pi}{\sqrt{2}}e^{a/2}
\right).
\end{align*}

Finally, we verify that this bound is less than
$0.01/T$. Using $T<1.586e^a$ and putting
$v=e^{a/2}\ge e^3>20$, we have
\begin{align*}
40T\exp\!\left(
-\frac{3\pi}{\sqrt{2}}e^{a/2}
\right)
&<
40(1.586)v^2
e^{-(3\pi/\sqrt{2})v}<
64v^2e^{-6.6v}.
\end{align*}
The function $v^2e^{-6.6v}$ is decreasing for $v\ge20$,
and
\begin{equation*}
64v^2e^{-6.6v}
\le
64\cdot20^2e^{-132}
<0.01.
\end{equation*}
The last inequality follows, for example, from
$e^{132}>132^4/24$.
Therefore,
\begin{equation}\label{H_21}
\frac{
\left|
\displaystyle\int_{-a/2}^{\infty}
(w+y)^{2z}E_\theta(w+y)\,dy
\right|
}{
2\pi^2|M_{9/4}|
}
<
\frac{0.01}{T}.
\end{equation}

\smallskip
\noindent
\textbf{\emph{The distant left $I'_d$ and vertical segment $I'_v$.}}
We first consider the remaining horizontal interval
\begin{equation*}
u=w+y,\qquad -a\le y\le-a/2.
\end{equation*}
On this interval, we use the weaker bound
$
20e^{-3\pi e^{a+y}\cos b}\le20.
$
By~\eqref{key-bound-H2},
\begin{equation*}
\left|(w+y)^{2z}E_\theta(w+y)\right|
\le
20(2\pi^2)
\left|e^{\phi(w+y)+9(w+y)/4}\right|.
\end{equation*}
Applying the triangle inequality and the distance-left
absolute-integral estimate from Part (iii) of Section~\ref{subsec:EstH1} with $\beta=9/4$,
we obtain
\begin{align*}
\frac{
\left|
\displaystyle\int_{-a}^{-a/2}
(w+y)^{2z}E_\theta(w+y)\,dy
\right|
}{
2\pi^2|M_{9/4}|
}&\le
\frac{1}{2\pi^2|M_{9/4}|}
\int_{-a}^{-a/2}
\left|(w+y)^{2z}E_\theta(w+y)\right|\,dy
\\
&\le
\frac{20}{|M_{9/4}|}
\int_{-a}^{-a/2}
\left|e^{\phi(w+y)+9(w+y)/4}\right|\,dy
\end{align*}
Recall that by~\eqref{eq:aT}, $T=|z|/a>630$ and $a\ge6$,
so that $T\le |z|/6$.
For $-a\le y\le-a/2$, the estimate \eqref{0.58} gives
$\Re\bigl(\phi(w+y)-\phi(w)\bigr)\le-0.58 |z|$.
Consequently,
\begin{align*}
\frac{20}{|M_{9/4}|}
\int_{-a}^{-a/2}
\left|e^{\phi(w+y)+9(w+y)/4}\right|\,dy
&=
\frac{20}{|G|}
\int_{-a}^{-a/2}
\exp\left(
\Re\bigl(\phi(w+y)-\phi(w)\bigr)+\frac{9y}{4}
\right)\,dy
\\
&\le
\frac{20}{|G|}
\int_{-a}^{-a/2}e^{-0.58 |z|}\,dy=
\frac{10a}{|G|}e^{-0.58 |z|}\\
&\le
\frac{10a\sqrt{T}}{1.64}e^{-0.58|z|}=
\frac{20|z|}{3.28\sqrt{T}}e^{-0.58|z|}
<
20|z|e^{-0.58|z|},
\end{align*}
where we used $|G|\ge1.64T^{-1/2}$ from~\eqref{low-G}.

We use the elementary
inequality $e^x>x^3/6$ for $x>0$.
Since $|z|\ge4000$, it follows that
$
e^{0.08|z|}>
\frac{(0.08|z|)^3}{6}
>
20|z|$,
and hence
\begin{equation*}
20|z|e^{-0.58|z|}<e^{-|z|/2}.
\end{equation*}
Similarly,
$
e^{|z|/2}>
\frac{|z|^3}{48}
>
\frac{100|z|}{6}
\ge100T,
$
which implies
\begin{equation*}
e^{-|z|/2}<\frac{1}{100T}=\frac{0.01}{T}.
\end{equation*}
Combining these estimates, we obtain
\begin{equation}\label{H_22}
\frac{
\left|
\displaystyle\int_{-a}^{-a/2}
(w+y)^{2z}E_\theta(w+y)\,dy
\right|
}{
2\pi^2|M_{9/4}|
}
<
20|z|e^{-0.58|z|}
<
e^{-|z|/2}
<
\frac{0.01}{T}.
\end{equation}

We next estimate the vertical segment
$$I_v'=i\int_{0}^{b} (it)^{2z}E_\theta(it)\,dt.$$
If $b=0$, then $I'_v=0$.
Otherwise, write $u=it$, where $t$ runs from $0$ to $b$.
Since $|t|\le|b|\le\pi/4$, we have
\begin{equation*}
\Re(e^{it})=\cos t\ge\frac1{\sqrt{2}}.
\end{equation*}
By the definition of $E_\theta$,
\begin{equation*}
E_\theta(it)
=
e^{it/4}
\sum_{k\ge2}
\left(2\pi^2k^4e^{2it}-3\pi k^2e^{it}\right)
e^{-\pi k^2e^{it}}.
\end{equation*}
Since $t$ is real, we have
$|e^{it/4}|=1$
and
\begin{equation*}
\left|e^{-\pi k^2e^{it}}\right|
=
\left|
e^{-\pi k^2\cos t}e^{-i\pi k^2\sin t}
\right|
=
e^{-\pi k^2\cos t}.
\end{equation*}
Moreover, the triangle inequality gives
\begin{equation*}
\left|2\pi^2k^4e^{2it}-3\pi k^2e^{it}\right|
\le
2\pi^2k^4+3\pi k^2.
\end{equation*}
Taking absolute values term by term and applying~\eqref{17-5}
with $\cos t\ge1/\sqrt{2}$, we obtain
\begin{align*}
|E_\theta(it)|
&\le
\sum_{k\ge2}
(2\pi^2k^4+3\pi k^2)e^{-\pi k^2\cos t}
\\
&\le
\sum_{k\ge1}
(2\pi^2k^4+3\pi k^2)e^{-\pi k^2\cos t}
\\
&=
(2\pi^2+3\pi)e^{-\pi\cos t}
+
2\pi^2\sum_{k\ge2}k^4e^{-\pi k^2\cos t}
+
3\pi\sum_{k\ge2}k^2e^{-\pi k^2\cos t}
\\
&\le
(2\pi^2+3\pi)e^{-\pi\cos t}
+
(2\pi^2\cdot17+3\pi\cdot5)e^{-4\pi\cos t}
\\
&\le
(2\pi^2+3\pi)e^{-\pi/\sqrt{2}}
+
(34\pi^2+15\pi)e^{-4\pi/\sqrt{2}}
<4.
\end{align*}

We also bound the factor $(it)^{2z}$.
By~\eqref{claim:proof 3.1},
$\arg z=b+\arctan(b/a)$,
so $b$ and $\Im z$ have the same sign. By~\eqref{argit},
\begin{equation*}
\operatorname{Log}(it)
=
\log|t|+i\arg(it),
\qquad
\arg(it)=\frac{\pi}{2}\operatorname{sgn}(t),
\end{equation*}
and therefore
$2\Im z\,\arg(it)=\pi|\Im z|$.
It follows that
\begin{align*}
|(it)^{2z}|
=
\left|\exp\!\left(2z\operatorname{Log}(it)\right)\right|
=
\exp\!\left(
2\Re z\log|t|-2\Im z\,\arg(it)
\right)
=
|t|^{2\Re z}e^{-\pi|\Im z|}
\le1,
\end{align*}
where we used $\Re z>0$ and $0<|t|\le\pi/4<1$.
As $t\to0$ along the segment, this modulus tends to zero,
so the endpoint $t=0$ is interpreted by a limit.
Consequently,
\begin{align*}
\left|
\int_0^{ib}u^{2z}E_\theta(u)\,du
\right|
&=
\left|
i\int_0^b(it)^{2z}E_\theta(it)\,dt
\right|
\le
\int_{\min\{0,b\}}^{\max\{0,b\}}
|(it)^{2z}|\,|E_\theta(it)|\,dt
\le
4|b|
\le\pi<4.
\end{align*}

Recall that
\begin{equation*}
|M_{9/4}|
=
e^{\Re\phi(w)+9a/4}|G|,
\qquad
|G|\ge1.64T^{-1/2},
\end{equation*}
and  $\Re\phi(w)>2|z|$.
It follows that 
\begin{align}\label{H_23}
\frac{
\left|
\displaystyle\int_0^{ib}
u^{2z}E_\theta(u)\,du
\right|
}{
2\pi^2|M_{9/4}|
}
\le
\frac{4}{
2\pi^2e^{\Re\phi(w)+9a/4}|G|
}
\le
\frac{4}{2\pi^2}
\frac{\sqrt{T}}{1.64}
e^{-2|z|}
<
\frac{0.01}{T}.
\end{align}
The last inequality follows from
$|z|\ge4000$, $T\le |z|/6$.

By the contour deformation,
\begin{align*}
I_\theta(z)
&=
\int_0^{ib}u^{2z}E_\theta(u)\,du
+
\int_{-a}^{-a/2}
(w+y)^{2z}E_\theta(w+y)\,dy
+
\int_{-a/2}^{\infty}
(w+y)^{2z}E_\theta(w+y)\,dy.
\end{align*}
Combining~\eqref{H_21},~\eqref{H_22} and~\eqref{H_23}, we conclude that
\begin{equation}\label{H_2123}
\frac{|I_\theta(z)|}{2\pi^2|M_{9/4}|}
\le
\frac{0.01}{T}
+
\frac{0.01}{T}
+
\frac{0.01}{T}
=
\frac{0.03}{T}.
\end{equation}

\subsubsection{Proof of Lemma~\ref{lem:xi-explicit-saddle}}

\begin{proof}
We now combine the estimates from Sections 3.1.1 and 3.1.2 to prove~\eqref{eq:xi-saddle-error}.
Fix $\beta\in\{5/4,9/4\}$.
By the contour decomposition~\eqref{Jbeta},
\begin{equation*}
J_\beta(z)=I_G+I_-+I_++I_d+I_v.
\end{equation*}
The Gaussian estimates
\eqref{eq:xi-gaussian-truncation} and~\eqref{epsilonG},
together with the local error bound
\eqref{eq:xi-saddle-local-error}, give
\begin{equation*}
\left|\frac{I_G}{M_\beta}-1\right|
\le
\frac{8.8e^{-0.007T}}{\sqrt T}
+\frac{14.4}{T}
+\frac{1.3}{T^2}.
\end{equation*}
Combining this with the horizontal-tail estimate
\eqref{bound:I-I+}, the distant-left estimate
\eqref{bound:Id}, and the vertical estimate
\eqref{bound:Iv}, we obtain
\begin{align*}
\left|\frac{J_\beta(z)}{M_\beta}-1\right|
&\le
\left|\frac{I_G}{M_\beta}-1\right|
+
\frac{|I_-|+|I_+|+|I_d|+|I_v|}{|M_\beta|}
\\
&\le
\frac{14.4}{T}
+\frac{1.3}{T^2}
+
\frac{8.8e^{-0.007T}+12.6e^{-0.0063T}}{\sqrt T}
+\frac{2}{100T}.
\end{align*}

For $c\in\{0.007,0.0063\}$,
\begin{equation*}
\frac{d}{dT}\left(\sqrt T\,e^{-cT}\right)
=
\sqrt T\,e^{-cT}
\left(\frac{1}{2T}-c\right)<0
\qquad(T\ge630).
\end{equation*}
Consequently,
\begin{equation*}
\begin{aligned}
\sqrt T\left(8.8e^{-0.007T}+12.6e^{-0.0063T}\right)\le
\sqrt{630}
\left(8.8e^{-4.41}+12.6e^{-3.969}\right)
<9.
\end{aligned}
\end{equation*}
It follows that
\begin{equation}\label{eq:xi-saddle-Jbeta-error}
\left|\frac{J_\beta(z)}{M_\beta}-1\right|
\le
\frac{14.4}{T}
+\frac{1.3}{T^2}
+\frac{9}{T}
+\frac{2}{100T}
<
\frac{24}{T},
\qquad
\beta\in\left\{\frac54,\frac94\right\},
\end{equation}
where the last inequality uses $T>630$.

We next identify the Gaussian main term with $H_*(z)$.
The saddle relation $we^w=2z/\pi$ implies 
\begin{equation*}
\phi(w)
=
2z\operatorname{Log}w-\pi e^w
=
2z\left(\operatorname{Log}w-\frac1w\right).
\end{equation*}
Moreover,
\begin{equation*}
\kappa=-\phi''(w)
=
\pi e^w\frac{1+w}{w}.
\end{equation*}
Thus, by~\eqref{eq:G},
\begin{equation*}
G^2=\frac{2\pi}{\kappa}
=e^{-w}\frac{2w}{1+w}.
\end{equation*}
Indeed, by~\eqref{b<pi/4} and
\eqref{eq:xi-saddle-curvature}, we have
$\Re w>0$ and $\Re\kappa>0$ throughout the sector, while
\begin{equation*}
\Re\frac{w}{1+w}
=
\frac{a(a+1)+b^2}{(a+1)^2+b^2}>0.
\end{equation*}
Hence both $G=\sqrt{2\pi}\,\kappa^{-1/2}$ and
\begin{equation*}
e^{-w/2}\sqrt{\frac{2w}{1+w}}
\end{equation*}
are holomorphic on a connected open neighborhood of the
sector, with the branches fixed by their positive values
on the positive real axis. Since they have the same square
and agree there, analytic continuation gives
\begin{equation*}
G=e^{-w/2}\sqrt{\frac{2w}{1+w}}.
\end{equation*}
By~\eqref{Gaussmain}, we obtain
\begin{align*}
2^{1-2z}2\pi^2M_{9/4}
&=
2^{1-2z}2\pi^2e^{\phi(w)+9w/4}G
\\
&=
4\pi^2e^{7w/4}\sqrt{\frac{2w}{1+w}}\,
\exp\!\left(\phi(w)-2z\log2\right)
\\
&=
4\pi^2e^{7w/4}\sqrt{\frac{2w}{1+w}}\,
\exp\!\left(
2z\left(\operatorname{Log}\frac w2-\frac1w\right)
\right)=H_*(z).
\end{align*}
Here we used
$\operatorname{Log}(w/2)=\operatorname{Log}w-\log2$,
which holds because $\Re w>0$.
The same representation shows that $H_*$ is holomorphic
and nonzero on an open neighborhood of the sector.

We now estimate the lower-order contribution involving
$J_{5/4}$.  Equation~\eqref{Gaussmain} gives
\begin{equation*}
\frac{M_{5/4}}{M_{9/4}}=e^{-w}.
\end{equation*}
By~\eqref{eq:xi-saddle-parameters},
$
e^{-a}={\pi|w|}/({2aT}).
$
Hence~\eqref{eq:xi-saddle-Jbeta-error} implies
\begin{align*}
\frac{|3\pi J_{5/4}(z)|}{2\pi^2|M_{9/4}|}
&=
\frac{3}{2\pi}
\left|\frac{J_{5/4}(z)}{M_{5/4}}\right|
\left|\frac{M_{5/4}}{M_{9/4}}\right|
\\
&\le
\frac{3}{2\pi}
\left(1+\frac{24}{T}\right)e^{-a}
\\
&=
\frac{3}{4T}\frac{|w|}{a}
\left(1+\frac{24}{T}\right)
\\
&<
\frac{3\cdot1.009}{4T}
\left(1+\frac{24}{630}\right)
<
\frac{0.8}{T},
\end{align*}
where we used~\eqref{1009} and $T>630$.

By~\eqref{H1-J} and the definition of $I_\theta$, the exact
decomposition of $H$ is
\begin{equation*}
H(z)
=
2^{1-2z}
\left(
2\pi^2J_{9/4}(z)
-3\pi J_{5/4}(z)
+I_\theta(z)
\right).
\end{equation*}
Since $H_*(z)=2^{1-2z}2\pi^2M_{9/4}$, it follows that
\begin{equation*}
\frac{H(z)}{H_*(z)}-1
=
\left(\frac{J_{9/4}(z)}{M_{9/4}}-1\right)
-
\frac{3\pi J_{5/4}(z)}{2\pi^2M_{9/4}}
+
\frac{I_\theta(z)}{2\pi^2M_{9/4}}.
\end{equation*}
Using
\eqref{eq:xi-saddle-Jbeta-error},
the preceding bound for $J_{5/4}$, and~\eqref{H_2123},
we obtain
\begin{align*}
\left|\frac{H(z)}{H_*(z)}-1\right|
&\le
\left|\frac{J_{9/4}(z)}{M_{9/4}}-1\right|
+
\frac{|3\pi J_{5/4}(z)|}{2\pi^2|M_{9/4}|}
+
\frac{|I_\theta(z)|}{2\pi^2|M_{9/4}|}
\\
&<
\frac{24}{T}
+\frac{0.8}{T}
+\frac{0.03}{T}
<
\frac{30}{T}.
\end{align*}
This proves the first inequality in~\eqref{eq:xi-saddle-error}.

Since $T>630$, the preceding estimate gives
\begin{equation*}
\left|\frac{H(z)}{H_*(z)}-1\right|<\frac{30}{T}<\frac1{21}<1.
\end{equation*}
Therefore,
in the sector~\eqref{sector},
$$H(z)=H_*(z)\left(1+\frac{H(z)}{H_*(z)}-1\right)\ne 0.$$ 

Since
$$
\Re\frac{H(z)}{H_*(z)}=\Re\left(1+\frac{H(z)}{H_*(z)}-1\right)=1+\Re\left(\frac{H(z)}{H_*(z)}-1\right)\ge 1-\left|\frac{H(z)}{H_*(z)}-1\right|.
$$
we obtain
\begin{equation*}
\Re\frac{H(z)}{H_*(z)}
>
1-\frac{30}{T}
>
\frac{20}{21}>0.
\end{equation*}
By continuity, this inequality holds on an open neighborhood
of the sector~\eqref{sector} and consequently,
$\operatorname{Log}\!\bigl(H(z)/H_*(z)\bigr)$ is holomorphic. For real $z\ge4000$, both $H(z)$ and $H_*(z)$ are
positive, so
\begin{equation*}
\operatorname{Log}\frac{H(z)}{H_*(z)}
=
\log\frac{H(z)}{H_*(z)}.
\end{equation*}

Using the power series for $\operatorname{Log}(1+v)$,  with $v(z)={H(z)}/{H_*(z)}-1,$
we conclude that
\begin{align*}
\left|\operatorname{Log}\frac{H(z)}{H_*(z)}\right|
&\le
\sum_{m\ge1}\frac{|v(z)|^m}{m}
\le
\frac{|v(z)|}{1-|v(z)|}
\le
\frac{30/T}{1-30/630}
=
\frac{31.5}{T}
<
\frac{32}{T}
<
0.051.
\end{align*}
This proves~\eqref{eq:xi-saddle-error} and completes
the proof.
\end{proof}

\subsection{Estimate of the initial logarithmic ratios $g_0(z)$}
In this subsection, we establish uniform estimates for the initial logarithmic ratio $g_0$.
Recall that $$\Lambda(z)=\frac{H(z)}{\Gamma(2z+1)}.$$
Define
\begin{equation}\label{eq:g0(z)}
g_0(z)
=
-\Delta^2\Log\Lambda(z)
=
2\Log\Lambda(z)-\Log\Lambda(z-1)-\Log\Lambda(z+1).
\end{equation}
Here $\Log\Lambda$ denotes a holomorphic logarithm of
$\Lambda$ on the disks considered below, normalized to
agree with the real logarithm on their real intervals.
The existence of this branch will be established below.

\begin{lemma}
\label{lem:explicit_analytic_initial_ratio}
Let $n\ge12000$ be real, and let
\begin{equation*}
D_0=\{z\in\mathbb C:|z-n|\le n/2\}.
\end{equation*}
Suppose that $g_0$ is the function defined in~\eqref{eq:g0(z)}, $w=W(2z/\pi)$,
and $$Q(z)=\frac{2w}{z(1+w)}.$$ Then
\begin{itemize}
\item[\rm(i)]  $g_0$ is holomorphic on a neighborhood of $D_0;$
\item[\rm(ii)] $g_0$ is real-valued on $[n/2,3n/2];$
\item[\rm(iii)] for every integer $k\in D_0$,
\begin{equation*}
g_0(k)
=
\log\frac{\lambda_k^2}{\lambda_{k-1}\lambda_{k+1}}.
\end{equation*}
\item[\rm(iv)] for every $z\in D_0$,
\begin{equation}\label{eq:g0-approximation}
|g_0(z)-Q(z)|\le\frac{20}{n^2},\qquad
\Re\bigl(ng_0(z)\bigr)
>\frac58,
\qquad
|ng_0(z)|<\frac92.
\end{equation}
Moreover,
\begin{equation}\label{54}
ng_0(n) >\frac54,
\end{equation}
\end{itemize}
\end{lemma}

\begin{proof}
\smallskip
\noindent
\textbf{\emph{{Proof of~(i--iii)}}}
Consider the open disk
\begin{equation*}
\widetilde D_0=\left\{z:|z-n|<\frac{2n}{3}\right\},
\end{equation*}
a neighborhood of $D_0$.
For $z\in \widetilde D_0$, we have 
$
|z|\ge n-|z-n|>\frac n3\ge 4000.
$

Since $\Re z=n+\Re(z-n)$ and
$\Im z=\Im(z-n)$,
by the Cauchy--Schwarz inequality,
\begin{align*}
|\Re(z-n)|+|\Im z|
&=
|\Re(z-n)|+|\Im(z-n)|\\
&\le
\sqrt{1^2+1^2}\,
\sqrt{(\Re(z-n))^2+(\Im(z-n))^2}=
\sqrt2\,|z-n|.
\end{align*}
Consequently, 
\begin{align*}
\Re z-|\Im z|
&=
n+\Re(z-n)-|\Im z|\ge
n-|\Re(z-n)|-|\Im z|\\
&\ge
n-\sqrt2\,|z-n|
>
n\left(1-\frac{2\sqrt2}{3}\right)>0,
\end{align*}
where the last inequality follows from  $|z-n|<2n/3$ and $2\sqrt2<3$.
Thus
\begin{equation*}
\Re z>|\Im z|\ge0,\qquad 
|\arg z|
=
\arctan\frac{|\Im z|}{\Re z}
<
\frac{\pi}{4}.
\end{equation*}
By Lemma~\ref{lem:xi-explicit-saddle}, $H(z)$ has no zeros in $\widetilde D_0$.

Since $\Gamma(2z+1)$ is holomorphic and nonzero
when $\Re z>0$,
it follows that $\Lambda(z)={H(z)}/\left({\Gamma(2z+1)}\right)$ is also holomorphic and nonzero
on $\widetilde D_0$.
The disk $\widetilde D_0$ is simply connected, so $H$, $\Gamma(2z+1)$,
and $\Lambda$ admit holomorphic logarithms on $\widetilde D_0$.

By Lemma~\ref{lem:theta_moments_and_interpolation},
$\Lambda(x)>0$ for real $x>-1/2$.
Since $\Gamma(2x+1)>0$ for $x>0$, we have
$H(x)>0$ for $x>0$.
We choose the logarithmic branches that agree with
the real logarithms on $(n/3,5n/3)$.
These choices give
\begin{equation}\label{eq:g0-log-Lambda}
\Log\Lambda(z)=\Log H(z)-\Log\Gamma(2z+1).
\end{equation}

For every $z\in D_0$, we have
\begin{equation*}
|(z\pm1)-n|
\le |z-n|+1
\le \frac n2+1
<\frac{2n}{3}.
\end{equation*}
Thus $\Log\Lambda(z)$ and $\Log\Lambda(z\pm1)$
are holomorphic on $\left\{z:|z-n|<\frac{2n}{3}-1\right\}$, and hence so is
\begin{equation*}
g_0(z)
=
2\Log\Lambda(z)-\Log\Lambda(z-1)-\Log\Lambda(z+1).
\end{equation*}
This proves (i).

For $x\in[n/2,3n/2]$, Lemma~\ref{lem:theta_moments_and_interpolation} gives
$
\Lambda(x), \Lambda(x\pm1)>0.$
Hence, 
$$g_0(x)=
2\log\Lambda(x)-\log\Lambda(x-1)-\log\Lambda(x+1)$$ is real-valued on $[n/2,3n/2]$.
More precisely, for $k\in[n/2,3n/2]\cap \mathbb{Z},$
\begin{align*}
g_0(k)
&=
2\log\Lambda(k)-\log\Lambda(k-1)-\log\Lambda(k+1)\\
&=
2\log\lambda_k-\log\lambda_{k-1}-\log\lambda_{k+1}=
\log\frac{\lambda_k^2}{\lambda_{k-1}\lambda_{k+1}}.
\end{align*}
This proves (ii) and (iii).

\smallskip
\noindent
\textbf{\emph{Proof of~(iv)}}
We now prove~\eqref{eq:g0-approximation} and~\eqref{54}.
We first show that $Q(z)$ is the main term of $g_0(z)$.

\textbf{\emph{The main term of $g_0$.}}
Recall that $w=W(2z/\pi)$.
On the disk $\widetilde D_0=\{z:|z-n|<2n/3\}$,
by~\eqref{eq:aT}, we have $\Re w\ge6$.
Therefore,
\begin{equation}\label{eq:g0-w-bounds}
|1+w|\ge7,
\qquad
\left|\frac{w}{1+w}\right|<1,
\end{equation}
where the second inequality follows from
$
|1+w|^2-|w|^2=1+2\Re w>0.
$
We also have
\begin{equation*}
\Re\frac{w}{1+w}
=
\frac{|w|^2+\Re w}{|1+w|^2}>0.
\end{equation*}

Recall that in
Lemma~\ref{lem:xi-explicit-saddle},
\begin{equation*}
H_*(z)=4\pi^2e^{7w/4}\sqrt{\frac{2w}{1+w}}
\exp\!\left(2z\left(\operatorname{Log}\frac w2-\frac1w\right)\right).
\end{equation*}
Define
\begin{equation*}
p(z)
=
2z\left(\operatorname{Log}\frac w2-\frac1w\right)\quad \textrm{and}\quad
h_0(z)
=
\log(4\pi^2)+\frac{7w}{4}
+\frac12\operatorname{Log}\frac{2w}{1+w}.
\end{equation*}
Then
$\Log H_*(z)=p(z)+h_0(z)$.
Let
\begin{equation*}
\varepsilon(z)
=
\operatorname{Log}\frac{H(z)}{H_*(z)}.
\end{equation*}
By~\eqref{eq:xi-saddle-error},
$\varepsilon$ is holomorphic on $\widetilde D_0$ and
satisfies
\begin{equation}\label{eq:g0-epsilon-bound}
|\varepsilon(z)|<0.051.
\end{equation}
For $x\in(n/3,5n/3)$, we have
\begin{equation*}
\log H(x)
=
\log H_*(x)+\varepsilon(x)
=
p(x)+h_0(x)+\varepsilon(x).
\end{equation*}
Since both sides are holomorphic on $\widetilde D_0$,
the identity theorem yields
\begin{equation}\label{eq:g0-log-H}
\Log H(z)=p(z)+h_0(z)+\varepsilon(z),
\qquad z\in\widetilde D_0.
\end{equation}

Differentiating $we^w=2z/\pi$, we obtain
\begin{equation}\label{eq:g0-w-derivative}
w'(z)=\frac{w}{z(1+w)}.
\end{equation}
It follows that
\begin{equation}\label{eq:g0-p-second}
p'(z)=
2\left(\operatorname{Log}\frac w2-\frac1w\right)
+
2zw'\left(\frac1w+\frac1{w^2}\right)=2\operatorname{Log}\frac w2\quad \textrm{and} \qquad
p''(z)=\frac{2w'}w=\frac2{z(1+w)}.
\end{equation}
In particular,
\begin{equation}\label{QZ}
\frac2z-p''(z)
=
\frac2z-\frac2{z(1+w)}
=
\frac{2w}{z(1+w)}
=
Q(z).
\end{equation}

By~\eqref{eq:g0-log-Lambda} and~\eqref{eq:g0-log-H},
\begin{equation*}
\Log\Lambda(z)
=
p(z)+h_0(z)+\varepsilon(z)-\Log\Gamma(2z+1).
\end{equation*}
Substituting into the definition of $g_0$, we obtain
\begin{align*}
g_0(z)
=&
2\Log\Lambda(z)-\Log\Lambda(z-1)-\Log\Lambda(z+1)\\
=&
-\Delta^2p(z)
-\Delta^2h_0(z)
-\Delta^2\varepsilon(z)
+\Delta^2\Log\Gamma(2z+1).
\end{align*}
Hence,
\begin{align}\label{g-Q}
g_0(z)-Q(z)
={}&
-\Delta^2p(z)
-\Delta^2h_0(z)
-\Delta^2\varepsilon(z)
+\Delta^2\Log\Gamma(2z+1)-\frac2z+p''(z)\nonumber\\
={}&
-\bigl(\Delta^2p(z)-p''(z)\bigr)
-\Delta^2h_0(z)
-\Delta^2\varepsilon(z)+
\left(
\Delta^2\Log\Gamma(2z+1)-\frac2z
\right).
\end{align}

We estimate these four terms
$$
\bigl(\Delta^2p(z)-p''(z)\bigr),\qquad
\Delta^2h_0(z),\qquad
\Delta^2\varepsilon(z),\qquad
\Delta^2\Log\Gamma(2z+1)-\frac2z
$$ separately.
For a function $f$ holomorphic on $\widetilde D_0$,
Taylor's formula with integral remainder gives
\begin{align}
\Delta^2f(z)
&=
\int_0^1(1-s)
\bigl(f''(z+s)+f''(z-s)\bigr)\,ds
\nonumber\\
&=
\int_{-1}^1(1-|s|)f''(z+s)\,ds,
\qquad z\in D_0.
\label{eq:g0-centered-difference}
\end{align}
We shall also use
$
\int_{-1}^1(1-|s|)\,ds=1.
$

\smallskip
\noindent
\textbf{\emph{Estimate of $\Delta^2p-p''$.}}
By~\eqref{eq:g0-p-second}
we have $p''(z)=2/({z(1+w)})$, where $w=W(2z/\pi)$.
Then 
\begin{equation*}
p'''(z)
=
-\frac2{z^2}
\left(
\frac1{1+w}+\frac{w}{(1+w)^3}
\right)\quad \textrm{and}\quad
p^{(4)}(z)
=
\frac1{z^3}
\left(
\frac4{1+w}
+\frac{4w}{(1+w)^3}
+\frac{2w^2(w+4)}{(1+w)^5}
\right).
\end{equation*}
Using~\eqref{eq:g0-w-bounds} and
$|w+4|\le|w|+4$, we obtain
\begin{align*}
|p^{(4)}(z)|
&\le
\frac1{|z|^3}
\left(
\frac47+\frac4{49}
+\frac2{49}+\frac8{343}
\right)=
\frac{246}{343|z|^3}
<
\frac1{|z|^3}.
\end{align*}

By Taylor's formula with integral remainder, we have
\begin{align*}
p(z+1)
&=p(z)+p'(z)+\frac12p''(z)+\frac16p'''(z)
+\frac16\int_0^1(1-s)^3p^{(4)}(z+s)\,ds,\\
p(z-1)
&=p(z)-p'(z)+\frac12p''(z)-\frac16p'''(z)
+\frac16\int_0^1(1-s)^3p^{(4)}(z-s)\,ds.
\end{align*}
Adding these identities, 
we obtain 
\begin{equation*}
p(z+1)+p(z-1)
=
2p(z)+p''(z)
+\frac16\int_0^1(1-s)^3
\bigl(p^{(4)}(z+s)+p^{(4)}(z-s)\bigr)\,ds.
\end{equation*}
Since $\Delta^2p(z)=p(z-1)-2p(z)+p(z+1)$, it follows that
\begin{equation*}
\Delta^2p(z)-p''(z)
=
\frac16\int_0^1(1-s)^3
\bigl(p^{(4)}(z+s)+p^{(4)}(z-s)\bigr)\,ds.
\end{equation*}
Hence, for $z\in D_0$,
\begin{align}
|\Delta^2p(z)-p''(z)|
&\le
\frac{2}{6(n/2-1)^3}
\int_0^1(1-s)^3\,ds=
\frac1{12(n/2-1)^3}
\le
\frac9{4n^3}
<
\frac1{n^2}.
\label{eq:g0-p-difference}
\end{align}
Here we used $n/2-1\ge n/3$ and $n\ge12000$.

\smallskip
\noindent
\textbf{\emph{Estimate of $\Delta^2h_0$.}}
Recall that $$
h_0(z)
=
\log(4\pi^2)+\frac{7w}{4}
+\frac12\operatorname{Log}\frac{2w}{1+w}.$$
Using~\eqref{eq:g0-w-derivative}, $w'(z)={w}/{(z(1+w))}$,
differentiating twice gives
\begin{align*}
h_0'(z)
&=
\frac74w'
+\frac12
\left(\frac1w-\frac1{1+w}\right)w'=
\frac1z
\left(
\frac{7w}{4(1+w)}
+\frac1{2(1+w)^2}
\right);\\
h_0''(z)
&=
\frac1{z^2}
\left(
-\frac{7w}{4(1+w)}
-\frac1{2(1+w)^2}
+\frac{7w}{4(1+w)^3}
-\frac{w}{(1+w)^4}
\right).
\end{align*}
By~\eqref{eq:g0-w-bounds}, $|1+w|\ge7$, and
$\left|{w}/{(1+w)}\right|<1$, we have 
\begin{align*}
|h_0''(z)|
&\le
\frac1{|z|^2}
\left(
\frac74+\frac1{98}+\frac1{28}+\frac1{343}
\right)=
\frac{617}{343|z|^2}
<
\frac{1.8}{|z|^2}.
\end{align*}
For $z\in D_0$ and $-1\le s\le1$,
\begin{equation*}
|z+s|\ge |z|-|s|\ge\frac n2-1.
\end{equation*}
Therefore, by~\eqref{eq:g0-centered-difference},
\begin{align}
|\Delta^2h_0(z)|
&\le
\int_{-1}^1(1-|s|)\,|h_0''(z+s)|\,ds<
\frac{1.8}{(n/2-1)^2}
<
\frac8{n^2},
\label{eq:g0-h0-difference}
\end{align}
where the last inequality uses $n\ge12000$.

\smallskip
\noindent
\textbf{\emph{Estimate of $\Delta^2\varepsilon$.}}
Fix a point $z_0$ satisfying $|z_0-n|\le n/2+1$. 
Consider the disk $\{z: |z-z_0|\le n/7\}$.
It is contained in $\widetilde D_0$
since $n/2+1+ n/7<{2n}/{3}$.
Let $C_z$ denote the positively oriented circle
$|z-z_0|=n/7$.
By Cauchy's integral formula, for $|v-z_0|<n/7$,
\begin{equation*}
\varepsilon(v)
=
\frac{1}{2\pi i}
\int_{C_z}
\frac{\varepsilon(\theta)}{\theta-v}\,d\theta.
\end{equation*}
Differentiating twice with respect to $v$ and then
setting $v=z_0$, we obtain
\begin{equation*}
\varepsilon''(z_0)
=
\frac{2!}{2\pi i}
\int_{C_z}
\frac{\varepsilon(\theta)}{(\theta-z_0)^3}\,d\theta.
\end{equation*}

Recall that $\varepsilon$ is holomorphic on $\widetilde D_0$, and the disk $\{z: |z-z_0|\le n/7\}$ is contained in $\widetilde D_0$.
By~\eqref{eq:g0-epsilon-bound}, we have 
$
|\varepsilon(z)|<0.051$
on $\{z: |z-z_0|\le n/7\}$.
Put $\rho=n/7$.
Since the length of $C_z$ is $2\pi n/7$, it follows that
\begin{equation*}
\begin{aligned}
|\varepsilon''(z_0)|
&\le
\frac{2!}{2\pi}
\int_{C_z}
\frac{|\varepsilon(\theta)|}{|\theta-z_0|^3}\,|d\theta|\le
\frac{2!}{2\pi}
\frac{0.051}{\rho^3}\,(2\pi\rho)=
\frac{2!\,(0.051)}{\rho^2}=
\frac{2!\,(0.051)}{(n/7)^2}
=
\frac{4.998}{n^2}
<
\frac5{n^2}.
\end{aligned}
\end{equation*}
Since $z_0$ was arbitrary, this estimate holds whenever
$|z_0-n|\le n/2+1$.
Let $z\in D_0$.
Then, for $-1\le s\le1$,
$
|z+s-n|
\le |z-n|+1
\le n/2+1,
$
and hence
\begin{equation*}
|\varepsilon''(z+s)|
\le \frac{4.998}{n^2}
< \frac5{n^2}.
\end{equation*}
Taylor's formula with integral remainder gives
\begin{equation*}
\varepsilon(z\pm1)-\varepsilon(z)\mp \varepsilon'(z)
=
\int_0^1(1-s)\varepsilon''(z\pm s)\,ds.
\end{equation*}
Hence, for $z\in D_0$,  we have $\left|\Delta^2\varepsilon(z)\right|<5/n^2$. Indeed,
\begin{align}\label{eq:g0-epsilon-difference}
\left|\Delta^2\varepsilon(z)\right|
&=
|\varepsilon(z-1)-2\varepsilon(z)+\varepsilon(z+1)|=
\left|\int_0^1(1-s)
\bigl(\varepsilon''(z+s)+\varepsilon''(z-s)\bigr)\,ds\right|\nonumber\\
&=
\left|\int_{-1}^1(1-|s|)\varepsilon''(z+s)\,ds\right|\le
\int_{-1}^1(1-|s|)\,|\varepsilon''(z+s)|\,ds\nonumber\\
&\le 
\frac{4.998}{n^2} \int_{-1}^1(1-|s|)\,ds
=\frac{4.998}{n^2}
<\frac5{n^2}.
\end{align}

\smallskip
\noindent
\textbf{\emph{Estimate of $\Delta^2\log\Gamma(2z+1)-2/z$.}}
The gamma functional equation gives
\begin{equation*}
\Gamma(2z+3)
=
(2z+2)(2z+1)\Gamma(2z+1),\qquad 
\Gamma(2z+1)
=
(2z)(2z-1)\Gamma(2z-1).
\end{equation*}
Consequently,
\begin{equation*}
\frac{\Gamma(2z+3)\Gamma(2z-1)}
{\Gamma(2z+1)^2}
=
\frac{(2z+1)(2z+2)}{(2z)(2z-1)}
=
\frac{(1+1/z)(1+1/(2z))}
{1-1/(2z)}.
\end{equation*}

On the disk $|z-n|<2n/3-1$, we have
$\Re z>n/3+1>1$.
The principal logarithms of
$1+1/z$, $1+1/(2z)$, and $1-1/(2z)$
are therefore holomorphic.
Using the branches fixed above, the logarithmic
identity holds on the positive real interval of this disk.
The identity theorem then gives
\begin{align}
\Delta^2\Log\Gamma(2z+1)
&=
\Log\Gamma(2z-1)-2\Log\Gamma(2z+1)
+\Log\Gamma(2z+3)
\nonumber\\
&=
\Log\left(1+\frac1z\right)
+
\Log\left(1+\frac1{2z}\right)
-
\Log\left(1-\frac1{2z}\right).
\label{eq:g0-gamma-identity}
\end{align}

For $z\in D_0$, we have $|z|\ge n/2\ge6000$.
Expanding the logarithms into their convergent
power series yields
\begin{align*}
\Delta^2\Log\Gamma(2z+1)-\frac2z
={}&
\left[
\Log\left(1+\frac1z\right)-\frac1z
\right]+
\left[
\Log\left(1+\frac1{2z}\right)
-
\Log\left(1-\frac1{2z}\right)
-\frac1z
\right]\\
={}&
\sum_{j\ge2}\frac{(-1)^{j+1}}{jz^j}
+
2\sum_{j\ge1}
\frac1{(2j+1)(2z)^{2j+1}}.
\end{align*}
For the first series,
\begin{align*}
\left|
\sum_{j\ge2}\frac{(-1)^{j+1}}{jz^j}
\right|
&\le
\frac12\sum_{j\ge2}\frac1{|z|^j}=
\frac1{2|z|^2(1-|z|^{-1})}\le
\frac2{n^2(1-2/n)}
<
\frac3{n^2}.
\end{align*}
For the second series,
\begin{align*}
\left|
2\sum_{j\ge1}
\frac1{(2j+1)(2z)^{2j+1}}
\right|
&\le
\frac23\sum_{j\ge1}\frac1{(2|z|)^{2j+1}}=
\frac1{12|z|^3\left(1-\frac1{4|z|^2}\right)}\le
\frac2{3n^3(1-1/n^2)}
<
\frac1{n^2}.
\end{align*}
The last inequalities hold for $n\ge12000$.
Combining these estimates, we obtain
\begin{equation}\label{eq:g0-gamma-difference}
\left|
\Delta^2\Log\Gamma(2z+1)-\frac2z
\right|
<
\frac4{n^2},
\qquad z\in D_0.
\end{equation}

\smallskip
\noindent
\textbf{\emph{Proof of~\eqref{eq:g0-approximation}.}}
Applying~\eqref{eq:g0-p-difference},
\eqref{eq:g0-h0-difference},
\eqref{eq:g0-epsilon-difference},
and~\eqref{eq:g0-gamma-difference}
to~\eqref{g-Q}, we obtain
\begin{equation*}
|g_0(z)-Q(z)|
<
\frac{1+8+5+4}{n^2}
=
\frac{18}{n^2}
<
\frac{20}{n^2}.
\end{equation*}
This proves the first inequality
in~\eqref{eq:g0-approximation}.

For $z\in D_0$,
\begin{equation*}
|z|\ge\frac n2,
\qquad
\frac n2\le\Re z\le\frac{3n}{2}.
\end{equation*}
Furthermore, $|z-n|^2\le n^2/4$ implies
\begin{equation*}
|z|^2
\le
2n\Re z-\frac{3n^2}{4}
\le
\frac32n\Re z,
\end{equation*}
where the second inequality uses $\Re z\le3n/2$.
Consequently,
\begin{equation}\label{eq:g0-inverse-bounds}
\Re\frac nz
=
\frac{n\Re z}{|z|^2}
\ge\frac23,
\qquad
\left|\frac nz\right|\le2.
\end{equation}
Since
\begin{equation*}
nQ(z)=\frac{2n}{z}-\frac{2n}{z(1+w)},
\end{equation*}
equations~\eqref{eq:g0-w-bounds}
and~\eqref{eq:g0-inverse-bounds} give
\begin{align*}
\Re\bigl(nQ(z)\bigr)
&\ge
2\Re\frac nz
-
2\left|\frac nz\right|\frac1{|1+w|}\ge
\frac43-\frac47
=
\frac{16}{21}.
\end{align*}
Similarly,
\begin{equation*}
|nQ(z)|
=
2\left|\frac nz\right|
\left|\frac{w}{1+w}\right|
\le4.
\end{equation*}

The approximation bound already proved gives
\begin{equation*}
|n(g_0(z)-Q(z))|
\le
\frac{20}{n}
\le
\frac1{600}.
\end{equation*}
It follows that
\begin{align*}
\Re\bigl(ng_0(z)\bigr)
&\ge
\Re\bigl(nQ(z)\bigr)
-
|n(g_0(z)-Q(z))|\ge
\frac{16}{21}-\frac1{600}
=
\frac{3193}{4200}
>
\frac58,
\end{align*}
and
\begin{equation*}
|ng_0(z)|
\le
|nQ(z)|+|n(g_0(z)-Q(z))|
\le
4+\frac1{600}
=
\frac{2401}{600}
<
\frac92.
\end{equation*}
This proves all three inequalities
in~\eqref{eq:g0-approximation}.

\smallskip
\noindent
\textbf{\emph{Proof of~\eqref{54}.}}
At $z=n$, the number $w=W(2n/\pi)$ is real
and satisfies $w\ge6$.
Hence
\begin{equation*}
nQ(n)
=
\frac{2w}{1+w}
=
2-\frac2{1+w}
\ge
2-\frac27
=
\frac{12}{7}.
\end{equation*}
Since $g_0(n)$ is real, the approximation bound gives
\begin{equation*}
ng_0(n)
\ge
nQ(n)-|n(g_0(n)-Q(n))|
\ge
\frac{12}{7}-\frac1{600}
=
\frac{7193}{4200}
>
\frac54.
\end{equation*}
This proves~\eqref{54} and completes the proof.
\end{proof}

\subsection{Estimate of the iterated logarithmic ratios $g_m(z)$}
In this subsection, we obtain estimates for the iterates $g_m(z)$ initialized at $g_0(z)$. 
We begin by defining $g_{m+1}(z)$ via the recurrence
\begin{equation}\label{eq:rece_g_j+1}
g_{m+1}(z)
=
2g_m(z)
-h(g_m(z-1))
+2h(g_m(z))
-h(g_m(z+1)),
\end{equation}
where $h(z)=\Log(1-e^{-z})$.
The following lemma shows that these iterates remain well-defined on successively smaller disks and satisfy uniform analytic bounds.

\begin{lemma}\label{lem:analytic_escape}
Let $n\ge 6000$. 
For every $0 \le m \le \lceil \log_2 n\rceil$, define
\begin{equation*}
D_m=\{z\in\mathbb C:|z-n|<2^{-{m}/{3}-1}n\}.
\end{equation*}
Suppose that there exists a function $g_0$ such that $g_0$ is holomorphic on $D_0$, real-valued on $D_0 \cap \mathbb{R}$, and
\begin{equation*}
\Re\bigl(n g_0(z)\bigr)\ge \frac58,
\quad
|n g_0(z)|\le \frac92
\quad \text{for}\ z\in D_0\quad \textrm{and}\quad
n g_0(n)\ge \frac54.
\end{equation*}
Then, for every $0 \le m \le \lceil \log_2 n\rceil$, the iterates $g_{m}$ defined in~\eqref{eq:rece_g_j+1} satisfy the following properties:
\begin{itemize}
\item[\rm(i)] $g_m(z)$ is well-defined and holomorphic on $D_m$;
\item[\rm(ii)] $\Re g_m(z)>0$ for all $z\in D_m$;
\item[\rm(iii)] $g_m(z)$ is real-valued on $D_m\cap\mathbb R$;
\item[\rm(iv)] for all $z\in D_m$, we have 
$\left|2^{-m}g_m(z)-g_0(z)\right|\le S_m<{1500}/{n^2}$,
where $S_0=0$ and
\begin{equation*}
S_m:=
\sum_{\ell=0}^{m-1}
\frac{3\cdot2^{-\ell}}{\delta_\ell^2}, \quad \textrm{where} \quad \delta_\ell=2^{-\ell/3-1}(1-2^{-1/3})n.
\end{equation*} 
\end{itemize}
\end{lemma}

\begin{proof}
We will prove (i--iv) by induction on $m$.
Before that, we give a basic estimate.
For $0\le m< \lceil \log_2 n\rceil$, let 
$R_m=2^{-{m}/{3}-1}n$.
Then
\begin{equation}\label{deltam}
\delta_m=R_m-R_{m+1}
=2^{-m/3-1}(1-2^{-1/3})n.
\end{equation}
Then $2^m<n$, and
$2^{-m/3}>n^{-1/3}$.
Since $2^{-1/3}<4/5$, it follows that
\begin{equation*}
\delta_m
=2^{-m/3-1}(1-2^{-1/3})n
>
\frac{n^{2/3}}{10}
>20,
\end{equation*}
where the last inequality follows from $n\ge6000$.
In particular,
\begin{equation*}
z\in D_{m+1}
\quad\Longrightarrow\quad
z-1,z,z+1\in D_m.
\end{equation*}
More generally, if $z\in D_{m+1}$ and $-1\le s\le1$, then
the disk centered at $z+s$ with radius $\delta_m-1$ is contained
in $D_m$.

Now we start the induction.
For $m=0$, these assertions follow directly from the hypotheses,
and (iv) is trivial since $S_0=0$.
Now assume that the assertions hold at some $m<\lceil \log_2 n\rceil$.
Since $n\ge 6000$,
by the induction hypothesis (iv), we have
\begin{equation*}
\left|\frac{g_m(z)}{2^m/n}-ng_0(z)\right|
\le nS_m
<\frac{1500}{n}
\le\frac14.
\end{equation*}
Since $\Re\bigl(n g_0(z)\bigr)\ge \frac58$ and $|n g_0(z)|\le \frac92$ we have 
\begin{equation}\label{eq:mod_w}
\Re \left(\frac{g_m(z)}{2^m/n}\right)
\ge
\frac58-\frac14
=
\frac38,
\quad
\text{and} \quad 
\left|\frac{g_m(z)}{2^m/n}\right|
\le
\frac92+\frac14
=
\frac{19}{4}
<5.
\end{equation}

\medskip
\noindent
\textbf{\emph{The auxiliary function $H_m$.}}
Let
\begin{equation}\label{dis}
H_m(z):=h(g_m(z))-\log\left(2^m/n\right)
=
\Log\left(
\frac{1-e^{-g_m(z)}}{2^m/n}
\right).
\end{equation}
We derive an upper bound for $|\Delta^2 H_m(z)|$, and we use it later in the induction.

We first verify that~\eqref{dis} is well-defined.
By~\eqref{eq:mod_w}, we have
\begin{equation}\label{eq:Re_v}
\Re \left(\frac{1-e^{-g_m(z)}}{2^m/n}\right)
\ge\frac{1-e^{-3\cdot 2^m/(8n)}}{2^m/n}
\ge
\frac{1-e^{-3/4}}{2}
>
\frac14.
\end{equation}
Hence, the
principal logarithm is well-defined and holomorphic there.
Moreover, 
\begin{equation*}
\left|\left(\frac{1-e^{-g_m(z)}}{2^m/n}\right)\right|=
\left|\frac{g_m(z)}{2^m/n}\int_0^1e^{-sg_m(z)}\,ds\right|\le|2^{-m}ng_m(z)|<5.
\end{equation*}
Combining~\eqref{eq:mod_w} with~\eqref{eq:Re_v}, 
\begin{equation*}
\frac14<\left|\frac{1-e^{-g_m(z)}}{2^m/n}\right|<5,\quad \left|\log\left|\frac{1-e^{-g_m(z)}}{2^m/n}\right|\right|<5/3,\quad \left|\arg \frac{1-e^{-g_m(z)}}{2^m/n}\right|<\pi/2<5/3
\end{equation*}
Hence,
\begin{align}\label{eq:mod_Hj}
|H_m(z)|
=
\left|\Log \left(\frac{1-e^{-g_m(z)}}{2^m/n}\right)\right|=\left|\log \left |\frac{1-e^{-g_m(z)}}{2^m/n}\right|+ i \Arg\left(\frac{1-e^{-g_m(z)}}{2^m/n}\right)\right|
<
\frac{5\sqrt2}{3}
<
\frac52.
\end{align}
Thus $H_m$ is holomorphic on $D_m$ and satisfies
$|H_m(z)|<5/2$ for all $z\in D_m$.

Fix $z\in D_{m+1}$ and $-1\le s\le1$.
Then the disk centered at $z+s$ with radius
$\delta_m-1$ is contained in $D_m$.
By~\eqref{eq:mod_Hj},
applying the Cauchy estimate (see~\cite[Theorem 10.26]{Rud87}) to $H_m$, and since $\delta_m>20$, we obtain
\begin{equation}\label{eq:secdiff_Hj}
|H_m''(z+s)|
\le
\frac{2!}{(\delta_m-1)^2}\cdot \frac{5}{2}
=
\frac{5}{(\delta_m-1)^2}
\le
\frac{6}{\delta_m^2}.
\end{equation}
The identity
\begin{align*}
\Delta^2H_m(z)&=\big(H_m(z+1)-H_m(z)\big)-\big(H_m(z)-H_m(z-1)\big)\\
&=\int_{0}^{1}H_m'(z+t)dt-\int_{-1}^{0}H_m'(z+t)dt \\
&=\int_{0}^{1}\left(H_m'(z)+\int_{0}^{t}H_m''(z+s)ds\right)dt
-\int_{-1}^{0}\left(H_m'(z)+\int_{0}^{t}H_m''(z+s)ds\right)dt \\
&=\int_{0}^{1}\left(\int_{0}^{t}H_m''(z+s)ds\right)dt
-\int_{-1}^{0}\left(\int_{0}^{t}H_m''(z+s)ds\right)dt \\
&=\int_{0}^{1}\left(\int_{s}^{1}dt\right)H_m''(z+s)ds
+\int_{-1}^{0}\left(\int_{-1}^{s}dt\right)H_m''(z+s)ds \\
&=\int_{-1}^{1}(1-|s|)H_m''(z+s)ds,
\end{align*}
therefore gives
\begin{equation}\label{eq:sec_Hj}
|\Delta^2H_m(z)|
\le
\frac6{\delta_m^2}
\int_{-1}^{1}(1-|s|)\,ds
=
\frac6{\delta_m^2}.
\end{equation}

\medskip
\noindent
\textbf{\emph{Proof of (i) and (iv).}}
Since $\Re g_m>0$ on $D_m$, for every $u\in D_m$,
\begin{equation*}
\Re(1-e^{-g_m(u)})
=
1-e^{-\Re g_m(u)}
\cos(\Im g_m(u))
>0.
\end{equation*}
Hence
\begin{equation*}
h(g_m(u))
=
\Log(1-e^{-g_m(u)})
\end{equation*}
is well-defined and holomorphic on $D_m$.

For $z\in D_{m+1}$, since
$
z-1,z,z+1\in D_m,
$
we have 
\begin{equation}\label{eq:proof-3.4}
g_{m+1}(z)
=
2g_m(z)
-h(g_m(z-1))
+2h(g_m(z))
-h(g_m(z+1))
\end{equation}
is well-defined and holomorphic on $D_{m+1}$.
This proves (i).

Since $h(g_m(z))=H_m(z)+\log (2^m/n)$ and $\log  (2^m/n)$ is independent of $z$, by~\eqref{eq:proof-3.4},
we can obtain that
\begin{equation}\label{eq:g_j+1}
g_{m+1}
=
2g_m-\Delta^2H_m.
\end{equation}
Dividing~\eqref{eq:g_j+1} by $2^{m+1}$ yields
\begin{equation*}
2^{-(m+1)}g_{m+1}
=
2^{-m}g_m
-
2^{-(m+1)}\Delta^2H_m.
\end{equation*}
Hence, by~\eqref{eq:sec_Hj},
\begin{equation}\label{eq:bound_g_j+1}
\left|
2^{-(m+1)}g_{m+1}(z)
-
2^{-m}g_m(z)
\right|
\le
\frac{3\cdot2^{-m}}{\delta_m^2}.
\end{equation}
Using the induction hypothesis, we have
\begin{equation}\label{eq:bound_S_j+1}
\begin{aligned}
\left|
2^{-(m+1)}g_{m+1}(z)-g_0(z)
\right|
&\le \left|
2^{-(m+1)}g_{m+1}(z)
-
2^{-m}g_m(z)
\right|
+
\left|
2^{-m}g_m(z)-g_0(z)
\right|
\\
&\le
\frac{3\cdot2^{-m}}{\delta_m^2}
+
S_m
=
S_{m+1}.
\end{aligned}
\end{equation}

It suffices to prove 
$$S_m=
\sum_{\ell=0}^{m-1}
\frac{3\cdot2^{-\ell}}{\delta_\ell^2}\le \frac{1500}{n^2}.$$
By~\eqref{deltam},
$\delta_\ell
=2^{-\ell/3-1}(1-2^{-1/3})n$,
we have
\begin{equation*}
\frac{3\cdot2^{-\ell}}{\delta_\ell^2}
=
\frac{12}{n^2(1-2^{-1/3})^2}
\,2^{-\ell/3}.
\end{equation*}
Therefore
\begin{equation*}
\sum_{\ell\ge0}
\frac{3\cdot2^{-\ell}}{\delta_\ell^2}
=
\frac{12}{n^2(1-2^{-1/3})^2}
\sum_{\ell\ge0}2^{-\ell/3}
=
\frac{12}{n^2(1-2^{-1/3})^3}.
\end{equation*}
Since $2^{-1/3}<4/5$,
\begin{equation*}
\frac{12}{n^2(1-2^{-1/3})^3}
<
\frac{1500}{n^2}.
\end{equation*}
Hence
\begin{equation}\label{1500}
\left|
2^{-(m+1)}g_{m+1}(z)-g_0(z)
\right|
\le
S_{m+1}
<
\frac{1500}{n^2},
\end{equation}
which proves (iv).

\medskip
\noindent
\textbf{\emph{Proof of (ii--iii)}}
By~\eqref{1500} and $n\ge 6000$,
we have
\begin{equation*}
\left|2^{-(m+1)}ng_{m+1}(z)-ng_0(z)\right|
<
\frac{1500}{n}
\le
\frac14.
\end{equation*}
Since $\Re\bigl(n g_0(z)\bigr)\ge \frac58$ and $|n g_0(z)|\le \frac92$ we have 
\begin{equation*}
\Re \left(\frac{g_{m+1}(z)}{2^{m+1}/n}\right)
\ge
\frac58-\frac14
=
\frac38.
\end{equation*}
Thus
\begin{equation*}
\Re g_{m+1}(z)>0
\quad \text{for}\ z\in D_{m+1}.
\end{equation*}

Now let $x\in D_{m+1}\cap\mathbb R$.
Then $x-1,x,x+1\in D_m\cap\mathbb R$.
By the induction hypothesis,
\begin{equation*}
g_m(x-1),\quad g_m(x),\quad g_m(x+1)
\end{equation*}
are positive real numbers. Hence
\begin{equation*}
h(g_m(x-1)),\quad
h(g_m(x)),\quad
h(g_m(x+1))
\end{equation*}
are also real. It follows from the recurrence~\eqref{eq:rece_g_j+1} that $g_{m+1}(x)\in\mathbb R$.
Thus all the induction assertions hold at step $m+1$.
This completes the proof.
\end{proof}

Combining the results of this section, we obtain the following proposition.

\begin{proposition}\label{prop:large-index}
Let $(\lambda_n)_{n\ge0}$ be the sequence defined in~\eqref{eq:F(x)}, and let
\begin{equation*}
A_{m,n}=(\mathcal L^m\lambda)_n,
\quad
r_{m,n}
=
\frac{A_{m,n}^2}{A_{m,n-1}A_{m,n+1}},
\end{equation*}
defined whenever the denominator is positive. 
Then, for every integer $n\ge12000$,
\begin{equation*}
A_{m,n}>0,
\quad
0\le m\le \left\lceil\log_2 n\right\rceil,
\end{equation*}
and the ratio $r_{\left\lceil\log_2 n\right\rceil,n}$ is well-defined and satisfies
$r_{\left\lceil\log_2 n\right\rceil,n}> e>(3+\sqrt5)/2$.
\end{proposition}

\begin{proof}
Fix an integer $n\ge12000$. First, Lemma~\ref{lem:explicit_analytic_initial_ratio} implies that $g_0$ in~\eqref{eq:g0(z)} is holomorphic on the disk $D_0 := \{z\in \mathbb{C} : |z-n| < n/2\}$ and real-valued on $D_0\cap \mathbb{R}$. Moreover, for all $z\in D_0$,
\begin{equation*}
\Re\bigl(n g_0(z)\bigr) \ge \frac{5}{8}, \quad |n g_0(z)| \le \frac{9}{2}, \quad \text{and} \quad n g_0(n) \ge \frac{5}{4}.
\end{equation*}
These estimates verify all initial hypotheses of Lemma~\ref{lem:analytic_escape}.

Next, define the iterates $g_m$ for $0\le m < \lceil \log_2 n\rceil$ by the recurrence
\begin{equation*}
g_{m+1}(z) = 2g_m(z) - h(g_m(z-1)) + 2h(g_m(z)) - h(g_m(z+1)),
\end{equation*}
where $h(z) = \operatorname{Log}(1-e^{-z})$. 
Since $n\ge12000 > 6000$, Lemma~\ref{lem:analytic_escape} ensures that, for each $0\le m \le \lceil \log_2 n\rceil$, the function $g_m$ is well-defined and holomorphic on $D_m := \{z\in\mathbb C : |z-n| < R_m\}$ with $R_m =2^{-m/3-1}n$, real-valued on $D_m\cap\mathbb R$, and satisfies $R_m - R_{m+1} > 1$ for all $0\le m < \lceil \log_2 n\rceil$.

By Lemma~\ref{lem:explicit_analytic_initial_ratio}, we have $g_0(k) = \log r_{0,k}$ for all $k\in D_0\cap \mathbb{Z}$. Hence, Lemma~\ref{lem:local_ratio_identification} yields
\begin{equation*}
A_{m,k} > 0 \quad \text{for all } k\in D_{m-1}\cap \mathbb{Z} \text{ and } 1\le m\le \lceil \log_2 n\rceil.
\end{equation*}
Because $n\in D_{m-1}$ for each $1\le m\le \lceil \log_2 n\rceil$, this implies $A_{m,n} > 0$ for all $1\le m\le \lceil \log_2 n\rceil$. 
Combining this with $A_{0,n} = \lambda_n > 0$ (see Lemma~\ref{lem:theta_moments_and_interpolation}), we conclude that
\begin{equation*}
A_{m,n} > 0 \quad \text{for all } 0\le m \le \lceil \log_2 n\rceil.
\end{equation*}

Finally, since $n\in D_{\lceil \log_2 n\rceil}$, Lemma~\ref{lem:local_ratio_identification} establishes the identification $g_{\lceil \log_2 n\rceil}(n) = \log r_{{\lceil \log_2 n\rceil},n}$. Setting $m = {\lceil \log_2 n\rceil}$ and $z = n$ in part~(iv) of Lemma~\ref{lem:analytic_escape}, we obtain
\begin{equation*}
2^{-{\lceil \log_2 n\rceil}}g_{\lceil \log_2 n\rceil}(n) \ge g_0(n) - \frac{1500}{n^2}.
\end{equation*}
Since $2^{\lceil \log_2 n\rceil} \ge n$, $n g_0(n) \ge 5/4$, and $1500/n \le 1500/12000=1/8$, it follows that
\begin{equation*}
g_{\lceil \log_2 n\rceil}(n) \ge \frac{2^{\lceil \log_2 n\rceil}}{n} \left( n g_0(n) - \frac{1500}{n} \right) \ge 1 \cdot \left( \frac{5}{4} - \frac{1}{8} \right) = \frac{9}{8} > 1.
\end{equation*}
Consequently,
\begin{equation*}
r_{{\lceil \log_2 n\rceil},n}=
e^{g_{\lceil \log_2 n\rceil}(n)}
>e
>\frac{3+\sqrt5}{2}.
\end{equation*}
This completes the proof.
\end{proof}

%=================================================
\section{The case $n\le 12000$ and $m\le 14$}\label{sec:finite}
%=================================================
In this section, we establish the positivity of $A_{m,n}$ for every integer $1\le n\le12000$ and $0\le m\le14$, together with the ratio bound $r_{14,n}>(3+\sqrt5)/2$.

We first review the necessary background on Legendre polynomials and Gauss--Legendre integration.
See, for instance,~\cite[p.~887]{AS64}.

\begin{lemma}\label{lem:gauss_legendre_facts}
Let $P_d(z)$ be the Legendre polynomial normalized by $P_d(1)=1$,
and let $z_1, \ldots, z_d \in (-1, 1)$ be its distinct zeros.
Define
\begin{equation*}
\omega_j=
\frac{2}
{(1-z_j^2)P_d'(z_j)^2},
\quad 1\le j\le d.
\end{equation*}
Then $\omega_j>0$, $\sum_{j=1}^d \omega_j=2$, and
\begin{equation}\label{eq:gauss-exact}
\int_{-1}^{1}p(z)\,dz
=
\sum_{j=1}^{d}\omega_j p(z_j)
\end{equation}
for every polynomial $p$ of degree at most $2d-1$.
Moreover, for $-1\le z\le1$, we have
\begin{equation}\label{eq:legendre-bounds}
|P_d(z)|\le1,
\quad
|P_d'(z)|\le d^2,
\quad
|P_d''(z)|\le d^4 .
\end{equation}
\end{lemma}

\begin{proof}
For the standard properties of the zeros of $P_d$, the positivity and explicit formula for the Gauss--Legendre weights, and the exactness of~\eqref{eq:gauss-exact} up to degree $2d-1$, we refer the reader to~\cite[p. 887]{AS64}.
It remains to verify the bounds in~\eqref{eq:legendre-bounds}.

The classical integral representation (see~\cite[p.~86]{Sze75})
\begin{equation*}
P_d(\cos\theta)
=
\frac{1}{\pi}
\int_0^\pi
\bigl(\cos\theta+i\sin\theta\cos u\bigr)^d\,du
\end{equation*}
gives $|P_d(z)|\le1$ for any $-1\le z\le1$.

The following recurrence relation follows from the standard Legendre recurrence and differential relations (see~\cite[p.~165]{SLL09})
\begin{equation*}
P'_{d+1}(z)-P'_{d-1}(z)=(2d+1)P_d(z).
\end{equation*}
Then by telescoping, we have
\begin{equation*}
P'_d(z)
=
\sum_{\substack{0\le j<d\\ d-j\ {\rm odd}}}
(2j+1)P_j(z).
\end{equation*}	
Hence,
\begin{equation*}
|P'_d(z)|
\le
\sum_{\substack{0\le j<d\\ d-j\ {\rm odd}}}(2j+1)
=
\frac{d(d+1)}{2}
\le d^2.
\end{equation*}
Applying the same bound to $P'_d$ gives
\begin{align*}
|P''_d(z)|
&\le
\sum_{\substack{0\le j<d\\ d-j\ {\rm odd}}}
(2j+1)|P'_j(z)| \le
d^2
\sum_{\substack{0\le j<d\\ d-j\ {\rm odd}}}
(2j+1)
=
\frac{d^3(d+1)}{2}
\le d^4 .
\end{align*}
This proves \eqref{eq:legendre-bounds}.
\end{proof}

Let 
\begin{equation}\label{eq:Kt}
K(t):=4e^{t/2}\sum_{j\ge1}(2\pi^2 j^4e^{4t}-3\pi j^2e^{2t})e^{-\pi j^2e^{2t}}.
\end{equation}
Recall that by~\eqref{eq:lambda-moment},
\begin{equation}\label{eq:lambda-moment-K}
\lambda_n=
\frac1{(2n)!}
\int_0^\infty
K(t)t^{2n}\,dt,
\end{equation}

Romik~\cite{Rom21} and O'Sullivan~\cite{Osu21} investigated kernel representations and coefficient asymptotics for the Riemann xi-function. 
In the following lemma, 
we establish explicit uniform bounds for $K$, 
an approximation of $K$ by its truncation $K_{16}$, 
and an estimate for the tail integral of $K$.

\begin{lemma}\label{lem:finite_kernel_bounds}
Let 
\begin{equation}\label{eq:Kt16}
K_{16}(t):=4e^{t/2}\sum_{j=1}^{16}(2\pi^2 j^4e^{4t}-3\pi j^2e^{2t})e^{-\pi j^2e^{2t}}.
\end{equation}
\begin{itemize}
\item[\rm(i)] Both $K$ and $K_{16}$ are holomorphic on a neighborhood of every compact
subset of $\{t\in\mathbb C:|\Im t|\leq 1/8 \}$, and, for $t=x+iy$, $x\geq-1/8$, $|y|\leq 1/8$,
\begin{equation}\label{eq:K_sup}
|K(t)|,\ |K_{16}(t)|\leq64\pi^2
\exp\bigl(9x/2-\pi\cos(1/4) e^{2x}\bigr).
\end{equation}
\item[\rm(ii)]  For real $t\geq0$,
\begin{equation}\label{eq:K-K16}
0\leq K(t)-K_{16}(t)\leq 10^{-300} K(t),\quad
\text{i.e.,} \
K_{16}(t)\leq K(t)\leq\frac{K_{16}(t)}{1-10^{-300}}.
\end{equation}
\item[\rm(iii)] For every $0\leq n\leq12016$,
\begin{equation}\label{eq:moment_K}
0\leq\int_5^\infty K(t)t^{2n}\,dt\leq \frac{64\pi^2\,5^{2n}\exp(45/2-\pi e^{10})}{\pi e^{10}}.
\end{equation}
\end{itemize}
\end{lemma}

\begin{proof}
We prove the three assertions separately. 
For convenience, write
\begin{equation*}
K(t)=\sum_{j\ge1}\widetilde{K}_j(t),
\end{equation*}
where $\widetilde{K}_j(t)=4e^{t/2}(2v_j(t)^2-3v_j(t))e^{-v_j(t)}$ with $v_j(t)=\pi j^2e^{2t}$.

\medskip
\noindent
\textbf{\emph{Holomorphy of $K$ and $K_{16}$.}}
Clearly, the function $\widetilde{K}_j(t)$
is entire for each fixed $j$. 
Hence $K_{16}$, being a finite sum of entire functions, is entire.
For the infinite series $K$, let
\begin{equation*}
C\subset\{t\in\mathbb C:|\Im t|\le 1/8\}
\end{equation*}
be compact. 
Since $1/8<\pi/4$,
we may choose $\eta>0$ such that
$1/8+\eta<\pi/4$.
Then there exist real numbers $r_1<r_2$ such that
$C$ is contained in the open set
\begin{equation*}
U
=
\left\{
t=x+iy:
r_1<x<r_2,\quad |y|<1/8+\eta
\right\}.
\end{equation*}

For $t=x+iy\in U$, we have $\Re v_j(t)=\pi j^2e^{2x}\cos(2y)$.
Since $x>r_1$ and $|y|<1/8+\eta<\pi/4$,
we obtain
\begin{equation*}
\Re v_j(t)
\ge
\pi j^2 e^{2r_1}
\cos(2(1/8+\eta)).
\end{equation*}
Consequently,
\begin{equation*}
|e^{-v_j(t)}|
=
e^{-\Re v_j(t)}
\le
e^{-\pi j^2 e^{2r_1}
\cos(2(1/8+\eta))}.
\end{equation*}
Furthermore,
\begin{equation*}
|v_j(t)|
=
\pi j^2e^{2x}
\le
\pi j^2 e^{2r_2},
\quad
|e^{t/2}|=e^{x/2}\le e^{r_2/2}.
\end{equation*}
Therefore there exists a constant $c>0$, depending only on $U$,
such that
\begin{equation*}
|\widetilde{K}_j(t)|
\le
cj^4e^{-\pi e^{2r_1}\cos(2(1/8+\eta))\,j^2}
\quad \text{for}\ t\in U.
\end{equation*}
Since
\begin{equation*}
\sum_{j\ge1}j^4e^{-\pi e^{2r_1}\cos(2(1/8+\eta))\,j^2}<\infty,
\end{equation*}
the Weierstrass $M$-test~\cite[Theorem~7.10]{Rud76} implies that
$\sum_{j\ge1}\widetilde{K}_j(t)$
converges uniformly on $U$.
Since each partial sum is holomorphic, Weierstrass's convergence theorem implies that $K$ is holomorphic on $U$. 
Because $C$ was arbitrary, $K$ is holomorphic on
a neighborhood of every compact subset of
$\{t:|\Im t|\le1/8\}$.

\medskip
\noindent
\textbf{\emph{Proof of \eqref{eq:K_sup}.}}
For $t=x+iy$, $x\geq-1/8$, $|y|\leq 1/8$, we have $|e^{t/2}|=e^{x/2}$,
\begin{equation*}
|v_j(t)|=\pi j^2 e^{2x},\quad |e^{-v_j(t)}|=e^{-\Re v_j(t)}=e^{-\pi j^2 e^{2x}\cos(2y)}.
\end{equation*}
Using $|2v_j^2-3v_j|\le2|v_j|^2+3|v_j|$, we can obtain that
\begin{align*}
|K(t)|
&\le
4e^{x/2}
\sum_{j\ge1}
\left[2(\pi e^{2x})^2j^4+3\pi e^{2x}j^2 \right]\exp(-\pi j^2 e^{2x}\cos(2y))
\\
&=
4e^{x/2}(\pi e^{2x})^2e^{-\pi e^{2x}\cos(2y)}
\sum_{j\ge1}
\left(
2j^4+\frac{3j^2}{\pi e^{2x}}
\right)
\exp(-\pi e^{2x}\cos(2y)(j^2-1)).
\end{align*}
Since $\pi e^{2x}>2$ for $x\ge -1/8$ and $j^2\le j^4$ for $j\ge1$, then
\begin{equation*}
2j^4+\frac{3j^2}{\pi e^{2x}}
\le
2j^4+\frac32j^2
\le
\frac72j^4.
\end{equation*}
Since $\pi e^{2x}\cos(2y)>2$ for $x\geq-1/8$, $|y|\leq 1/8$, then
\begin{equation*}
e^{-\pi e^{2x}\cos(2y)(j^2-1)}
\le
e^{-2(j^2-1)}.
\end{equation*}
Hence
\begin{equation*}
|K(t)|
\le
4e^{x/2}(\pi e^{2x})^2e^{-\pi e^{2x}\cos(2y)}
\frac72
\sum_{j\ge1}
j^4e^{-2(j^2-1)}.
\end{equation*}

For any $j\ge1$, we obtain
\begin{equation*}
\frac{(j+1)^4e^{-2((j+1)^2-1)}}{j^4e^{-2(j^2-1)}}
=
\left(\frac{j+1}{j}\right)^4
e^{-2(2j+1)}
\le
16e^{-6}.
\end{equation*}
Then we have
\begin{equation*}
\frac72 \cdot \sum_{j\ge1}j^4e^{-2(j^2-1)}
\le
\frac72 \cdot \sum_{j\ge1}(16e^{-6})^{j-1}
=
\frac72 \cdot \frac{1}{1-16e^{-6}}<16.
\end{equation*}
Consequently,
\begin{align*}
|K(t)|
&\le
64e^{x/2}(\pi e^{2x})^2e^{-\pi e^{2x}\cos(2y)}\\
&\le
64e^{x/2}(\pi e^{2x})^2e^{-e^{2x}\pi\cos(1/4)}\\
&=64\pi^2
\exp\left(
\frac92x-e^{2x}\pi\cos\frac{1}{4} 
\right).
\end{align*}

The same argument applies to $K_{16}$, since after taking absolute values the truncated sum is bounded by the corresponding infinite positive majorant. This proves~\eqref{eq:K_sup}.

\medskip
\noindent
\textbf{\emph{Proof of \eqref{eq:K-K16}.}}
We now restrict to the real axis and assume $t\ge0$. 
In this case, we have $\pi e^{2t}\ge \pi >3$ and $v_j(t)=\pi j^2e^{2t}>3$. 
Hence,
\begin{equation*}
\widetilde{K}_j(t)=4e^{t/2}(2v_j(t)^2-3v_j(t))e^{-v_j(t)}>0.
\end{equation*}
and
\begin{equation}\label{positive}
K(t)>0,
\quad
K_{16}(t)>0,
\quad
K(t)-K_{16}(t)
=
\sum_{j\ge17}\widetilde{K}_j(t)\ge0.
\end{equation}
Since $\widetilde{K}_1(t)
=4e^{t/2}(2(\pi e^{2t})^2-3\pi e^{2t})e^{-\pi e^{2t}}$,
we have
\begin{align*}
\frac{\widetilde{K}_j(t)}{\widetilde{K}_1(t)}
=
j^2
\frac{2\pi e^{2t}j^2-3}{2\pi e^{2t}-3}
e^{-\pi e^{2t}(j^2-1)}.
\end{align*}
Because $\pi e^{2t}>3$, then
\begin{equation*}
\frac{2\pi e^{2t}j^2-3}{2\pi e^{2t}-3}
\le
2j^2.
\end{equation*}
Therefore,
\begin{equation}\label{eq:KjK1}
\frac{\widetilde{K}_j(t)}{\widetilde{K}_1(t)}
\le
2j^4e^{-3(j^2-1)}.
\end{equation}

For $j\ge17$, we have
\begin{equation*}
\frac{(j+1)^4e^{-3((j+1)^2-1)}}{j^4e^{-3(j^2-1)}}
=
\left(\frac{j+1}{j}\right)^4
e^{-3(2j+1)}
\le
\left(\frac{18}{17}\right)^4e^{-105}
<
\frac12.
\end{equation*}
Thus,
\begin{equation*}
\sum_{j=17}^{\infty}j^4e^{-3(j^2-1)}
\le
17^4e^{-3(17^2-1)}
\sum_{m\ge0}2^{-m}
=
2\cdot17^4e^{-3(17^2-1)}.
\end{equation*}
Using \eqref{eq:KjK1}, we therefore obtain
\begin{align*}
\frac{K(t)-K_{16}(t)}{\widetilde{K}_1(t)}
&=
\sum_{j=17}^{\infty}
\frac{\widetilde{K}_j(t)}{\widetilde{K}_1(t)}
\le
2\sum_{j=17}^{\infty}
j^4e^{-3(j^2-1)}
\le
4\cdot17^4e^{-864}.
\end{align*}
Moreover, $4\cdot17^4e^{-864}<10^{-300}$.
This inequality can also be verified directly by outward interval arithmetic.
Hence
\begin{equation*}
K(t)-K_{16}(t)
\le
10^{-300}\widetilde{K}_1(t).
\end{equation*}
Since all summands are positive,
\begin{equation*}
\widetilde{K}_1(t)\le K(t),
\end{equation*}
and therefore
\begin{equation*}
0\le
K(t)-K_{16}(t)
\le
10^{-300}K(t).
\end{equation*}
This proves \eqref{eq:K-K16}.

\medskip
\noindent
\textbf{\emph{Proof of \eqref{eq:moment_K}.}}
Let $t\ge0$ be real. 
In this case, since $y = 0$,
as in the proof of~\eqref{eq:K_sup}, we have 
\begin{equation}\label{eq:real-kernel-bound}
0<K(t)
\le
64\pi^2
\exp\left(
\frac92t-\pi e^{2t}
\right).
\end{equation}

Fix an integer $0\le n\le12016$.
Multiplying \eqref{eq:real-kernel-bound} by $t^{2n}$, we get
\begin{equation*}
K(t)t^{2n}
\le
64\pi^2
\exp\left(
2n\log t+\frac92t-\pi e^{2t}
\right).
\end{equation*}
Define $g_n(t)=2n\log t+\frac92t-\pi e^{2t}$.
Then
\begin{equation*}
K(t)t^{2n}
\le
64\pi^2e^{g_n(t)}.
\end{equation*}

For $t\ge5$, we have
\begin{equation*}
g_n'(t)
=
\frac{2n}{t}
+\frac92
-2\pi e^{2t}.
\end{equation*}
Since $n\le12016$ and $t\ge5$, the numerical inequality
\begin{equation*}
\frac{2n}{t}+\frac92
\le
\frac{2\cdot12016}{5}+\frac92
<
\pi e^{10}
\end{equation*}
therefore gives
\begin{align*}
g_n'(t)
\le
\pi e^{10}-2\pi e^{2t}
\le
-\pi e^{2t}
\le
-\pi e^{10}.
\end{align*}
Integrating this inequality from $5$ to $t$ gives
\begin{equation*}
g_n(t)-g_n(5)
=
\int_5^t g_n'(s)\,ds
\le
-\pi e^{10}(t-5),
\end{equation*}
and hence
\begin{equation*}
g_n(t)
\le
g_n(5)-\pi e^{10}(t-5).
\end{equation*}
Now
\begin{equation*}
g_n(5)
=
2n\log5+\frac{45}{2}-\pi e^{10}.
\end{equation*}
It follows that, for $t\ge5$,
\begin{equation*}
K(t)t^{2n}
\le
64\pi^2
5^{2n}
\exp\left(
\frac{45}{2}-\pi e^{10}
\right)
e^{-\pi e^{10}(t-5)}.
\end{equation*}
Integrating over $[5,\infty)$, we obtain
\begin{align*}
\int_5^\infty K(t)t^{2n}\,dt
&\le
64\pi^2
5^{2n}
\exp\left(
\frac{45}{2}-\pi e^{10}
\right)
\int_5^\infty
e^{-\pi e^{10}(t-5)}\,dt
\\
&=
64\pi^2
5^{2n}
\exp\left(
\frac{45}{2}-\pi e^{10}
\right)
\frac{1}{\pi e^{10}}.
\end{align*}
Since $K(t)>0$ on the real axis,
\begin{equation*}
\int_5^\infty K(t)t^{2n}\,dt\ge0.
\end{equation*}
This proves \eqref{eq:moment_K} and completes the proof.
\end{proof}

The following lemma provides an explicit error bound for the
Gauss--Legendre quadrature used in our numerical verification
of the finite cases.

\begin{lemma}\label{lem:finite_gaussian_rule}
Let $z_1, \ldots, z_d $ be the distinct zeros of the Legendre polynomial $P_d$ and $$\omega_j=
\frac{2}
{(1-z_j^2)P_d'(z_j)^2}.$$
If a function $f(z)$ is holomorphic in a neighborhood of the closed disk $\{z: \vert{}z-c\vert{} \le \rho\}$ for some $c\in\mathbb{C}$ and $\rho>0$, and satisfies $\vert{}f(z)\vert{} \le M$ on this disk, then
\begin{equation}\label{eq:F9-detailed}
\left|
\int_{c-\rho/2}^{c+\rho/2} f(t)\,dt
-
\frac{\rho}{2}\sum_{j=1}^{d}\omega_j f(c+\rho z_j/2)
\right|
\le
4\rho M\,4^{-d}.
\end{equation}
\end{lemma}

\begin{proof}
Since $f$ is holomorphic in a neighborhood of
$\{z:|z-c|\le \rho\}$,
it admits a Taylor expansion at $c$,
\begin{equation*}
f(z)
=
\sum_{j\ge0} a_j(z-c)^j,
\quad
a_j
=
\frac{1}{2\pi i}
\int_{|z-c|=\rho}
\frac{f(z)}{(z-c)^{j+1}}\,dz.
\end{equation*}
Taking absolute values and using $|f(z)|\le M$, we can obtain that
\begin{equation}\label{eq:cauchy-coefficient-bound}
|a_j|
\le
\frac{1}{2\pi}
\int_{|z-c|=\rho}
\frac{|f(z)|}{|z-c|^{j+1}}\,|dz|
\le
\frac{1}{2\pi}
\frac{M}{\rho^{j+1}}
\int_{|z-c|=\rho}|dz|
=\frac{M}{\rho^j} \quad \text{for $j\ge0$}.
\end{equation}

Write
\begin{equation*}
f(z)=R_1(z)+R_2(z),
\end{equation*}
where
\begin{equation*}
R_1(z)
=
\sum_{j=0}^{2d-1}a_j(z-c)^j,
\quad
\text{and} \
R_2(z)
=
\sum_{j\ge2d}a_j(z-c)^j.
\end{equation*}

Consider the integral $\int_{c-\rho/2}^{c+\rho/2} f(t)\,dt$.
Setting $t=c+\rho x/2$ gives
\begin{equation*}
\int_{c-\rho/2}^{c+\rho/2}R_1(t)\,dt
=
\frac{\rho}{2}\int_{-1}^{1}R_1(c+\rho x/2)\,dx.
\end{equation*}
Since $R_1$ has degree at most $2d-1$, the polynomial $R_1(c+\rho x/2)$
also has degree at most $2d-1$. 
By Lemma~\ref{lem:gauss_legendre_facts}, we have
\begin{equation*}
\int_{-1}^{1}R_1(c+\rho x/2)\,dx
=
\sum_{j=1}^{d}
\omega_j R_1(c+\rho z_j/2).
\end{equation*}
Therefore
\begin{equation}\label{eq:gauss-exact-polynomial}
\int_{c-\rho/2}^{c+\rho/2}R_1(t)\,dt
=
\frac{\rho}{2}\sum_{j=1}^{d}
\omega_j R_1(c+\rho z_j/2).
\end{equation}
Using $f=R_1+R_2$
and~\eqref{eq:gauss-exact-polynomial}, we have
\begin{align*}
&
\int_{c-\rho/2}^{c+\rho/2}f(t)\,dt
-
\frac{\rho}{2}\sum_{j=1}^{d}\omega_j f(c+\rho z_j/2)
\\
&=
\int_{c-\rho/2}^{c+\rho/2}
\bigl(R_1(t)+R_2(t)\bigr)\,dt
-
\frac{\rho}{2}\sum_{j=1}^{d}
\omega_j
\bigl(
R_1(c+\rho/2z_j)+R_2(c+\rho/2 z_j)
\bigr)
\\
&=
\left[
\int_{c-\rho/2}^{c+\rho/2}R_1(t)\,dt
-
\frac{\rho}{2}\sum_{j=1}^{d}\omega_j R_1(c+\rho z_j/2)
\right]
+
\left[
\int_{c-\rho/2}^{c+\rho/2}R_2(t)\,dt
-
\frac{\rho}{2}\sum_{j=1}^{d}\omega_j R_2(c+\rho z_j/2)
\right]
\\
&=
\int_{c-\rho/2}^{c+\rho/2}R_2(t)\,dt
-
\frac{\rho}{2}\sum_{j=1}^{d}\omega_j R_2(c+\rho z_j/2)
\end{align*}
Then
\begin{equation}\label{eq:error-split}
\left|
\int_{c-\rho/2}^{c+\rho/2}f(t)\,dt
-
\frac{\rho}{2}\sum_{j=1}^{d}\omega_j f(c+\rho z_j/2)
\right|
\le
\left|
\int_{c-\rho/2}^{c+\rho/2}R_2(t)\,dt
\right|
+
\frac{\rho}{2}
\left|
\sum_{j=1}^{d}\omega_jR_2(c+\rho z_j/2)
\right|.
\end{equation}

Hence, we only need to consider $R_2$.
By~\eqref{eq:cauchy-coefficient-bound}, we obtain that
\begin{equation}\label{eq:Taylor-tail-bound}
|R_2(t)|
\le
\sum_{j\ge2d}
|a_j|\,|t-c|^j
\le
M
\sum_{j\ge2d}
\left(\frac{|t-c|}{\rho}\right)^j
\le
M
\sum_{j\ge2d}
\left(\frac12\right)^j
=2M4^{-d}.
\end{equation}
Since the interval $[c-\rho/2,c+\rho/2]$ has length $\rho$, we have
\begin{equation}\label{eq:integral-tail-error}
\left|
\int_{c-\rho/2}^{c+\rho/2}R_2(t)\,dt
\right|
\le
\int_{c-\rho/2}^{c+\rho/2}|R_2(t)|\,dt
\le
2\rho M4^{-d}.
\end{equation}

By Lemma~\ref{lem:gauss_legendre_facts}, we have $-1<z_j<1$.
Then $|c+\rho z_j/2-c|\le \rho/2$.
Consequently,
\begin{equation*}
|R_2(c+\rho z_j/2)|
\le
2M4^{-d}.
\end{equation*}
By Lemma~\ref{lem:gauss_legendre_facts} again, we can obtain that
\begin{align}\label{eq:quadrature-tail-error}
\frac{\rho}{2}
\left|
\sum_{j=1}^{d}\omega_j R_2(c+\rho z_j/2)
\right|
&\le
\frac{\rho}{2}
\sum_{j=1}^{d}
\omega_j|R_2(c+\rho z_j/2)|
\le
\frac{\rho}{2}
\sum_{j=1}^{d}
\omega_j(2M4^{-d})\nonumber\\
&=
\rho M4^{-d}
\sum_{j=1}^{d}\omega_j
=
2\rho M4^{-d}. 
\end{align}

Substituting \eqref{eq:integral-tail-error} and
\eqref{eq:quadrature-tail-error} into \eqref{eq:error-split}, we obtain
\begin{align*}
\left|
\int_{c-\rho/2}^{c+\rho/2}f(t)\,dt
-
\frac{\rho}{2}\sum_{j=1}^{d}\omega_j f(c+\rho z_j/2)
\right|
\le
2\rho M4^{-d}
+
2\rho M4^{-d}
=
4\rho M4^{-d}.
\end{align*}
This proves \eqref{eq:F9-detailed}.
\end{proof}

We have estimated $\int_5^\infty K(t)t^{2n}\,dt$ in~\eqref{eq:moment_K}.
We now approximate $$\int_0^5K_{16}(t)t^{2n}\,dt.$$
We first partition $[0,5]$ into forty subintervals of length $1/8$: $[s/8, (s+1)/8]$, for $s=0,1,\ldots, 39$.
Using Lemma~\ref{lem:finite_gaussian_rule},
for each subinterval $[s/8, (s+1)/8]$, we construct a disk with center $c=(2s+1)/16$ and radius $\rho=1/8$,
which covers the subinterval $[s/8, (s+1)/8]$.
By~\eqref{eq:F9-detailed},
we estimate the integral of $K_{16}(t)t^{2n}$ in each subinterval.
Summing the estimates over all subintervals yields the estimate of $\int_0^5K_{16}(t)t^{2n}\,dt$.
Finally, we derive lower and upper bounds for $\lambda_n$.

\begin{lemma}\label{lem:finite_moment_enclosures}
Let $z_j$ be the zeros of the Legendre polynomial $P_{256}$
and
$$\omega_j=
\frac{2}
{(1-z_j^2)P_{256}'(z_j)^2}.$$
For $0\leq n\leq 12016$, define
\begin{equation}\label{eq:Qn}
Q_n=\frac{1}{16}\sum_{s=0}^{39}\sum_{j=1}^{256}
\omega_j\cdot \left(\frac{2s+1+z_j}{16}\right)^{2n} \cdot K_{16}\left(\frac{2s+1+z_j}{16}\right)
\end{equation}
and
\begin{equation*}
G_n(x)=n\log(x^2+1/64)+9x/2-\pi\cos(1/4)e^{2x}\quad
\text{for}\ x\geq -1/8.
\end{equation*}
\begin{itemize}
\item[\rm(i)] There exists a unique $x_n^\ast \in [1/8,\infty)$ such that
\begin{equation*}
G_n(x)\le G_n(x_n^\ast),\quad \text{for all}\ x\ge 1/8,
\end{equation*}
and equality holds if and only if $x=x_n^\ast$.
\item[\rm(ii)] 
Set
\begin{equation}\label{eq:bound_G}
M_n=64\pi^2\exp\left(\max\left(
9/16-n\log32,\, G_n(x_n^\ast)\right)\right).
\end{equation}
Then we have 
\begin{equation}\label{eq:bound_la_n}
\frac{Q_n-20M_n4^{-256}}{(2n)!}\leq\lambda_n
\leq
\frac{1}{(2n)!}\cdot
\left(
\frac{Q_n+20M_n4^{-256}}{1-10^{-300}}
+
\frac{64\pi^2\,5^{2n}\exp(45/2-\pi e^{10})}{\pi e^{10}}
\right).
\end{equation}
\end{itemize}
\end{lemma}

\begin{proof}
Partition $[0,5]$ into forty subintervals of length $1/8$: $[s/8, (s+1)/8]$, for $s=0,1,\ldots, 39$, 
with centers $c_s=(2s+1)/16$. 
For each $s$, consider the closed disk
\begin{equation*}
D_s:=\{z:|z-c_s|\le 1/8\}.
\end{equation*}
Since $c_s$ is real, every $z=x+iy\in D_s$ satisfies
$|y|\le 1/8$.
Moreover, since $c_s\ge1/16$,
we have 
$x\ge -1/16>-1/8$.
Hence every $D_s$ is contained in the strip 
$\{t=x+iy: x\geq-1/8,|y|\leq 1/8\}$.

We now bound the function $K_{16}(t) t^{2n}$ in the strip $\{t=x+iy: x\geq-1/8,|y|\leq 1/8\}$.
By~\eqref{eq:K_sup},
\begin{align}\label{eq:proof54-basic-bound}
|K_{16}(x+iy)|\cdot|x+iy|^{2n}&\leq64\pi^2
\exp\bigl(9x/2-\pi\cos(1/4) e^{2x}\bigr)\cdot|x+iy|^{2n}\nonumber\\
&\le
64\pi^2
(x^2+1/64)^n
\exp\bigl(9x/2-\pi\cos(1/4) e^{2x}\bigr)
=64\pi^2 e^{G_n(x)}.
\end{align}

\medskip
\noindent
\textbf{\emph{Proof of (i).}}
We now show that 
\begin{equation*}
G_n(x)=n\log(x^2+1/64)+9x/2-\pi\cos(1/4)e^{2x}
\end{equation*}
has a unique maximum for $x\ge 1/8$.
Differentiating $G_n(x)$ with respect to $x$ yields
\begin{equation*}
G_n'(x)
=
\frac{2nx}{x^2+1/64}
+\frac92
-2\pi\cos(1/4) e^{2x},
\end{equation*}
and
\begin{equation*}
G_n''(x)
=
\frac{2n(1/64-x^2)}
{(x^2+1/64)^2}
-4\pi\cos(1/4) e^{2x}.
\end{equation*}
It is easy to check that $G_n''(x)<0$ for any $x\ge 1/8$.
Therefore $G_n'$ is strictly decreasing on $[1/8,\infty)$, and
\begin{equation*}
G_n'(x)\longrightarrow-\infty
\quad (x\to\infty).
\end{equation*}

If $G_n'(1/8)\le0$,
set $x_n^\ast=1/8$.
Since $G_n'$ is strictly decreasing, 
$G_n'(x)<0$ for any $x\ge 1/8$.
Thus $x_n^\ast=1/8$ is the unique maximizer of
$G_n(x)$ on $[1/8,\infty)$, and
\begin{equation*}
G_n(x)\le G_n(1/8)\le n\log(32)+9/16.
\end{equation*}
If $G_n'(1/8)>0$,
since $G_n'$ is strictly decreasing and tends to $-\infty$, there
exists a unique point $x_n^\ast>1/8$ such that
$G_n'(x_n^\ast)=0$.
In this case, we have $G_n(x)\le G_n(x_n^\ast)$.
Since $ G'_n$ is strictly decreasing,
we have that $x_n^\ast$ is unique.
This proves (i).

\medskip
\noindent
\textbf{\emph{Proof of (ii).}}
We have proved that for $x\ge 1/8$, there is a unique $x_n^\ast$, such that $G_n(x)\le G_n(x_n^\ast)$.
For $-1/8\le x \le 1/8$,
we have
\begin{equation*}
G_n(x)=n\log(x^2+1/64)+9x/2-\pi\cos(1/4)e^{2x}
\le
n\log(1/32)+\frac{9}{16}=\frac{9}{16}-n\log 32.
\end{equation*}
It follows from \eqref{eq:proof54-basic-bound} that
for every one of the forty disks,
$D_s$, we have
\begin{align*}
|K_{16}(x+iy)(x+iy)^{2n}|
\le
64\pi^2\exp\left(\max\left(
9/16-n\log32,\, G_n(x_n^\ast)\right)\right)
=M_n.
\end{align*}
Applying Lemma~\ref{lem:finite_gaussian_rule} with $f(t)=K_{16}(t)t^{2n}$, $\rho=1/8$, and $c=c_s=(2s+1)/16$,
we obtain
\begin{equation*}
\left|
\int_{2s/16}^{(2s+2)/16}
K_{16}(t)t^{2n}\,dt
-
\frac{1}{16}\sum_{j=1}^{256}
\omega_j \left(\frac{2s+1+z_j}{16}\right)^{2n}
K_{16}\left(\frac{2s+1+z_j}{16}\right)
\right|
\le
\frac{M_n4^{-256}}{2}.
\end{equation*}
Summing this estimate over $s=0,\ldots,39$, and using the triangle
inequality, yields
\begin{equation*}
\left|
\int_0^5K_{16}(t)t^{2n}\,dt-Q_n
\right|
\le
40\cdot4\rho M_n4^{-256}=20M_n4^{-256}.
\end{equation*}
Hence
\begin{equation}\label{eq:proof54-truncated-moment}
Q_n-20M_n4^{-256}
\le
\int_0^5K_{16}(t)t^{2n}\,dt
\le
Q_n+20M_n4^{-256}.
\end{equation}

By~\eqref{eq:K-K16} of
Lemma~\ref{lem:finite_kernel_bounds}, for every real $t\ge0$, we have
\begin{equation*}
K_{16}(t)\le K(t)
\le\frac{K_{16}(t)}{1-10^{-300}}.
\end{equation*}
Since $t^{2n}\ge0$,
integration over $[0,5]$ gives
\begin{equation}\label{eq:proof54-kernel-comparison}
\int_0^5K_{16}(t)t^{2n}\,dt
\le
\int_0^5K(t)t^{2n}\,dt
\le
\frac{1}{1-10^{-300}}
\int_0^5K_{16}(t)t^{2n}\,dt.
\end{equation}

On the one hand, the positivity~\eqref{positive} of $K(t)$ gives
\begin{equation*}
\int_0^\infty K(t)t^{2n}\,dt
\ge
\int_0^5K(t)t^{2n}\,dt.
\end{equation*}
Then by \eqref{eq:proof54-truncated-moment}, \eqref{eq:proof54-kernel-comparison} and Lemma~\ref{lem:theta_moments_and_interpolation}, we have
\begin{equation}\label{la>}
\lambda_n
=
\frac1{(2n)!}
\int_0^\infty K(t)t^{2n}\,dt
\ge
\frac{Q_n-20M_n4^{-256}}{(2n)!}.
\end{equation}

On the other hand, since
\begin{equation*}
\int_0^\infty K(t)t^{2n}\,dt
=
\int_0^5K(t)t^{2n}\,dt
+
\int_5^\infty K(t)t^{2n}\,dt.
\end{equation*}
By~\eqref{eq:moment_K},
\eqref{eq:proof54-truncated-moment}  and~\eqref{eq:proof54-kernel-comparison}, we have
\begin{align*}
\int_0^\infty K(t)t^{2n}\,dt
&\le
\frac{Q_n+20M_n4^{-256}}{1-10^{-300}}+\frac{64\pi^2\,5^{2n}\exp(45/2-\pi e^{10})}{\pi e^{10}}.
\end{align*}
Dividing by $(2n)!$, we obtain
\begin{equation}\label{la<}
\lambda_n
\le
\frac{1}{(2n)!}\cdot
\left(
\frac{Q_n+20M_n4^{-256}}{1-10^{-300}}
+
\frac{64\pi^2\,5^{2n}\exp(45/2-\pi e^{10})}{\pi e^{10}}
\right).
\end{equation}

Combining~\eqref{la>} and~\eqref{la<}, we have
\begin{equation*}
\frac{Q_n-20M_n4^{-256}}{(2n)!}
\le
\lambda_n
\le
\frac{1}{(2n)!}\cdot
\left(
\frac{Q_n+20M_n4^{-256}}{1-10^{-300}}
+
\frac{64\pi^2\,5^{2n}\exp(45/2-\pi e^{10})}{\pi e^{10}}
\right).
\end{equation*}
This proves~\eqref{eq:bound_la_n} and completes the proof.
\end{proof}

In the following, we establish the positivity of $A_{m,n}$ and strong-ratio bounds for $r_{14,n}$ for all $1\le n \le 12000$. 
We use interval arithmetic together with the analytic estimates established above.
This approach has also been used by Brent, Platt, and Trudgian in
their work on the prime number theorem and sums over zeros of the Riemann zeta function~\cite{BPT21, BPT22}.

First, we recall the basic rules of interval arithmetic;
see \cite[Sections~2--4]{MKC09}.
For real intervals $X=[a,b]$ and $Y=[c,d]$,
the basic arithmetic operations are defined by
\begin{equation*}
X\circ Y:=\{x\circ y:x\in X,\ y\in Y\},
\quad \circ\in\{+,-,\cdot,/\},
\end{equation*}
where $0\notin Y$ is assumed for division.
The resulting set is an interval whose endpoints are given by
\begin{equation*}
X\circ Y=
\left[
\min\{a\circ c,a\circ d,b\circ c,b\circ d\},
\max\{a\circ c,a\circ d,b\circ c,b\circ d\}
\right].
\end{equation*}
In particular,
\begin{equation*}
X+Y=[a+c,b+d],
\quad
X-Y=[a-d,b-c].
\end{equation*}

The fundamental \emph{inclusion property} is
\begin{equation*}
x\in X,\quad y\in Y
\quad\Longrightarrow\quad
x\circ y\in X\circ Y.
\end{equation*}
The \emph{outward-rounded interval arithmetic} is an interval arithmetic in which, at every operation, lower endpoints are rounded
toward negative infinity and upper endpoints toward positive infinity.

\begin{proposition}\label{prop:finite_escape_certificate}
Let $(\lambda_n)_{n\ge0}$ be the sequence defined in~\eqref{eq:F(x)}, and let
$
A_{m,n}=(\mathcal L^m\lambda)_n.
$ 
Then
\begin{equation}\label{eq:Amn>0}
A_{m,n}>0\quad \text{for}\ 0\leq m\leq14,\,0\leq n\leq12016-m,
\end{equation}
and
\begin{equation}\label{eq:bound_r_14}
r_{14,n}=\frac{A_{14,n}^2}{A_{14,n-1}A_{14,n+1}}>10>
\frac{3+\sqrt5}{2}\quad \text{for}\ 1\leq n\leq12000.
\end{equation}
\end{proposition}

\begin{proof}
The proof is based on outward-rounded interval arithmetic,
with analytic estimates required for the computation already established in Lemmas~\ref{lem:gauss_legendre_facts}--\ref{lem:finite_moment_enclosures}.
We divide the proof into four steps.
In the first two steps, we estimate $Q_n$ in~\eqref{eq:Qn} and $M_n$ in~\eqref{eq:bound_G}.
In Step 3, we prove $A_{m,n}>0$ and in the last step, we verify $r_{14,n}\ge (3+\sqrt 5)/2$.
Further computational details and reproducibility information are provided in~\ref{app:source}.

\medskip
\noindent
\textbf{\emph{Step 1: Estimate of $Q_n$.}}
Recall that
\begin{equation*}
Q_n=\frac{1}{16}\sum_{s=0}^{39}\sum_{j=1}^{256}
\omega_j\cdot \left(\frac{2s+1+z_j}{16}\right)^{2n} \cdot K_{16}\left(\frac{2s+1+z_j}{16}\right),
\end{equation*}
where $z_j$ are the zeros of the Legendre polynomial $P_{256}$
and
\begin{equation*}
\omega_j=
\frac{2}
{(1-z_j^2)P_{256}'(z_j)^2}.
\end{equation*}

We first compute approximate zeros $c_j$ of $P_{256}$ and verify that each interval $[c_j-10^{-100},c_j+10^{-100}]$ contains the exact zero $z_j$.
Substituting these intervals and certified bounds on $P_{256}'(z_j)$ into the weight formula yields a positive interval containing each $\omega_j$.
Finally, evaluating the double sum with outward-rounded interval arithmetic produces intervals for $Q_n$ for all $0 \le n \le 12016$.

\medskip
\noindent
\textbf{\emph{Step 2: Estimate of $M_n$.}}
Recall that
\begin{equation*}
M_n=64\pi^2\exp\left(\max\left(
9/16-n\log32,\, G_n(x_n^\ast)\right)\right),
\end{equation*}
where $G_n(x)=n\log(x^2+1/64)+9x/2-\pi\cos(1/4) e^{2x}$.

By Lemma~\ref{lem:finite_moment_enclosures}, $G_n''(x) < 0$ for all $x \ge 1/8$, so $G_n'(x)$ is strictly decreasing. 
If $G_n'(1/8) \le 0$, then the maximum point $x_n^*$ of $G_n(x)$ is $1/8$.
If $G_n'(1/8) > 0$, then there exists an interval $[\ell_n, u_n]$, such that $x_n^* \in [\ell_n, u_n]$.
In either case, we can obtain an upper bound for $M_n$ for each $0 \le n \le 12016$.

\medskip
\noindent
\textbf{\emph{Step 3: Verification of $A_{m,n}>0$.}}
Recall that by~\eqref{eq:bound_la_n}, we have
\begin{equation*}
\frac{Q_n-20M_n4^{-256}}{(2n)!}
\le \lambda_n \le
\frac{1}{(2n)!}\cdot
\left(
\frac{Q_n+20M_n4^{-256}}{1-10^{-300}}
+
\frac{64\pi^2\,5^{2n}\exp(45/2-\pi e^{10})}{\pi e^{10}}
\right).
\end{equation*}
Using the resulting intervals for $Q_n$ together with the outward upper bounds for $M_n$, this inequality yields a closed interval $I_{0,n}$ containing $\lambda_n = A_{0,n}$.

Starting from $I_{0,n}$, let
\begin{equation}\label{eq:rece:I}
I_{m+1,0}=I_{m,0}^2, \quad \text{and} \quad I_{m+1,n}=I_{m,n}^2-I_{m,n-1}I_{m,n+1}, \quad 1\le n\le12015-m.
\end{equation}
We prove inductively that $A_{m,n}\in I_{m,n}$ for all generated terms.
This holds for $m = 0$ by the definition of $I_{0,n}$ above.
If it holds at step $m$, then by the inclusion property of interval arithmetic,
\begin{equation*}
A_{m+1,0}=A_{m,0}^2\in I_{m,0}^2=I_{m+1,0},
\end{equation*}
and for $1\le n\le 12015-m$,
\begin{equation*}
A_{m+1,n}
=
A_{m,n}^2-A_{m,n-1}A_{m,n+1}
\in
I_{m,n}^2-I_{m,n-1}I_{m,n+1}
=
I_{m+1,n}.
\end{equation*}
Thus the inclusion holds at every step.

Finally, we obtain that every such interval $I_{m,n}$ has a strictly positive lower endpoint, which gives
\begin{equation*}
A_{m,n}>0 \quad \text{for all } 0\le m\le14 \text{ and } 0\le n\le12016-m.
\end{equation*}

\medskip
\noindent
\textbf{\emph{Step 4: Verification of $r_{14,n}>10$.}}
At step $m=14$, using outward-rounded interval arithmetic, we have that the lower endpoints of the intervals
\begin{equation}\label{eq:ratio_I14}
\frac{I_{14,n}^2}{I_{14,n-1}I_{14,n+1}}
\end{equation}
strictly exceed the integer $10$ for all $1\le n\le 12000$.
Since $A_{14,n} \in I_{14,n}$, this establishes
\begin{equation*}
r_{14,n} = \frac{A_{14,n}^2}{A_{14,n-1}A_{14,n+1}} > 10 > \frac{3+\sqrt{5}}{2} \quad \text{for all } 1\le n\le 12000,
\end{equation*}
which completes the proof of~\eqref{eq:bound_r_14}.
\end{proof}

%=================================================
\section{Proof of the main theorem}\label{sec:closure}
%=================================================
In this section, we combine the results of Sections~\ref{sec:analytic} and~\ref{sec:finite} to complete the proof of Theorem~\ref{thm:main}.
We begin with a general lemma, which shows that under the assumed initial conditions, the operator $\mathcal{L}$ preserves strict positivity and the strong-ratio property.

\begin{lemma}\label{lem:variable_escape_global_closure}
Let $\alpha=(a_n)_{n\ge0}$ be a strictly positive sequence and let $A_{m,n}=(\mathcal L^m\alpha)_n$.
For every integer $n\ge 1$,
suppose that there exists a finite integer
$\tau(n)\ge0$ such that
\begin{itemize}
\item[\rm(i)] $A_{m,n}>0$ for $0\le m\le \tau(n);$
\item[\rm(ii)] the ratio $
r_{\tau(n),n}={A_{\tau(n),n}^2}/
({A_{\tau(n),n-1}A_{\tau(n),n+1}})
$, as defined in~\eqref{eq:rmk&gmk}, 
is well-defined and satisfies
\begin{equation}\label{eq:intr>R}
r_{\tau(n),n}\ge \frac{3+\sqrt5}{2}.
\end{equation}
\end{itemize}
Then, for every integer $n\ge1$ and every $m\ge\tau(n)$,
the ratio $r_{m,n}$ is well-defined and
\begin{equation*}
A_{m,n}>0
\quad\text{and}\quad
r_{m,n}\ge\frac{3+\sqrt5}{2}.
\end{equation*}
\end{lemma}

\begin{proof}
We prove simultaneously by induction on $m$ that
$
A_{m,n}>0$ for every $n\ge0$,
and that
$r_{m,n}\ge ({3+\sqrt5})/{2}$
for every $n\ge1$ satisfying $\tau(n)\le m$.
In particular, all ratios appearing in the second assertion
are well-defined.

For $m=0$,  it is trivial that $A_{0,n}=a_n>0$. For every $n\ge1$ with $\tau(n)=0$,
the second assertion follows from condition~{\rm(ii)}.

Suppose that both assertions hold at some $m\ge0$.
We first establish positivity at level $m+1$.
For $n=0$, the boundary relation gives
$A_{m+1,0}=A_{m,0}^{2}>0$.
Now let $n\ge1$.
If $m<\tau(n)$, then $m+1\le\tau(n)$, and
condition~{\rm(i)} gives $A_{m+1,n}>0$.
If $m\ge\tau(n)$, then the induction hypothesis gives
$r_{m,n}\ge (3+\sqrt 5)/2>1$, and hence
\begin{equation*}
A_{m+1,n}
=
A_{m,n}^{2}-A_{m,n-1}A_{m,n+1}
=
A_{m,n}^{2}
\left(1-\frac{1}{r_{m,n}}\right)>0.
\end{equation*}
Thus $A_{m+1,n}>0$ for every $n\ge0$.

Since all entries at levels $m$ and $m+1$ are positive,
we have 
$0<A_{m+1,k}\le A_{m,k}^{2}$
for every $k\ge0$.
We now consider $r_{m+1,n}$.
Let $n\ge1$ satisfy $\tau(n)\le m+1$.
If $\tau(n)=m+1$, then condition~{\rm(ii)} gives
$r_{m+1,n}\ge (3+\sqrt 5)/2$.
Otherwise, $\tau(n)\le m$, and therefore
\begin{align*}
r_{m+1,n}&=\frac{A_{m+1,n}^2}{A_{m+1,n-1}A_{m+1,n+1}}\ge
\frac{A_{m,n}^4 (1-1/r_{m,n})^2}{A_{m,n-1}^2A_{m,n+1}^2}\\
&=
r_{m,n}^2\left(1-\frac{1}{r_{m,n}}\right)^2
=
(r_{m,n}-1)^2
\ge \left(\frac{3+\sqrt5}{2}-1\right)^2=\frac{3+\sqrt5}{2}.
\end{align*}
This completes the proof.
\end{proof}

\begin{proof}[Proof of Theorem~\ref{thm:main}]
By Lemma~\ref{lem:theta_moments_and_interpolation}, we have 
$\lambda_n>0$ for any $n\ge0$.
For $n\ge1$, set
\begin{equation*}
\tau(n)=
\begin{cases}
14, & 1\le n<12000,\\
\lceil\log_2 n\rceil, & n\ge12000.
\end{cases}
\end{equation*}
If $1\le n<12000$, Proposition~\ref{prop:finite_escape_certificate} gives
\begin{equation*}
A_{m,n}>0\quad \text{for}\ 0\le m\le\tau(n),
\quad
\text{and}\
r_{\tau(n),n}>10>\frac{3+\sqrt5}{2}.
\end{equation*}
If $n\ge12000$, Proposition~\ref{prop:large-index} yields
\begin{equation*}
A_{m,n}>0\quad \text{for}\ 0\le m\le\tau(n),
\quad
\text{and} \
r_{\tau(n),n}> e>\frac{3+\sqrt5}{2}.
\end{equation*}
Hence the hypotheses of Lemma~\ref{lem:variable_escape_global_closure}
are satisfied. Therefore
\begin{equation*}
A_{m,n}>0
\quad \text{for}\ m,n\ge0.
\end{equation*}
This completes the proof.
\end{proof}

%=================================================
\section*{Declaration of AI usage}
%=================================================
During the preparation of this manuscript, the authors used generative AI tools to assist in exploring possible approaches, checking technical details, and improving the exposition. The authors verified the mathematical arguments, computations, proofs, and results and take full responsibility for the content of this manuscript.

%=================================================

%=================================================
\appendix
\section{Computer-assisted verification and reproducibility}\label{app:source}
%=================================================
This appendix records the computational details underlying Proposition~\ref{prop:finite_escape_certificate}, including the software environment, reproduction command, and implementation used to certify the required inequalities. 
It is included to make the computer-assisted part of the proof fully reproducible.
%=================================================
\addtocontents{toc}{\protect\setcounter{tocdepth}{0}}
\subsection{Computational environment and reproduction}
%=================================================
The finite verification was carried out using Python 3.12.14,
\path|mpmath| 1.3.0, and \path|gmpy2| 2.3.1.  From the supplementary
directory, the computation can be reproduced by running
\begin{lstlisting}
python scripts/xi_finite_certificate.py --nmax 12016 --q 256 --dps 130 --workers 8 --steps 14 --certify-to 12000 --tag current_review
\end{lstlisting}
Python assertions must be enabled; in particular, the options
\texttt{-O}, \texttt{-OO}, and \texttt{PYTHONOPTIMIZE} should not be used.

The computation certifies all $256$ Gauss--Legendre node brackets,
the enclosing intervals for
$\lambda_0,\ldots,\lambda_{12016}$, the positivity of the iterates
through step $14$, and the $12{,}000$ ratio inequalities required in
Proposition~\ref{prop:finite_escape_certificate}.  All interval computations
are performed with outward rounding as described below.
%=================================================
\addtocontents{toc}{\protect\setcounter{tocdepth}{0}}
\subsection{Structure of the computer-assisted proof of Proposition~\ref{prop:finite_escape_certificate}}
%=================================================
In this part, we describe the interval arithmetic verification of the positivity of $A_{m,n}$ and the ratio bounds $r_{14,n} > 10$ in the proof of Proposition~\ref{prop:finite_escape_certificate}.

\medskip
\noindent
\textbf{\emph{Step 1: Estimate of  $Q_n$.}}
The program first computes approximate zeros of $P_{256}$, which are
used only to choose candidate root brackets.  
For each approximate zero $c_j$ ($1\le j\le 256$), the program considers the interval $[c_j-10^{-100},\,c_j+10^{-100}]$.
Using the three-term recurrence~\cite[p. 165]{SLL09}
\begin{equation*}
P_0(z)=1,\quad
P_1(z)=z,\quad
P_{d+1}(z)
=
\frac{(2d+1)zP_d(z)-dP_{d-1}(z)}{d+1},
\end{equation*}
combined with outward-rounded interval arithmetic, the program verifies that $P_{256}$ takes opposite signs at the endpoints of $256$ pairwise disjoint intervals.
Hence every such interval contains a zero of $P_{256}$.
Since $P_{256}$ has degree $256$, each interval contains exactly one
zero. 
Denote this zero by $z_j$.  
Thus $|z_j-c_j|\le 10^{-100}$.

For the corresponding weights, we use the standard Legendre polynomial identity~\cite[p. 165]{SLL09}
\begin{equation*}
P_d'(z)
=
d\frac{zP_d(z)-P_{d-1}(z)}{z^2-1},
\end{equation*}
to evaluate $P_{256}'(c_j)$.
By the mean value theorem and Lemma~\ref{lem:gauss_legendre_facts}, we have
\begin{equation*}
|P_{256}'(z_j)-P_{256}'(c_j)|
\le
\sup_{[-1,1]}|P_{256}''|\cdot|z_j-c_j|
\le
256^4\cdot10^{-100}.
\end{equation*}

The implementation uses the slightly larger radius $2\cdot256^4\cdot10^{-100}$,
where the factor $2$ also absorbs the small outward-conversion width of the decimal midpoint.
Hence the interval
\begin{equation*}
P_{256}'(c_j)
+
[-2\cdot256^4\cdot10^{-100},\,2\cdot256^4\cdot10^{-100}]
\end{equation*}
contains the exact value $P_{256}'(z_j)$.

The program checks that this derivative interval does not contain zero.
Substituting the certified intervals for $z_j$ and $P_{256}'(z_j)$ into
\begin{equation*}
\omega_j=
\frac{2}{(1-z_j^2)P_{256}'(z_j)^2}
\end{equation*}
yields a certified positive interval containing the exact weight $\omega_j$.

The program then evaluates the sums
\begin{equation*}
Q_n=\frac{1}{16}\sum_{s=0}^{39}\sum_{j=1}^{256}
\omega_j\cdot \left(\frac{2s+1+z_j}{16}\right)^{2n} \cdot K_{16}\left(\frac{2s+1+z_j}{16}\right) \quad
\text{for}\ 0\leq n\leq 12016,
\end{equation*}
by using $130$ decimal digits and directed rounding.
At each node, the factors are positive, so the powers $t^{2n}$ are accumulated by $\operatorname{term}_{n+1}=\operatorname{term}_n\,t^2$
with the lower endpoint rounded downward and the upper endpoint rounded upward.  
The segment sums and the final sums are combined in the same way.

\medskip
\noindent
\textbf{\emph{Step 2: Estimate of $M_n$.}}
By~\eqref{eq:bound_G}, we have
\begin{equation*}
M_n=64\pi^2\exp\left(\max\left(
9/16-n\log32,\, G_n(x_n^\ast)\right)\right).
\end{equation*}
For the function
\begin{equation*}
G_n(x)
=
n\log(x^2+1/64)+\frac{9x}{2}-\pi\cos(1/4) e^{2x},
\end{equation*}
Lemma~\ref{lem:finite_moment_enclosures} shows that
$G_n''(x)<0$ for any $x\ge 1/8$.
Hence $G_n'(x)$ is strictly decreasing.

For each $0\le n\le12016$, the program first evaluates $G_n'(1/8)$.
If $G_n'(1/8)\le0$, then $G_n$ is decreasing on $[1/8,\infty)$, and its maximum occurs at $x=1/8$.
Otherwise, an approximate zero of $G_n'$ is used to select a candidate interval $[\ell_n,u_n]\subset [1/8,\infty)$.
The program then certifies, by outward interval arithmetic, that
\begin{equation*}
G_n'(\ell_n)>0 \quad \text{and} \quad G_n'(u_n)<0.
\end{equation*}
Since $G_n'$ is strictly decreasing, the unique maximum point $x_n^\ast$ of $G_n$ must lie in $[\ell_n,u_n]$, so that $G_n(x)\le G_n(x_n^\ast)$.
In either case, the program evaluates an outward upper bound for $M_n$ via~\eqref{eq:bound_G}.

\medskip
\noindent
\textbf{\emph{Step 3: Verification of $A_{m,n}>0$.}}
Using the resulting interval for $Q_n$, together with the outward upper bound of $M_n$, \eqref{eq:bound_la_n} gives a closed interval $I_{0,n}$ containing $\lambda_n$:
\begin{equation*}
\frac{Q_n-20M_n4^{-256}}{(2n)!}\leq\lambda_n
\leq
\frac{1}{(2n)!}\cdot
\left(
\frac{Q_n+20M_n4^{-256}}{1-10^{-300}}
+
\frac{64\pi^2\,5^{2n}\exp(45/2-\pi e^{10})}{\pi e^{10}}
\right).
\end{equation*}
The reciprocal factorial is evaluated recursively from
\begin{equation*}
d_0=1,\quad
d_{n+1}
=
\frac{d_n}{(2n+1)(2n+2)}.
\end{equation*}
The program checks that every lower endpoint of $I_{0,n}$ is positive.
Hence $\lambda_n>0$ for any $0\le n\le12016$.
Starting from the initial intervals $I_{0,n}$, the program computes for each $0\le m\le14$:
\begin{equation*}
I_{m+1,0}=I_{m,0}^2, \quad \text{and} \quad I_{m+1,n}=I_{m,n}^2-I_{m,n-1}I_{m,n+1}, \quad 1\le n\le12015-m.
\end{equation*}
All operations are carried out with outward rounding.

Since evaluating $A_{m+1,n}$ requires the term $A_{m,n+1}$, the available range of $n$ decreases by one at each iteration.
Thus, at step $m$, we only calculate $A_{m,n}$ up to $n = 12016-m$, without assuming that any values outside this range are zero.
The program checks that every interval $I_{m,n}$ has a positive lower
endpoint for
\begin{equation*}
0\le m\le14,
\quad
0\le n\le12016-m.
\end{equation*}
Since $A_{m,n}\in I_{m,n}$, this proves
\begin{equation*}
A_{m,n}>0
\quad \text{for}\ 0\le m\le14,\, 0\le n\le12016-m,
\end{equation*}
which is~\eqref{eq:Amn>0}.

\medskip
\noindent
\textbf{\emph{Step 4: Verification of $r_{14,n}>10$.}}
At step $m = 14$, the available intervals $I_{14,n}$ reach index $n = 12016 - 14 = 12002$, ensuring that the three positive intervals $I_{14,n-1}, I_{14,n}, I_{14,n+1}$ are available for all $1 \le n \le 12000$.
The program evaluates the ratio intervals
\begin{equation*}
\frac{I_{14,n}^2}{I_{14,n-1}I_{14,n+1}}
\end{equation*}
using outward-rounded interval arithmetic, and verifies that subtracting the exact integer $10$ yields an interval with a strictly positive lower bound.
In particular,
the smallest lower endpoint occurs at $n = 12000$, where the ratio interval is approximately given by
\begin{equation*}
[10.9042985947999313085054,\, 10.9042985947999313085055].
\end{equation*}
We emphasize that the subtraction and positivity check are performed on the exact binary interval representations rather than on these displayed decimal values.
Since $A_{14,n} \in I_{14,n}$ for all generated terms, we have
\begin{equation*}
r_{14,n} = \frac{A_{14,n}^2}{A_{14,n-1}A_{14,n+1}} > 10 > \frac{3+\sqrt{5}}{2} \quad \text{for all } 1 \le n \le 12000,
\end{equation*}
which completes the proof of~\eqref{eq:bound_r_14}.

The following table records the correspondence between the finite
verification in Proposition~\ref{prop:finite_escape_certificate} and the
functions used in the implementation.

\begin{center}
\small
\renewcommand{\arraystretch}{1.15}
\begin{tabular}{@{}
>{\raggedright\arraybackslash}p{0.23\textwidth}
>{\raggedright\arraybackslash}p{0.25\textwidth}
>{\raggedright\arraybackslash}p{0.42\textwidth}@{}}
\toprule
\textbf{Step in the proof} & \textbf{Functions} & \textbf{Certified result} \\
\midrule

Legendre nodes and weights
&
\texttt{legendre\_pair}, \texttt{certified\_nodes}
&
Disjoint root intervals, opposite endpoint signs, and positive weights.
\\ \addlinespace[6pt]

Bounds for $K$ and $K_{16}$
&
\texttt{kernel\_truncated}, \texttt{audit\_constants}
&
Certified bounds for the error $|K-K_{16}|$.
\\ \addlinespace[6pt]

Moment sums~\eqref{eq:Qn}
&
\texttt{segment\_moments}
&
Positive interval bounds for $Q_n$.
\\ \addlinespace[6pt]

Error bounds $M_n$ in~\eqref{eq:bound_G}
&
\texttt{phase\_error\_bounds}
&
Intervals $[l_n,u_n]$ for $x_n^*$ and upper bounds for
$M_n$ in~\eqref{eq:bound_G}.
\\ \addlinespace[6pt]

Taylor coefficients $\lambda_n$ in~\eqref{eq:bound_la_n}
&
\texttt{enclose\_coefficients}
&
Positive intervals containing $\lambda_0,\ldots,\lambda_{12016}$.
\\ \addlinespace[6pt]

Iterates and ratios in~\eqref{eq:rece:I} and~\eqref{eq:ratio_I14}
&
\texttt{check\_iterates}, \texttt{main}
&
Positivity of $A_{m,n}$ through step $14$, and ratio bounds
$r_{14,n}>10$ in~\eqref{eq:bound_r_14}.
\\
\bottomrule
\end{tabular}
\end{center}

%=================================================
\addtocontents{toc}{\protect\setcounter{tocdepth}{0}}
\subsection{Certified data and interval arithmetic}
%=================================================
The supplementary files contain six exact JSON records for the node
brackets and weights, the interval enclosures of the sums $Q_n$, the
bounds in~\eqref{eq:bound_la_n}, the coefficient intervals, the final ratio intervals, and the iteration summaries.
They also record the software versions, the source hash, and the parameters used in the computation.
An endpoint tuple
\path|[sign,mantissa,exponent,bitcount]| represents the exact binary value
\begin{equation*}
(-1)^{\mathrm{sign}}\mathrm{mantissa}\,2^{\mathrm{exponent}}.
\end{equation*}
The mantissas are arbitrary-precision integers.  Thus the stored interval
endpoints are the exact binary enclosures produced by the computation,
rather than decimal approximations reconstructed from printed output.
The displayed decimal values are obtained from these endpoints by outward
rounding using exact integer division.

The interval computations use \path|mpmath| 1.3.0, with
\path|gmpy2| 2.3.1 as the integer backend.  The elementary interval
operations used in the computation include exponential, logarithm, cosine,
pi, and square root.  The positive-arithmetic routines specify
\path|round_floor| and \path|round_ceiling| explicitly.  Approximate
\path|mp| calculations are used only to select the candidate node brackets
and the intervals $[l_n,u_n]$.  
These intervals are subsequently certified by outward interval inequalities, so no approximate computation is used as an unchecked enclosure endpoint.

The separate script
\path|scripts/check_analytic_constants.py| verifies, by outward interval
arithmetic, the $34$ scalar inequalities used in the analytic estimates.
These numerical checks are supplementary to the analytic and inductive
arguments in the proof.
%=================================================
\addtocontents{toc}{\protect\setcounter{tocdepth}{0}}
\subsection{Source code for the finite verification}
%=================================================
The source code used for the finite verification is available in the
accompanying GitHub repository at
\url{https://github.com/YXLiu97/riemann-xi-infinite-log-concavity}.
The SHA-256 hash of
\path|scripts/xi_finite_certificate.py|
used for the finite verification in this paper is
\begin{center}\small\ttfamily
\seqsplit{01816166fa0b317f0b116ca3ac45ffa198ec9348aa3c73d48f4d8a571e0211fe}.
\end{center}
This hash identifies the exact version of the program used to obtain the finite verification described above.
\end{document}